\documentclass[12pt,a4paper]{amsart}
\usepackage{color,graphicx,fullpage,amssymb,tikz}
\usepackage{url}
\usepackage[colorlinks=true]{hyperref}
\usepackage{doi}
\usepackage[skip=10pt]{caption}
\newtheorem{theorem}{Theorem}[section]
\newtheorem{lemma}[theorem]{Lemma}
\newtheorem{proposition}[theorem]{Proposition}
\newtheorem{corollary}[theorem]{Corollary}

\newtheorem{maintheorem}{Theorem}

\theoremstyle{definition}
\newtheorem{definition}[theorem]{Definition}
\newtheorem{remark}[theorem]{Remark}
\newtheorem{example}[theorem]{Example}
\newtheorem{question}[theorem]{Question}
\newcommand{\Q}{\mathbb{Q}}
\newcommand{\Qp}{\mathbb{Q}_p}
\newcommand{\Cp}{\mathbb{C}_p}
\newcommand{\Zp}{\mathbb{Z}_p}
\newcommand{\N}{\mathbb{N}}
\newcommand{\R}{\mathbb{R}}
\newcommand{\Z}{\mathbb{Z}}
\newcommand{\C}{\mathbb{C}}

\newcommand{\dd}{\mathrm{d}}
\newcommand{\ii}{\mathrm{i}}
\newcommand{\e}{\mathrm{e}}
\newcommand{\f}{\mathrm{f}}
\newcommand{\h}{\mathrm{h}}
\newcommand{\sphere}{\mathrm{S}^2_p}
\newcommand{\Circle}{\mathrm{S}^1_p}
\newcommand{\letpprime}{Let $p$ be a prime number}
\newcommand{\ocal}{\mathcal{O}}
\newcommand{\M}{\mathcal{M}}

\newcommand{\axes}{\draw (-1,0) -- (1,0); \draw (0,-1) -- (0,1);}
\newcommand{\bDSq}{\overline{\DSq}}

\renewcommand{\le}{\leqslant}
\renewcommand{\ge}{\geqslant}
\DeclareMathOperator{\im}{Im}

\DeclareMathOperator{\ord}{ord}

\DeclareMathOperator{\DSq}{DSq}
\DeclareMathOperator{\Span}{span}

\numberwithin{equation}{section}

\title{$p$-adic symplectic geometry of integrable systems and Weierstrass-Williamson theory I}
\author[Luis Crespo, \'Alvaro Pelayo]{Luis Crespo\,\,\,\,\,\, \'Alvaro Pelayo}

\address{Luis Crespo,
	Departamento de Matem\'{a}ticas, Estad\'{i}stica y Computaci\'{o}n, Universidad de Cantabria, Av.~de Los Castros 48, 39005 Santander, Spain}
\email{luis.cresporuiz@unican.es}

\address{\'Alvaro Pelayo,
	Facultad de Ciencias Matem\'aticas,
	Universidad Complutense de Madrid, 28040 Madrid, Spain, and Real Academia de Ciencias Exactas, F\'isicas y Naturales, Spain}
\email{alvpel01@ucm.es}

\begin{document}

\maketitle

\begin{center}
	\emph{In memory of Professor Vladimir Voevodsky (1966--2017).}
\end{center}

\begin{abstract}
	The local symplectic theory of integrable systems is fundamental to understand their global theory, as well as the behavior near singularities of fundamental models from classical and quantum mechanics which are known to be integrable, such as the Jaynes-Cummings model and the coupled angular momenta. We establish the foundations of the local symplectic geometry of $p$-adic integrable systems on $4$-dimensional $p$-adic analytic symplectic manifolds, by classifiying all their possible local models. In order to do this we develop a new approach, of independent interest, to the theory of Weierstrass and Williamson concerning the diagonalization of real matrices by real symplectic matrices. We show that this approach can be generalized to $p$-adic matrices, leading to a classification of real $(2n)$-by-$(2n)$ matrices and of $p$-adic $2$-by-$2$ and $4$-by-$4$ matrix normal forms, including, up to dimension $4$, the classification in the degenerate case, for which the literature is limited even in the real case. A combination of these results and the Hardy-Ramanujan formula shows that both the number of $p$-adic matrix normal forms and the number of local models of $p$-adic integrable systems grow almost exponentially with their dimensions, in strong contrast with the real case. These results fit in a program, proposed a decade ago by Voevodsky, Warren and the second author, to develop a $p$-adic theory of integrable systems with the goal of later implementing it using proof assistants. In the present paper we shall prove all of the statements concerning integrable systems; their proofs rely on the $p$-adic analog of the Weierstrass and Williamson theory of matrices. The proofs concerning this new $p$-adic theory of matrices are of an algebraic and arithmetical nature, and we give them in the sequel paper, part II.
	
\end{abstract}

\section{Introduction and main results about integrable systems}\label{sec:intro}

\subsection{$p$-adic geometry, symplectic geometry and integrable systems}

In a recent lecture \cite{Lurie} Lurie commented that \emph{``roughly speaking $p$-adic geometry, or rigid analytic geometry, is a version of the theory of complex manifolds where instead of using complex numbers you use something like $p$-adic numbers''}. In the present paper we attempt to start studying this direction for the closely related class of symplectic manifolds: we will concentrate on the \emph{local theory} of integrable systems with $p$-adic coefficients, on $p$-adic analytic symplectic manifolds.\footnote{The relation between complex and symplectic structures on manifolds already appears implicitly in the pioneering work of Kodaira \cite{Kodaira} and in a well known paper by Thurston \cite{Thurston}, as well as in many other contributions including for instance \cite{BazMun,Delzant,DuiPel-reduced,DuiPel-complex,FGG}.} The local theory has two crucial applications: first, it is a fundamental ingredient for understanding the behavior of integrable systems near their singularities, and second, it is a stepping stone for understanding their global theory. The local theory is necessary to understand important mechanical systems such as the Jaynes-Cummings model from optics  \cite{PelVuN-spin} (a $p$-adic version of which we studied in \cite{CrePel-JC}) and the coupled angular momenta \cite{LeFPel}.

More concretely, for any prime number $p$, the $p$-adic numbers $\Qp$ form an extension field of the rational numbers $\Q$ which plays a prominent role in various parts of geometry, as seen for instance in the aforementioned recent lecture by Lurie and Scholze-Weinstein's lectures \cite{SchWei}. The present paper takes a first step in introducing $p$-adic methods in symplectic geometry of integrable systems. We focus on $p$-adic matrix theory and $p$-adic symplectic geometry and applying them to fully describe the local theory of $p$-adic integrable systems $F=(f_1,f_2):(M,\omega)\to(\Qp)^2$ on $p$-adic analytic symplectic $4$-manifolds $(M,\omega)$, as part of a general approach to this new field proposed ten years ago by Voevodsky, Warren and the second author \cite[Section 7]{PVW}. Our techniques rely on the theory of $p$-adic extension fields.

When we started this paper we did not expect the local theory of $p$-adic integrable systems to be so rich: the possibilities in the real case represent, in comparison, only a tiny proportion of what occurs in the $p$-adic case (see Table \ref{table:num-integrable} which for example tells us that there are exactly $211$ local models in the case of $2$-adic symplectic $4$-manifolds, while there are exactly $4$ such models with real coefficients). Based on what we have found in this paper and \cite{CrePel-JC}, we believe that \emph{$p$-adic symplectic geometry} and its application to \emph{$p$-adic integrable systems} are promising research directions. Another example of this richness can be seen in the $p$-adic version of the real Jaynes-Cummings model from quantum optics, treated in our previous paper \cite{CrePel-JC}.

We expect that the results of this paper can be formalized using the proof assistant Coq, in the setting of homotopy type theory and Voevodsky's Univalent Foundations, which was the original motivation for Voevodsky, Warren and the second author to formalize the construction of the $p$-adic numbers in this setting in \cite{PVW}. We refer to \cite{APW,PelWar,PelWar2} for an introduction to the Voevodsky's univalence axiom and homotopy type theory.

\subsection{The two goals of this paper and applications}

This paper has two closely related goals. Our first goal is to fully describe the local symplectic geometry of $p$-adic integrable systems on symplectic $4$-manifolds, that is, we explicitly classify their local models and give a concrete list of their formulas. In order to do this, we need to extend the seminal theory of Weierstrass \cite{Weierstrass} and Williamson \cite{Williamson} concerning the diagonalization of real symmetric matrices by means of symplectic matrices, to $p$-adic $4$-by-$4$ matrices, which is the second goal of the paper. That is, in order to prove our theorems about $p$-adic integrable systems we use as stepping stones the theorems about $p$-adic matrices. In the present paper we will (only) state the theorems about $p$-adic matrices and then use them to give full proofs of the theorems concerning $p$-adic integrable systems. But we do not include here proofs of the theorems about $p$-adic matrices, which appear instead in the sequel paper \cite{CrePel-matrix}. These proofs, which are of an algebraic and arithmetical nature, are quite elaborate and better organized as a separate paper.

The Weierstrass-Williamson theory of matrices has crucial applications in many areas including the theory of quantum states in quantum physics \cite{Gosson,SSM,WHTH}, hence this paper provides a new tool to further explore $p$-adic analogues of these applications in symplectic geometry and beyond.

As an application of our results and the Hardy-Ramanujan formula \cite{HarRam} (obtained also by Uspensky \cite{Uspensky}) in number theory, we confirm that the number of $p$-adic $(2n)$-by-$(2n)$ matrix normal forms grows asymptotically at least with ${\rm e}^{\pi \sqrt{2n/3}}/{4\sqrt{3}n}$, which in particular implies that the number of local normal forms of $p$-adic integrable systems on $2n$-dimensional symplectic manifolds at a rank $0$ critical point grows in the same way. This is in strong contrast with the real case, where the number of normal forms of integrable systems at a rank $0$ critical point is quadratic in the dimension.

\subsection{Main results about integrable systems: Theorems \ref{thm:integrable}, \ref{thm:num-integrable} and \ref{thm:num-integrable2}}

Recall that a \emph{quadratic residue modulo $p$} is an integer which is congruent to a perfect square modulo $p$; if this does not hold, then the integer is called a \emph{quadratic non-residue\footnote{For any prime $p>2$, the number of quadratic non-residues modulo $p$ is $(p-1)/2$. For example, if $p=17$, the quadratic non-residues modulo $17$ are $3,5,6,7,10,11,12$ and $14$.} modulo $p$.} Quadratic non-residues play a crucial role in our main theorems, for which we will need the following definition.

\begin{definition}[Non-residue sets and coefficient functions]\label{def:sets}
	\letpprime. If $p\equiv 1\mod 4$, let $c_0$ be the smallest quadratic non-residue modulo $p$. We define the \emph{non-residue sets}
	\[X_p=\begin{cases}
		\{1,c_0,p,c_0p,c_0^2p,c_0^3p,c_0p^2\} & \text{if }p\equiv 1\mod 4; \\
		\{1,-1,p,-p,p^2\} & \text{if }p\equiv 3\mod 4; \\
		\{1,-1,2,-2,3,-3,6,-6,12,-18,24\} & \text{if }p=2.
	\end{cases}\]
	\[Y_p=\begin{cases}
		\{c_0,p,c_0p\} & \text{if }p\equiv 1\mod 4; \\
		\{-1,p,-p\} & \text{if }p\equiv 3\mod 4; \\
		\{-1,2,-2,3,-3,6,-6\} & \text{if }p=2.
	\end{cases}\]
	We also define the \emph{coefficient functions} $\mathcal{C}_i^k:Y_p\times(\Qp)^4\to\Qp$ and $\mathcal{D}_i^k:Y_p\times(\Qp)^4\to\Qp$, for $k\in\{1,2\}$, $i\in\{0,1,2\}$, by
	\[\mathcal{C}_0^1(c,t_1,t_2,a,b)=\frac{ac}{2(c-b^2)},\,\,\mathcal{C}_1^1(c,t_1,t_2,a,b)=\frac{b}{b^2-c},\,\,\mathcal{C}_2^1(c,t_1,t_2,a,b)=\frac{1}{2a(c-b^2)},\]
	\[\mathcal{C}_0^2(c,t_1,t_2,a,b)=\frac{abc}{2(b^2-c)},\,\,\mathcal{C}_1^2(c,t_1,t_2,a,b)=\frac{c}{c-b^2},\,\,\mathcal{C}_2^2(c,t_1,t_2,a,b)=\frac{b}{2a(b^2-c)},\]
	\[\mathcal{D}_0^1(c,t_1,t_2,a,b)=-\frac{t_1+bt_2}{2a},\,\mathcal{D}_1^1(c,t_1,t_2,a,b)=-bt_1-ct_2,\,\mathcal{D}_2^1(c,t_1,t_2,a,b)=-\frac{ac(t_1+bt_2)}{2},\]
	\[\mathcal{D}_0^2(c,t_1,t_2,a,b)=-\frac{bt_1+ct_2}{2a},\,\mathcal{D}_1^2(c,t_1,t_2,a,b)=-c(t_1+bt_2),\,\mathcal{D}_2^2(c,t_1,t_2,a,b)=-\frac{ac(bt_1+ct_2)}{2}.\]
\end{definition}

For a concise review of the basic concepts of $p$-adic integrable systems and $p$-adic symplectic geometry see \cite[Section 3]{CrePel-JC}. In the results below we use the following terminology. \letpprime.

\begin{itemize}
	\item A \emph{$p$-adic analytic symplectic manifold} is a pair $(M,\omega)$ where $M$ is a $p$-adic analytic manifold and $\omega$ is a closed non-degenerate $2$-form on $M$. We say that $\omega$ is a \emph{symplectic form.}
	\item Let $(M_1,\omega_1)$ and $(M_2,\omega_2)$ be $p$-adic analytic symplectic manifolds. Let $m\in M$. A \emph{local symplectomorphism} $\phi:U_1\to U_2$ \emph{centered at $m$} is a $p$-adic analytic diffeomorphism between some open sets $U_1\subset M_1$ and $U_2\subset M_2$, such that $m\in U_1$ and $\phi^*\omega_2=\omega_1$.
	\item Let $(M,\omega)$ be a $p$-adic analytic symplectic manifold. By \emph{local symplectic coordinates $(x_1, \xi_1, \ldots , x_n, \xi_n)$ with the origin a point $m\in M$} we mean coordinates given by a local symplectomorphism  $\phi :  (U\subset M, \omega) \to (\phi(U) \subset (\Qp)^{2n}, \omega_0)$ centered at $m$ such that $\phi(m)=(0, \ldots, 0)$.
	\item A $p$-adic analytic map $F:=(f_1,\ldots,f_n):(M,\omega)\to(\Qp)^n$ on a $p$-adic analytic symplectic manifold of dimension $2n$ is a \emph{$p$-adic analytic integrable system} if $\{f_i,f_j\}=0$ for all $1\le i\le j\le n$ and the set where the $n$ differential $1$-forms $\dd f_1,\ldots,\dd f_n$ are linearly independent is dense in $M$.
\end{itemize}

The notions of $p$-adic analytic function and of critical point of a $p$-adic analytic function are reviewed in Appendix \ref{sec:critical}. For the notion of critical point of $p$-adic integrable system and its rank see Definitions \ref{def:critical}, \ref{def:nondeg-integrable}. The notation $\ocal(3)$ means terms of degree at least $3$. The following two results, Theorems \ref{thm:integrable}, \ref{thm:num-integrable}, concern $p$-adic integrable systems on $p$-adic analytic symplectic $4$-manifolds, hence $n=2$ in the definitions above. On the other hand, Theorem \ref{thm:num-integrable2}, refers to $p$-adic integrable systems on $p$-adic analytic symplectic manifolds of any dimension $2n$.

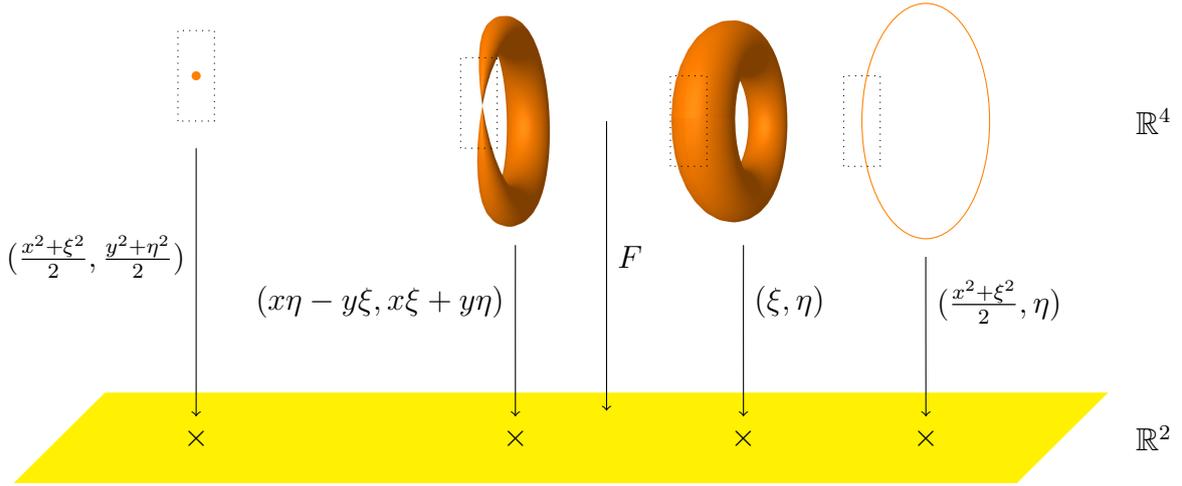
\begin{figure}
	\begin{tikzpicture}[scale=1.2]
		\fill[yellow] (-3,1.5)--(8,1.5)--(9,2.5)--(-2,2.5);
		\node (toro) at (5,4) {\includegraphics[height=2cm]{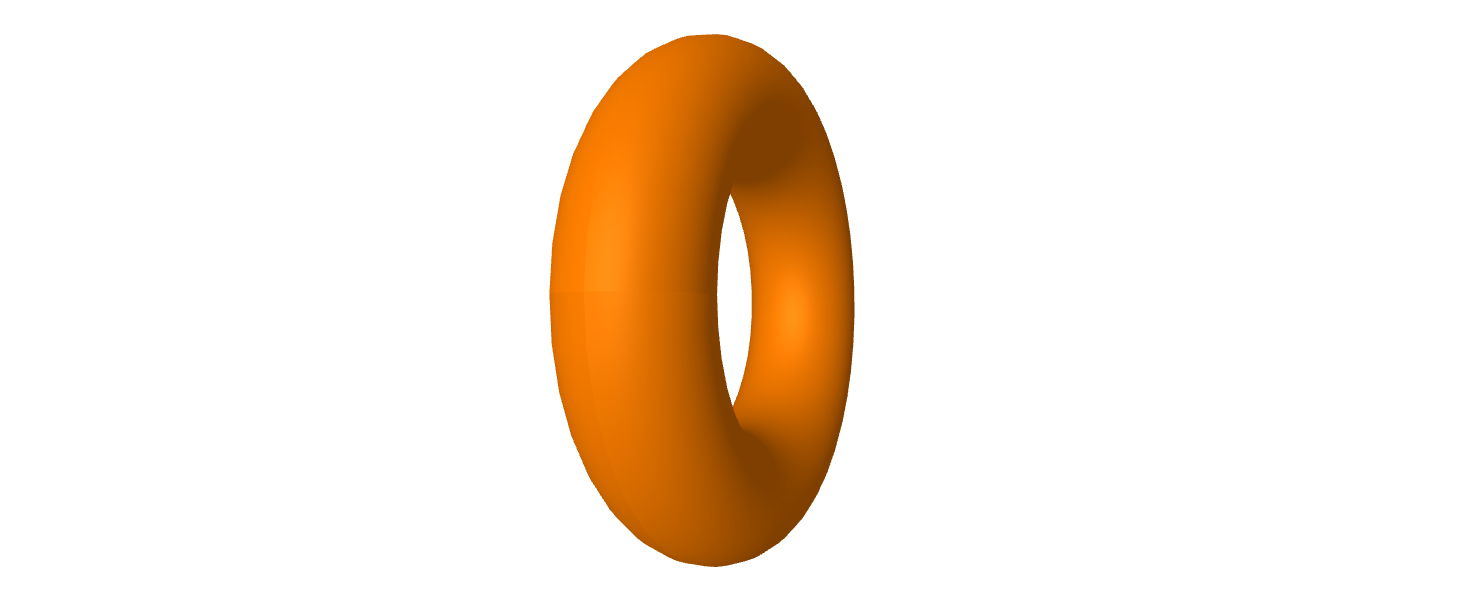}};
		\node (imtoro) at (5,2) {$\times$};
		\draw[->] (toro) -- (5,2.7) node[right] {$(\xi,\eta)$} -- (imtoro.north);
		\draw[dotted] (4.6,3.7)--(4.6,4.3)--(4.8,4.3)--(4.8,3.7)--(4.6,3.7);
		\node (toroap) at (2.5,4) {\includegraphics[height=2cm,trim=20cm 0 20cm 0,clip]{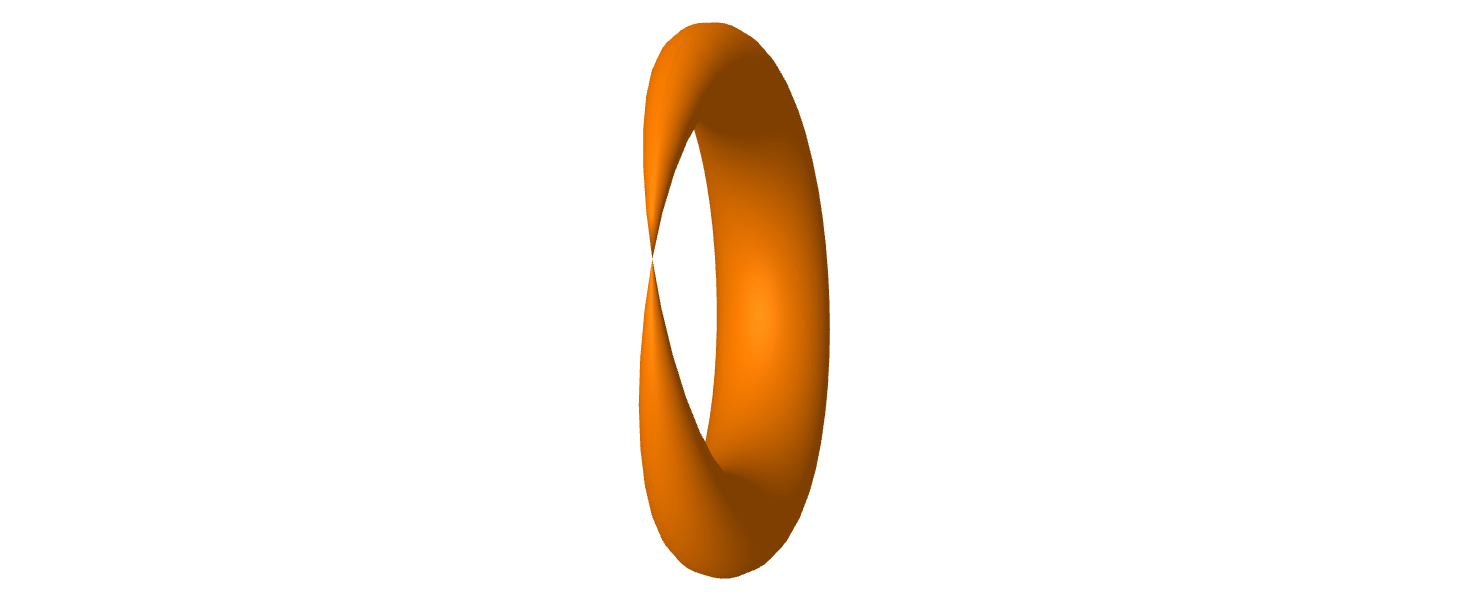}};
		\node (imtoroap) at (2.5,2) {$\times$};
		\draw[->] (toroap) -- (2.5,2.7) node[left] {$(x\eta-y\xi,x\xi+y\eta)$}--(imtoroap.north);
		\draw[dotted] (2.1,3.7)--(2.1,4.3)--(2.3,4.3)--(2.3,3.7)--(2.1,3.7);
		\draw[orange] (7,4) ellipse (0.4cm and 0.8cm);
		\node (imS1) at (7,2) {$\times$};
		\draw[->] (7,3) -- (7,2.7) node[right] {$(\frac{x^2+\xi^2}{2},\eta)$}--(imS1.north);
		\draw[dotted] (6.5,3.7)--(6.5,4.3)--(6.7,4.3)--(6.7,3.7)--(6.5,3.7);
		\fill[orange] (-1,4) circle (0.05);
		\node (impunto) at (-1,2) {$\times$};
		\draw[->] (-1,3.7) -- (-1,2.8) node[left] {$(\frac{x^2+\xi^2}{2},\frac{y^2+\eta^2}{2})$}--(impunto.north);
		\draw[dotted] (-1.1,3.7)--(-1.1,4.3)--(-0.9,4.3)--(-0.9,3.7)--(-1.1,3.7);
		\draw[->] (3.5,4) -- (3.5,3) node[right] {$F$} -- (3.5,2.2);
		\node (R2) at (9.5,2) {$\R^2$};
		\node (R4) at (9.5,4) {$\R^4$};
	\end{tikzpicture}
	\caption{Normal forms of regular and critical points of elliptic-elliptic, focus-focus and elliptic-regular type of an integrable system $F:\R^4\to\R^2$. Some of these can be normal forms of Theorem \ref{thm:integrable} (see Remark \ref{rem:elliptic}).}
	\label{fig:normal-forms}
\end{figure}

\begin{maintheorem}[$p$-adic integrable local models in dimension $4$]\label{thm:integrable}
	\letpprime. Let $X_p, Y_p, \mathcal{C}_i^k, \mathcal{D}_i^k$ be the non-residue sets and coefficient functions in Definition \ref{def:sets}. Let $(M,\omega)$ be a $p$-adic analytic symplectic manifold of dimension $4$ and let $F:(M,\omega)\to(\Qp)^2$ be a $p$-adic analytic integrable system. Let $m$ be a non-degenerate critical point of $F$. Then there exist local symplectic coordinates $(x,\xi,y,\eta)$ with the origin at $m$ and an invertible matrix $B\in\M_2(\Qp)$ such that in these coordinates we have:
	\begin{equation}\label{eq:integrable}
		B\circ(F-F(m))=(g_1,g_2)+\ocal(3),
	\end{equation}
	where the expression of $(g_1,g_2)$ depends on the rank of $m\in\{0,1\}$. If $m$ is a rank $0$ critical point then one of the following situations occurs:
	\begin{enumerate}
		\item There exist $c_1,c_2\in X_p$ such that $g_1(x,\xi,y,\eta)=x^2+c_1\xi^2,g_2(x,\xi,y,\eta)=y^2+c_2\eta^2$;
		\item There exists $c\in Y_p$ such that $g_1(x,\xi,y,\eta)=x\eta+cy\xi,g_2(x,\xi,y,\eta)=x\xi+y\eta$;
		\item There exist $c,t_1$ and $t_2$ corresponding to one row of Table \ref{table:canonical} and $(a,b)\in\{(1,0),(a_1,b_1)\}$, where $(a_1,b_1)$ is given in the row in question, such that
		\[g_k(x,\xi,y,\eta)=\sum_{i=0}^{2}\mathcal{C}_i^k(c,t_1,t_2,a,b)x^iy^{2-i}+\sum_{i=0}^{2}\mathcal{D}_i^k(c,t_1,t_2,a,b)\xi^i\eta^{2-i},\]
		for $k\in\{1,2\}$.
	\end{enumerate}
	Otherwise, if $m$ is a rank $1$ point, then there exists $c\in X_p$ such that $g_1(x,\xi,y,\eta)=x^2+c\xi^2$ and $g_2(x,\xi,y,\eta)=\eta$.
	
	Furthermore, if there are two sets of local symplectic coordinates in which $F$ has one of these forms, then the pair $(g_1,g_2)$ corresponding to the first set of local symplectic coordinates and the pair $(g_1',g_2')$ corresponding to the second set of local symplectic coordinates are in the same case; if it is case (2) or (3), or a rank $1$ point, they coincide, and in case (1) they coincide up to ordering.
\end{maintheorem}

\begin{table}
	\setlength{\parskip}{16pt}
	\setlength{\abovecaptionskip}{4pt}
	\footnotesize
	\begin{center}
		\begin{tabular}{|c|c|c|c|c|}\hline
			\multicolumn{5}{|c|}{$p\equiv 1\mod 4$} \\ \hline
			$c$ & $t_1$ & $t_2$ & $a_1$ & $b_1$ \\ \hline
			$c_0$ & $p$ & $0$ & $p$ & $1/p$ \\ \cline{2-5}
			& $0$ & $1$ & $p$ & $0$ \\ \cline{2-5}
			& $0$ & $p$ & $p$ & $0$ \\ \hline
			$p$ & $c_0$ & $0$ & $1$ & $1$ \\ \cline{2-5}
			& $0$ & $1$ & $c_0$ & $0$ \\ \cline{2-5}
			& $0$ & $c_0$ & $c_0$ & $0$ \\ \hline
			$c_0p$ & $c_0$ & $0$ & $1$ & $1$ \\ \cline{2-5}
			& $0$ & $1$ & $c_0$ & $0$ \\ \cline{2-5}
			& $0$ & $c_0$ & $c_0$ & $0$ \\ \hline
			\multicolumn{5}{|c|}{$p\equiv 3\mod 4$} \\ \hline
			$c$ & $t_1$ & $t_2$ & $a_1$ & $b_1$ \\ \hline
			$-1$ & $p$ & $0$ & $b_0$ & $a_0/b_0$ \\ \cline{2-5}
			& $a_0$ & $b_0$ & $p$ & $0$ \\ \cline{2-5}
			& $pa_0$ & $pb_0$ & $p$ & $0$ \\ \hline
			$p$ & $-1$ & $0$ & $1$ & $1$ \\ \cline{2-5}
			& $0$ & $1$ & $-1$ & $0$ \\ \hline
			$-p$ & $-1$ & $0$ & $1$ & $1$ \\ \cline{2-5}
			& $0$ & $1$ & $-1$ & $0$ \\ \hline
		\end{tabular}
		
		\begin{tabular}{|c|c|c|c|}\hline
			\multicolumn{4}{|c|}{$p=2\wedge c=-1$} \\ \hline
			$t_1$ & $t_2$ & $a_1$ & $b_1$ \\ \hline
			$2$ & $0$ & $1$ & $2$ \\ \hline
			$3$ & $0$ & $1$ & $1$ \\ \hline
			$6$ & $0$ & $1$ & $1$ \\ \hline
			$1$ & $1$ & $3$ & $0$ \\ \hline
			$3$ & $3$ & $3$ & $0$ \\ \hline
			$1$ & $2$ & $2$ & $0$ \\ \hline
			$2$ & $4$ & $2$ & $0$ \\ \hline
			$-1$ & $3$ & $2$ & $0$ \\ \hline
			$-2$ & $6$ & $2$ & $0$ \\ \hline
		\end{tabular}\qquad
		\begin{tabular}{|c|c|c|c|}\hline
			\multicolumn{4}{|c|}{$p=2\wedge c=2$} \\ \hline
			$t_1$ & $t_2$ & $a_1$ & $b_1$ \\ \hline
			$-1$ & $0$ & $1$ & $2$ \\ \hline
			$3$ & $0$ & $1$ & $1$ \\ \hline
			$-3$ & $0$ & $1$ & $1$ \\ \hline
			$0$ & $1$ & $-1$ & $0$ \\ \hline
			$0$ & $3$ & $-1$ & $0$ \\ \hline
			$1$ & $1$ & $-1$ & $0$ \\ \hline
			$3$ & $3$ & $-1$ & $0$ \\ \hline
			$2$ & $1$ & $3$ & $0$ \\ \hline
			$-2$ & $-1$ & $3$ & $0$ \\ \hline
			$6$ & $3$ & $3$ & $0$ \\ \hline
			$-6$ & $-3$ & $3$ & $0$ \\ \hline
		\end{tabular}\qquad
		\begin{tabular}{|c|c|c|c|}\hline
			\multicolumn{4}{|c|}{$p=2\wedge c=-2$} \\ \hline
			$t_1$ & $t_2$ & $a_1$ & $b_1$ \\ \hline
			$-1$ & $0$ & $1$ & $1$ \\ \hline
			$3$ & $0$ & $1$ & $-2$ \\ \hline
			$-3$ & $0$ & $1$ & $1$ \\ \hline
			$0$ & $1$ & $3$ & $0$ \\ \hline
			$0$ & $3$ & $3$ & $0$ \\ \hline
			$1$ & $1$ & $-1$ & $0$ \\ \hline
			$-1$ & $-1$ & $-1$ & $0$ \\ \hline
			$-2$ & $1$ & $-1$ & $0$ \\ \hline
			$2$ & $-1$ & $-1$ & $0$ \\ \hline
		\end{tabular}\qquad
		\begin{tabular}{|c|c|c|c|}\hline
			\multicolumn{4}{|c|}{$p=2\wedge c=3$} \\ \hline
			$t_1$ & $t_2$ & $a_1$ & $b_1$ \\ \hline
			$-1$ & $0$ & $1$ & $1$ \\ \hline
			$2$ & $0$ & $1$ & $1$ \\ \hline
			$-2$ & $0$ & $1$ & $3$ \\ \hline
			$0$ & $1$ & $2$ & $0$ \\ \hline
			$0$ & $2$ & $2$ & $0$ \\ \hline
			$1$ & $1$ & $-1$ & $0$ \\ \hline
			$-1$ & $-1$ & $-1$ & $0$ \\ \hline
			$3$ & $1$ & $-1$ & $0$ \\ \hline
			$-3$ & $-1$ & $-1$ & $0$ \\ \hline
		\end{tabular}
		
		\begin{tabular}{|c|c|c|c|}\hline
			\multicolumn{4}{|c|}{$p=2\wedge c=-3$} \\ \hline
			$t_1$ & $t_2$ & $a_1$ & $b_1$ \\ \hline
			$-1$ & $0$ & $2$ & $1/2$ \\ \hline
			$2$ & $0$ & $2$ & $1/2$ \\ \hline
			$-2$ & $0$ & $1$ & $-6$ \\ \hline
			$0$ & $1$ & $-1$ & $0$ \\ \hline
			$0$ & $2$ & $-1$ & $0$ \\ \hline
			$1$ & $2$ & $2$ & $0$ \\ \hline
			$-1$ & $-2$ & $2$ & $0$ \\ \hline
			$2$ & $4$ & $2$ & $0$ \\ \hline
			$-2$ & $-4$ & $2$ & $0$ \\ \hline
			$-6$ & $1$ & $-1$ & $0$ \\ \hline
			$-12$ & $2$ & $-1$ & $0$ \\ \hline
		\end{tabular}\qquad
		\begin{tabular}{|c|c|c|c|}\hline
			\multicolumn{4}{|c|}{$p=2\wedge c=6$} \\ \hline
			$t_1$ & $t_2$ & $a_1$ & $b_1$ \\ \hline
			$-1$ & $0$ & $1$ & $1$ \\ \hline
			$3$ & $0$ & $1$ & $6$ \\ \hline
			$-3$ & $0$ & $1$ & $1$ \\ \hline
			$0$ & $1$ & $3$ & $0$ \\ \hline
			$0$ & $3$ & $3$ & $0$ \\ \hline
			$1$ & $1$ & $-1$ & $0$ \\ \hline
			$-1$ & $-1$ & $-1$ & $0$ \\ \hline
			$6$ & $1$ & $-1$ & $0$ \\ \hline
			$-6$ & $-1$ & $-1$ & $0$ \\ \hline
		\end{tabular}\qquad
		\begin{tabular}{|c|c|c|c|}\hline
			\multicolumn{4}{|c|}{$p=2\wedge c=-6$} \\ \hline
			$t_1$ & $t_2$ & $a_1$ & $b_1$ \\ \hline
			$-1$ & $0$ & $1$ & $-6$ \\ \hline
			$3$ & $0$ & $1$ & $1$ \\ \hline
			$-3$ & $0$ & $1$ & $1$ \\ \hline
			$0$ & $1$ & $-1$ & $0$ \\ \hline
			$0$ & $3$ & $-1$ & $0$ \\ \hline
			$1$ & $1$ & $-1$ & $0$ \\ \hline
			$3$ & $3$ & $-1$ & $0$ \\ \hline
			$-6$ & $1$ & $3$ & $0$ \\ \hline
			$6$ & $-1$ & $3$ & $0$ \\ \hline
			$-18$ & $3$ & $3$ & $0$ \\ \hline
			$18$ & $-3$ & $3$ & $0$ \\ \hline
		\end{tabular}
	\end{center}
	\caption{Parameters for the normal form (3) of Theorem \ref{thm:integrable}.
		In the table, for $p\equiv 1\mod 4$, $c_0$ is the smallest quadratic non-residue modulo $p$. For $p\equiv 3\mod 4$, $a_0$ and $b_0$ are such that $a_0^2+b_0^2\equiv -1\mod p$. For $p=2$ there are many more possible parameters, and they are separated by the value of $c$.}
	\label{table:canonical}
\end{table}

\begin{remark}
	We have a result analogous to Theorem \ref{thm:integrable} about degenerate systems (Theorems \ref{thm:real-degenerate} and \ref{thm:padic-degenerate}), though we have not been able to prove uniqueness in that case (that is, to determine which models are equivalent to each other). We refer to Figures \ref{fig:fibers1} and \ref{fig:fibers2} for a depiction of some of the cases covered by Theorem \ref{thm:integrable} and to Figure \ref{fig:normal-forms} for an illustration of the real case.
\end{remark}

The following statement gives a precise count of the normal forms appearing in 
Theorem \ref{thm:integrable}.

\begin{maintheorem}[Number of $p$-adic integrable local models, in dimension $4$]\label{thm:num-integrable}
	\letpprime. Let $X_p, Y_p, \mathcal{C}_i^k, \mathcal{D}_i^k$ be the non-residue sets and coefficient functions in Definition \ref{def:sets}. Let $(M,\omega)$ be a $p$-adic analytic symplectic $4$-manifold. Then the following statements hold:
	\begin{enumerate}
		\item If $p\equiv 1\mod 4$, there are exactly $49$ normal forms for a rank $0$ non-degenerate critical point, and exactly $7$ normal forms of a rank $1$ non-degenerate critical point, of a $p$-adic analytic integrable system $F:(M,\omega)\to(\Qp)^2$ up to local symplectomorphisms centered at the critical point;
		\item If $p\equiv 3\mod 4$, there are exactly $32$ normal forms for a rank $0$ non-degenerate critical point, and exactly $5$ normal forms of a rank $1$ non-degenerate critical point, of a $p$-adic analytic integrable system $F:(M,\omega)\to(\Qp)^2$ up to local symplectomorphisms centered at the critical point;
		\item If $p=2$, there are exactly $211$ normal forms for a rank $0$ non-degenerate critical point, and exactly $11$ normal forms of a rank $1$ non-degenerate critical point, of a $p$-adic analytic integrable system $F:(M,\omega)\to(\Qp)^2$ up to local symplectomorphisms centered at the critical point.
	\end{enumerate}
	
	In the three cases above, the normal forms for a rank $0$ point are given by
	\[\Big\{(x^2+c_1\xi^2,y^2+c_2\eta^2):c_1,c_2\in X_p\Big\}
	\cup\Big\{(x\eta+cy\xi,x\xi+y\eta):c\in Y_p\Big\}\]
	\[\cup\Big\{\Big(\sum_{i=0}^{2}\mathcal{C}_i^1(c,t_1,t_2,a,b)x^iy^{2-i}+\sum_{i=0}^{2}\mathcal{D}_i^1(c,t_1,t_2,a,b)\xi^i\eta^{2-i},\]\[\sum_{i=0}^{2}\mathcal{C}_i^2(c,t_1,t_2,a,b)x^iy^{2-i}+\sum_{i=0}^{2}\mathcal{D}_i^2(c,t_1,t_2,a,b)\xi^i\eta^{2-i}\Big):\]\[(a,b)\in\Big\{(1,0),(a_1,b_1)\Big\},c,t_1,t_2,a_1,b_1\text{ in one row of Table \ref{table:canonical}}\Big\}\]
	and those for a rank $1$ point are given by
	\[\Big\{(x^2+c\xi^2,\eta):c\in X_p\Big\}.\]
\end{maintheorem}

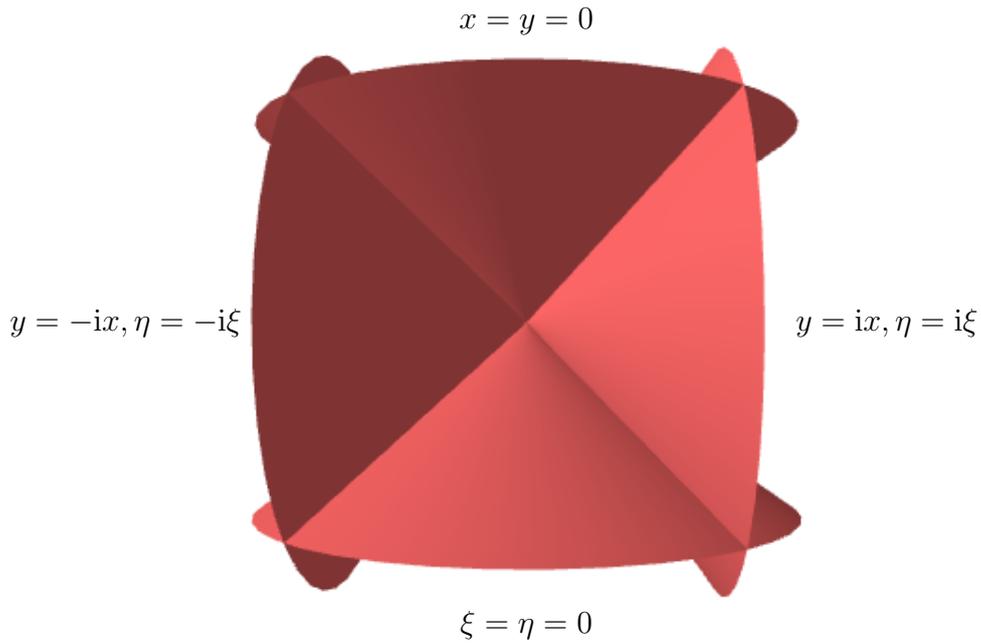
\begin{figure}
	\begin{tikzpicture}
		\node at (0,0) {\includegraphics[trim=5cm 1cm 5cm 1cm,scale=0.7,clip]{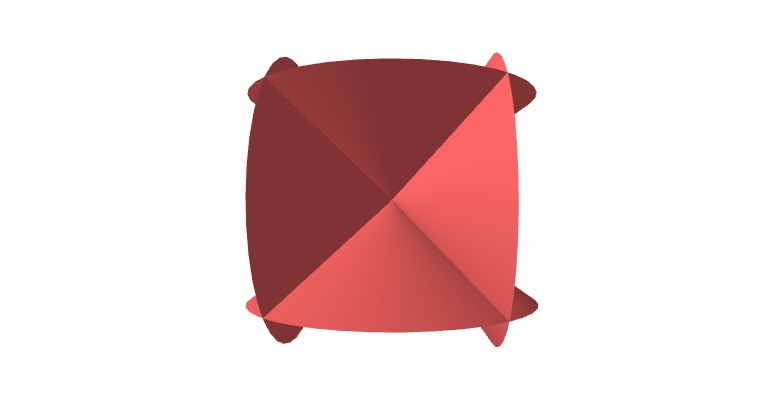}};
		\node at (0,4) {$x=y=0$};
		\node at (0,-4) {$\xi=\eta=0$};
		\node[right] at (3.4,0) {$y=\ii x,\eta=\ii \xi$};
		\node[left] at (-3.6,0) {$y=-\ii x,\eta=-\ii\xi$};
	\end{tikzpicture}
	\caption{Symbolic representation of $2$-dimensional fiber of focus-focus model if $p\equiv 1\mod 4$, as a case of point (1) of Theorem \ref{thm:integrable}, which coincides with the elliptic-elliptic model. The four ``cones'' are $2$-dimensional planes in $4$-dimensional space.}
	\label{fig:fibers1}
\end{figure}

\begin{figure}
	\begin{tikzpicture}
		\node at (0,0) {\includegraphics[trim=5cm 0cm 5cm 1cm,scale=0.7,clip]{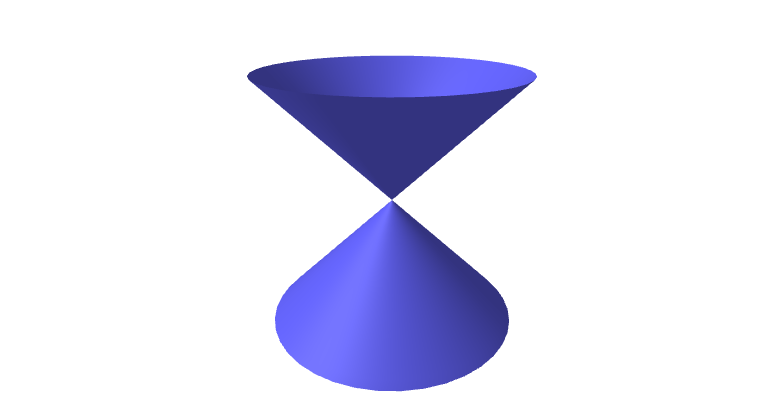}};
		\node at (0,4.5) {$x=y=0$};
		\node at (0,-5) {$\xi=\eta=0$};
	\end{tikzpicture}
	\caption{Symbolic representation of $2$-dimensional fiber of focus-focus model if $p\not\equiv 1\mod 4$, as a case of point (2) of Theorem \ref{thm:integrable}, which coincides with the fiber in the real case. The two ``cones'' are actually $2$-dimensional planes in $4$-dimensional space that meet at a point.}
	\label{fig:fibers2}
\end{figure}

\begin{remark}
	In the real case, there are exactly $4$ normal forms for a rank $0$ non-degenerate critical point:
	\[\Big\{(x^2+\xi^2,y^2+\eta^2),(x^2+\xi^2,y\eta),(x\xi,y\eta),(x\eta-y\xi,x\xi+y\eta)\Big\}\]
	and exactly $2$ normal forms for a rank $1$ non-degenerate critical point:
	\[\Big\{(x^2+\xi^2,\eta),(x\xi,\eta)\Big\}.\]
\end{remark}

\begin{remark}\label{rem:elliptic}
	In the $p$-adic category, the elliptic-elliptic and elliptic-regular points are normal forms of Theorem \ref{thm:integrable}: the former is a rank $0$ point in case (1) with $c_1=c_2=1$, and the latter is a rank $1$ point with $c=1$. The focus-focus point may also appear as a normal form, in case (2) with $c=-1$, but only if $-1\in Y_p$, which happens if $p\not\equiv 1\mod 4$. Actually, if $p\equiv 1\mod 4$, the normal form of a focus-focus point is elliptic-elliptic; hence, the focus-focus and elliptic-elliptic points are locally symplectomorphic if and only if $p\equiv 1\mod 4$.
\end{remark}

\begin{definition}\label{def:f}
	\letpprime\ and let $n$ be a positive integer. To each partition $P=(a_1,\ldots,a_k)$ of $n$ we associate the function
	\[f_{P,p}(x_1,\xi_1,\ldots,x_n,\xi_n)=\sum_{i=1}^k\left(\frac{x_{b_{i-1}+1}^2}{2}+\sum_{j=b_{i-1}+1}^{b_i-1}\xi_j x_{j+1}+\frac{p\xi_{b_i}^2}{2}\right),\]
	where $b_0=0$ and $b_i=\sum_{j=0}^{i}a_j$, for $i\in\{1,\ldots,k\}$.
\end{definition}

Note that, for every partition $P$, $f_{P,p}$ has a critical point at the origin.

\begin{maintheorem}[Number of $p$-adic integrable local models, arbitrary dimension]\label{thm:num-integrable2}
	Let $n$ be a positive integer. \letpprime. The number of local normal forms of $p$-adic analytic integrable systems on $2n$-dimensional $p$-adic analytic symplectic manifolds at a rank $0$ non-degenerate critical point, up to local symplectomorphisms centered at the critical point, grows asymptotically at least with
	\[\frac{\e^{\pi\sqrt{2n/3}}}{4n\sqrt{3}}.\]
	Explicitly, for any two partitions $P$ and $Q$ of $n$, any two $p$-adic analytic integrable systems on a $2n$-dimensional $p$-adic analytic manifold containing $f_{P,p}$ and $f_{Q,p}$, respectively, as a component, are not equivalent by local symplectomorphisms centered at the origin, where $f_{P,p}$ is as given in Definition \ref{def:f}.
\end{maintheorem}

\begin{table}[h]
	\begin{tabular}{|*{6}{c|}}\hline
		$2n$ & $\R$ & $\Q_2$ & $\Q_3$ & $\Q_5$ & $\Q_7$ \\ \hline
		2 & 2 & 11 & 5 & 7 & 5 \\ \hline
		4 & 4 & 211 & 32 & 49 & 32 \\ \hline
		6 & 6 & 1883 & 123 & 234 & 129 \\ \hline
		8 & 9 & 21179 & 495 & 1054 & 525 \\ \hline
		10 & 12 & 161343 & 1595 & 4021 & 1787 \\ \hline
	\end{tabular}
	\caption{The number of families of normal forms of integrable systems on $\R^{2n}$ and $(\Qp)^{2n}$ at a rank $0$ critical point. Data is extracted from Theorem \ref{thm:num-forms-lower-bound}. For the $p$-adic case in dimension greater than $4$, the numbers are only lower bounds. The actual number of forms might be even larger.}
	\label{table:num-integrable}
\end{table}

\begin{remark}
	These results indicate that a global theory of $p$-adic integrable systems, which will probably be based on gluing local models, will include a large number of phenomena which do not occur in the real case. We have computed explicit lower bounds of the number of normal forms in Table \ref{table:num-integrable}.
\end{remark}

\begin{remark}
	The Weierstrass-Williamson theory of matrices has crucial applications in many areas including the theory of quantum states in quantum physics \cite{Gosson,SSM,WHTH}, hence this paper provides a new tool to further explore $p$-adic analogues of these applications in symplectic geometry and beyond.
\end{remark}

The field of $p$-adic geometry is extensive, see \cite{Schneider,SchWei} and the references therein. We recommend that readers consult Zelenov's article \cite{Zelenov} for a construction involving the $p$-adic numbers in the context of symplectic vector spaces, the $p$-adic Heisenberg group and the Maslov index; see also the article by Hu-Hu \cite{HuHu}. $p$-adic geometry is also fundamental in mathematical physics and the theory of integrable systems, see for example \cite{CrePel-JC,Dragovich-quantum,Dragovich-harmonic,DKKV,GKPSW,PVW,VlaVol}. For an introduction to different aspects of symplectic geometry, including its relations to mechanics and Poisson geometry, we refer to the survey articles \cite{Eliashberg,Pelayo-hamiltonian,Weinstein-symplectic,Weinstein-Poisson} and the books \cite{BatWei,Cannas,HofZeh,MarRat,McDSal,OrtRat}. We believe that the ideas concerning local normal forms in this paper, in conjunction with \cite{GLSS}, could help establish a classification of line bundles over $p$-adic analytic symplectic manifolds.

\subsection*{Structure of the paper}

In Section \ref{sec:matrices} we state the other main results of the paper, which concern matrices and are proved in \cite{CrePel-matrix}. Section \ref{sec:algclosed} studies the problem of the symplectic classification of matrices in an algebraically closed field, where it is much simpler to give a classification. Section \ref{sec:singularities} formulates these theorems concerning matrices in terms of non-degeneracy of critical points of functions. Section \ref{sec:integrable} applies the work of the previous section to integrable systems. Section \ref{sec:JC} applies the results about integrable systems to the concrete case of the $p$-adic Jaynes-Cummings system. Section \ref{sec:examples} contains some examples and remarks. Section \ref{sec:circle} discusses the circle actions generated by the integrable systems of Theorem \ref{thm:integrable}. In Section \ref{sec:final} we make some final remarks. We conclude the paper with an appendix which recalls the basic definitions and results concerning the $p$-adic numbers.

\medskip
\textbf{Acknowledgments.}
We thank Pilar Bayer and Enrique Arrondo for discussions and suggestions which have improved the paper. We are really grateful to an anonymous referee for helpful comments which have improved the paper.

The second author thanks the Dean of the School of Mathematical Sciences Antonio Br\'u and the Chair of the Department of Algebra, Geometry and Topology at the Complutense University of Madrid, Rutwig Campoamor, for their support and excellent resources he is being provided with to carry out the FBBVA project.

Part of this paper was written during a visit of the second author to the Universidad de Cantabria and the Universidad Internacional Men\'endez Pelayo in the summer of 2024 and he thanks both institutions for their hospitality. Another part of this paper was written while the first author was visiting the Complutense University of Madrid in September-October 2024 and he thanks this institution.

\medskip
\textbf{Funding.}
The first author is funded by grant PID2022-137283NB-C21 of MCIN/AEI/ 10.13039/501100011033 / FEDER, UE. The second author is funded by a FBBVA (Bank Bilbao Vizcaya Argentaria Foundation) Grant for Scientific Research Projects with title \textit{From Integrability to Randomness in Symplectic and Quantum Geometry}.

\section{Main results concerning matrices: Theorems \ref{thm:williamson}--\ref{thm:num-forms-lower-bound}}\label{sec:matrices}

In this section we state our main classification results concerning normal forms of $p$-adic matrices in dimensions $2$ and $4$: Theorems \ref{thm:williamson}--\ref{thm:num-forms-lower-bound}. We will use these results as stepping stones to prove the results concerning integrable systems (Theorems \ref{thm:integrable}--\ref{thm:num-integrable2} stated in Section \ref{sec:intro}), but they are also of independent interest and they can be read independently of all the material concerning $p$-adic analytic integrable systems and $p$-adic analytic functions. The proofs of these results are long and tedious and are better organized as a separate paper: we prove them in the sequel \cite{CrePel-matrix} to the present paper.

\subsection{The real Weierstrass-Williamson classification}\label{sec:matrices-real}

Let $\Omega_0$ be the matrix of the standard symplectic form $\omega_0$ on $\R^{2n}$, that is, a block-diagonal matrix of size $2n$ with all blocks equal to
\[\begin{pmatrix}
	0 & 1 \\
	-1 & 0
\end{pmatrix},\]
i.e.
\[\Omega_0=\begin{pmatrix}
	0 & 1 &   &   &   &   &   \\
	-1& 0 &   &   &   &   &   \\
	  &   & 0 & 1 &   &   &   \\
	  &   &-1 & 0 &   &   &   \\
	  &   &   &   & \ddots &   &   \\
	  &   &   &   &   & 0 & 1 \\
	  &   &   &   &   &-1 & 0
\end{pmatrix}\]

\begin{definition}[Symplectic matrix]
	A matrix $S\in\M_{2n}(\R)$ is \emph{symplectic} if $S^T\Omega_0S=\Omega_0$, that is, it leaves invariant the standard symplectic form (the same definition applies if we replace $\R$ by an arbitrary field $F$).
\end{definition}
For a symmetric matrix $M\in\M_{2n}(\R)$ such that the eigenvalues of $\Omega_0^{-1}M$ are pairwise distinct, the Weierstrass-Williamson classification says that there is a symplectic matrix $S$ such that $S^T MS=N$, where
$N$ is a block-diagonal matrix with each block equal to
\[\begin{pmatrix}
	r_i & 0 \\
	0 & r_i
\end{pmatrix},
\begin{pmatrix}
	0 & r_i \\
	r_i & 0
\end{pmatrix}
\text{ or }
\begin{pmatrix}
	0 & r_{i+1} & 0 & r_i \\
	r_{i+1} & 0 & -r_i & 0 \\
	0 & -r_i & 0 & r_{i+1} \\
	r_i & 0 & r_{i+1} & 0
\end{pmatrix},\]
for some $r_i\in\R,1\le i\le n$, which are called \textit{elliptic block}, \textit{hyperbolic block} and \textit{focus-focus block}.

Quite often the Weierstrass-Williamson classification is stated only for positive-definite matrices: in fact the condition that all eigenvalues of $\Omega_0^{-1}M$ are pairwise distinct is implied by $M$ being positive definite, and in this case only the elliptic block appears. (For applications to integrable systems the condition on $\Omega_0^{-1}M$ ``translates'' to the notion of a critical point being \emph{non-degenerate}; we'll see this later.) In fact, this is the particular case of what is often called Williamson's theorem (that is, what we call the Weierstrass-Williamson classification) which is due to Weierstrass \cite{Weierstrass} (we learned this fact from the book by Hofer and Zehnder \cite[Theorem 8]{HofZeh}).

Of course, the case which Williamson treats is much more complicated and interesting: he is able to deal with the completely general situation in which the eigenvalues are not necessarily pairwise distinct. The problem with this case is that it is not feasible to write all the possibilities, for arbitrary dimension, in a compact form because the size of the blocks which are needed can increase without bound. However Williamson does provide a complete list of $4$-by-$4$ matrix normal forms at the end of his paper \cite{Williamson}. These are the possibilities, expressed in a different way to align with the conventions of our paper:
\[\begin{pmatrix}
	0 & r & 0 & 0 \\
	r & 0 & 0 & 0 \\
	0 & 0 & 0 & s \\
	0 & 0 & s & 0
\end{pmatrix},
\begin{pmatrix}
	0 & 0 & 0 & 0 \\
	0 & a & 0 & 0 \\
	0 & 0 & 0 & s \\
	0 & 0 & s & 0
\end{pmatrix},
\begin{pmatrix}
	0 & 0 & 0 & 0 \\
	0 & a & 0 & 0 \\
	0 & 0 & 0 & 0 \\
	0 & 0 & 0 & b
\end{pmatrix},\]
\[\begin{pmatrix}
	0 & r & 0 & 0 \\
	r & 0 & 1 & 0 \\
	0 & 1 & 0 & r \\
	0 & 0 & r & 0
\end{pmatrix},
\begin{pmatrix}
	0 & 0 & 0 & 0 \\
	0 & 0 & 1 & 0 \\
	0 & 1 & 0 & 0 \\
	0 & 0 & 0 & a
\end{pmatrix},
\begin{pmatrix}
	0 & r & 0 & 0 \\
	r & 0 & 0 & 0 \\
	0 & 0 & s & 0 \\
	0 & 0 & 0 & s
\end{pmatrix},\]
\[\begin{pmatrix}
	0 & 0 & 0 & 0 \\
	0 & a & 0 & 0 \\
	0 & 0 & s & 0 \\
	0 & 0 & 0 & s
\end{pmatrix},
\begin{pmatrix}
	r & 0 & 0 & 0 \\
	0 & r & 0 & 0 \\
	0 & 0 & s & 0 \\
	0 & 0 & 0 & s
\end{pmatrix},
\begin{pmatrix}
	a & 0 & 0 & r \\
	0 & 0 & -r & 0 \\
	0 & -r & a & 0 \\
	r & 0 & 0 & 0
\end{pmatrix},
\begin{pmatrix}
	0 & s & 0 & r \\
	s & 0 & -r & 0 \\
	0 & -r & 0 & s \\
	r & 0 & s & 0
\end{pmatrix},\]
where $r,s\in\R$ and $a,b\in\{1,-1\}$. Since in the present paper we only provide a classification of $p$-adic $4$-by-$4$ matrices (also of $p$-adic $2$-by-$2$ matrices, but this case is simpler), it is precisely the list above which is most relevant to us. Nonetheless we state the classification in full generality in Appendix \ref{sec:real}. We will provide a new proof of the classification in \cite{CrePel-matrix}.

\subsection{The $p$-adic Weierstrass-Williamson classification}\label{sec:padic-williamson}

In this article, we will give such a classification in the case when the real field $\R$ is replaced by the $p$-adic field $\Qp$, where $p$ is an arbitrary prime number, and $V$ has dimension $2$ or $4$. Recall that the \emph{field of $p$-adic numbers} is defined as follows: it is the completion of $\Q$ with respect to a non-archimedean absolute value defined as \[|x|_p=p^{-\ord_p(x)},\] where $\ord_p(x)$ is the greatest power of $p$ that divides $x$. The basic concepts about $p$-adic numbers can be found in \cite{Gouvea,Schneider}.

The classification is completely different for $p=2$ than $p\ne 2$; for the latter case, in turn, it depends on the class of $p$ modulo $4$.

Our main results are the following classifications of $p$-adic symmetric matrices.

\begin{maintheorem}[$p$-adic classification, $2$-by-$2$ case]\label{thm:williamson}
	\letpprime. Let $M\in\M_2(\Qp)$ be a symmetric matrix. Let $X_p,Y_p$ be the non-residue sets in Definition \ref{def:sets}. Then, there exists a symplectic matrix $S\in\M_2(\Qp)$ and either $c\in X_p$ and $r\in\Qp$, or $c=0$ and $r\in Y_p\cup\{1\}$, such that
	\[S^T MS=\begin{pmatrix}
		r & 0 \\
		0 & cr
	\end{pmatrix}.\]
	Furthermore, if two symplectic matrices $S$ and $S'$ reduce $M$ to the normal form of the right hand side of the equality above, then the two normal forms are equal.
\end{maintheorem}

In Theorem \ref{thm:williamson} we are not saying that the value of $(c,r)$ is unique (which essentially is, but not quite), but that the canonical matrix obtained at the right-hand side is unique.

In the statement below, and also Theorem \ref{thm:num-forms1}, we use the following terminology.

\begin{definition}
	Let $n$ be a positive integer. \letpprime. We say that two $(2n)$-by-$(2n)$ matrices $M$ and $M'$ with coefficients in $\Qp$ are \emph{congruent via a symplectic matrix} if there exists a $(2n)$-by-$(2n)$ symplectic matrix $S$ with coefficients in $\Qp$ such that $S^TMS=M'$ (same definition works for arbitrary fields).
\end{definition}

The number of normal forms appearing in Theorem \ref{thm:williamson} is as follows:

\begin{maintheorem}[Number of inequivalent $p$-adic $2$-by-$2$ matrix normal forms]\label{thm:num-forms2}
	\letpprime. Let $X_p,Y_p$ be the non-residue sets in definition \ref{def:sets}. Then the following statements hold.
	\begin{enumerate}
		\item If $p\equiv 1\mod 4$, there are exactly $7$ infinite families of normal forms of $2$-by-$2$ $p$-adic matrices with one degree of freedom up to congruence via a symplectic matrix:
		\[\Big\{\Big\{r\begin{pmatrix}
			1 & 0 \\
			0 & c
		\end{pmatrix}:r\in\Qp\Big\}:c\in X_p\Big\},\]
		and exactly $4$ isolated normal forms, which correspond to $c=0$:
		\[\Big\{\begin{pmatrix}
			r & 0 \\
			0 & 0
		\end{pmatrix}:r\in Y_p\cup\{1\}\Big\}.\]
		\item If $p\equiv 3\mod 4$, there are exactly $5$ infinite families of normal forms of $2$-by-$2$ $p$-adic matrices with one degree of freedom up to congruence via a symplectic matrix, with the same formula as above, and exactly $4$ isolated normal forms.
		\item If $p=2$, there are exactly $11$ infinite families of normal forms of $2$-by-$2$ $p$-adic matrices with one degree of freedom up to congruence via a symplectic matrix, also with the same formula, and exactly $8$ isolated normal forms.
	\end{enumerate}
	This is in contrast with the real case, where there are exactly $2$ families, elliptic and hyperbolic, and $2$ isolated normal forms. Here by ``infinite family'' we mean a family of normal forms of the form $r_1M_1+r_2M_2+\ldots+r_kM_k$, where $r_i\in\Qp$ are parameters and $k$ is the number of degrees of freedom.
\end{maintheorem}

Hence, already in dimension $2$, the $p$-adic situation is much richer than its real counterpart. The situation is even more surprising in dimension $4$. This is the classification in the case where $\Omega_0^{-1}M$ has all eigenvalues distinct, where $\Omega_0$ is the same matrix as before for dimension $4$:
\[\Omega_0=\begin{pmatrix}
	0 & 1 & 0 & 0 \\
	-1 & 0 & 0 & 0 \\
	0 & 0 & 0 & 1 \\
	0 & 0 & -1 & 0
\end{pmatrix}\]

\begin{maintheorem}[$p$-adic classification, $4$-by-$4$, non-degenerate case]\label{thm:williamson4}
	\letpprime. Let $\Omega_0$ be the matrix of the standard symplectic form on $(\Qp)^4$. Let $X_p,Y_p$ be the non-residue sets in Definition \ref{def:sets}. Let $M\in\M_4(\Qp)$ be a symmetric matrix such that all the eigenvalues of $\Omega_0^{-1}M$ are distinct. Then there exists a symplectic matrix $S\in\M_4(\Qp)$ and $r,s\in\Qp$ such that one of the following three possibilities holds:
	\begin{enumerate}
		\item There exist $c_1,c_2\in X_p$ such that
		\[S^TMS=\begin{pmatrix}
			r & 0 & 0 & 0 \\
			0 & c_1r & 0 & 0 \\
			0 & 0 & s & 0 \\
			0 & 0 & 0 & c_2s
		\end{pmatrix}.\]
		\item There exists $c\in Y_p$ such that
		\[S^TMS=\begin{pmatrix}
			0 & s & 0 & r \\
			s & 0 & cr & 0 \\
			0 & cr & 0 & s \\
			r & 0 & s & 0
		\end{pmatrix}.\]
		\item There exist $c,t_1$ and $t_2$ corresponding to one row of the Table \ref{table:canonical} such that $S^TMS$ is equal to the matrix
		\[
		\renewcommand{\arraystretch}{2}
		\begin{pmatrix}
			\dfrac{ac(r-bs)}{c-b^2} & 0 & \dfrac{sc-rb}{c-b^2} & 0 \\
			0 & \dfrac{-r(t_1+bt_2)-s(bt_1+ct_2)}{a} & 0 & -r(bt_1+ct_2)-sc(t_1+bt_2) \\
			\dfrac{sc-rb}{c-b^2} & 0 & \dfrac{r-bs}{a(c-b^2)} & 0 \\
			0 & -r(bt_1+ct_2)-sc(t_1+bt_2) & 0 & ac(-r(t_1+bt_2)-s(bt_1+ct_2))
		\end{pmatrix}\]
		where $(a,b)$ are either $(1,0)$ or $(a_1,b_1)$ of the corresponding row.
	\end{enumerate}
	Furthermore, if two matrices $S$ and $S'$ reduce $M$ to one of the normal forms in the right-hand side of the three equalities above, then the two normal forms are in the same case; if it is case (2) or (3), they coincide, and in case (1) they coincide up to exchanging the $2$ by $2$ diagonal blocks. Moreover, the family of normal forms (that is, the normal form except for $r$ and $s$) and the matrix $S$ are uniquely determined by the eigenvectors of $\Omega_0^{-1}M$.
\end{maintheorem}

See Figure \ref{fig:formas} for a diagram of the classes in the statement of Theorem \ref{thm:williamson4}.

\begin{figure}
	\setlength{\parskip}{0.5cm}
	\includegraphics[scale=0.8]{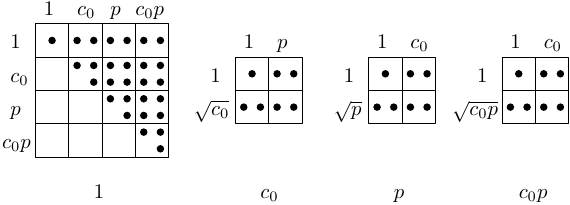}
	
	\includegraphics[scale=0.8]{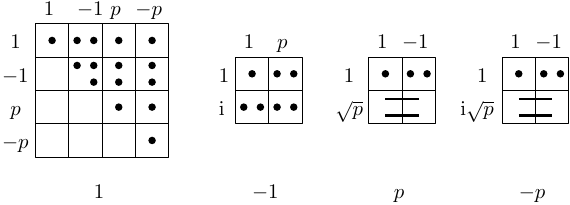}
	
	\includegraphics[scale=0.8]{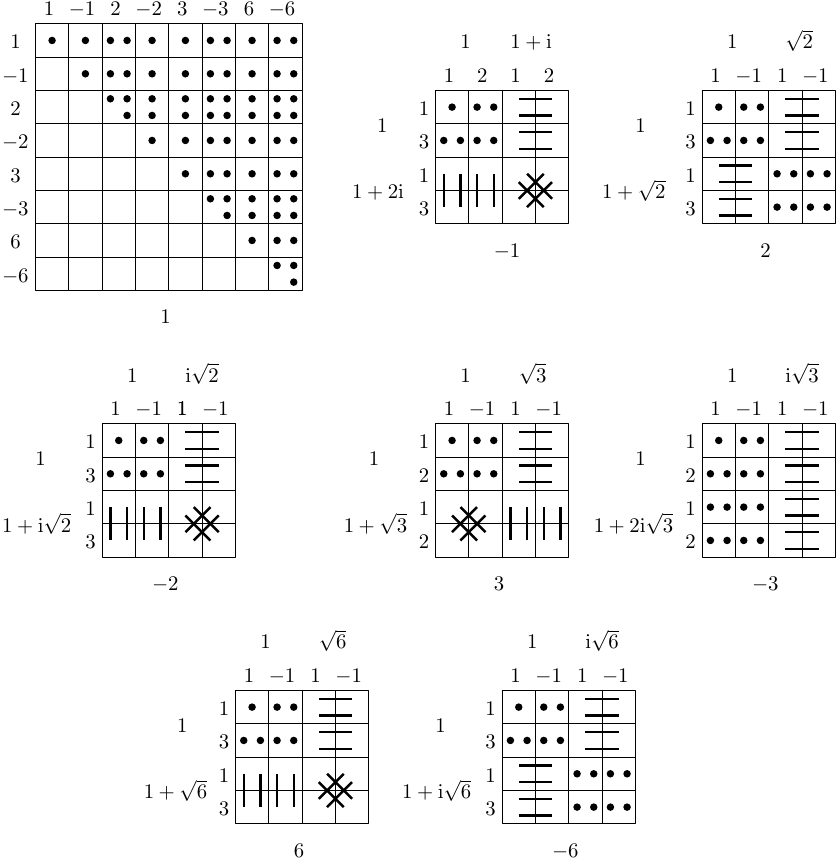}
	\caption{A diagram of the normal forms of Theorem \ref{thm:williamson4} for $p\equiv 1\mod 4$ (first row), $p\equiv 3\mod 4$ (second row) and $p=2$ (third to fifth row). Each point (if it is in a single cell) or line (if it is in two cells) represents a normal form. The numbers below the tables represent the first extension of $\Qp$ (the one containing the squares of the eigenvalues of $\Omega_0^{-1}M$) and those in the rows and columns represent the second extension (for the eigenvalues themselves). In the first table in each block, there are two such extensions corresponding to the row and the column; in the other ones, there is one extension obtained multiplying the numbers in the row and column.}
	\label{fig:formas}
\end{figure}

\begin{definition}[Non-residue function]\label{def:h}
	\letpprime. If $p\equiv 1\mod 4$, let $c_0$ be the smallest quadratic non-residue modulo $p$. We define the \emph{non-residue function}:
	\[h_p:Y_p\to\Qp\text{ given by }\begin{cases}
		h_p(c_0)=p,h_p(p)=h_p(c_0p)=c_0 & \text{if }p\equiv 1\mod 4; \\
		h_p(-1)=p,h_p(p)=h_p(-p)=-1 & \text{if }p\equiv 3\mod 4; \\
		h_p(-1)=h_p(-2)=h_p(3)=h_p(6)=-1, & \\
		h_p(-3)=h_p(-6)=2,h_p(2)=3 & \text{if }p=2.
	\end{cases}\]
\end{definition}

For the cases where the eigenvalues of $\Omega_0^{-1}M$ are not pairwise distinct, we have:

\begin{maintheorem}[$p$-adic classification, $4$-by-$4$, degenerate case]\label{thm:williamson4-deg}
	\letpprime. Let $\Omega_0$ be the matrix of the standard symplectic form on $(\Qp)^4$. Let $X_p,Y_p$ be the non-residue sets in Definition \ref{def:sets}. Let $h_p:Y_p\to\Qp$ be the non-residue function in Definition \ref{def:h}. Let $M\in\M_4(\Qp)$ a symmetric matrix such that $\Omega_0^{-1}M$ has at least one multiple eigenvalue. Then there exists a symplectic matrix $S\in\M_4(\Qp)$ such that one of the following three possibilities holds:
	\begin{enumerate}
		\item There exist $r,s\in\Qp$ and $c_1,c_2\in X_p\cup\{0\}$ such that $S^TMS$ has the form in the case (1) of Theorem \ref{thm:williamson4}. Moreover, if $c_1=0$ then $r\in Y_p\cup\{1\}$, and if $c_2=0$ then $s\in Y_p\cup\{1\}$.
		\item There exists $r\in\Qp$ such that
		\[S^TMS=\begin{pmatrix}
			0 & r & 0 & 0 \\
			r & 0 & 1 & 0 \\
			0 & 1 & 0 & r \\
			0 & 0 & r & 0
		\end{pmatrix}\]
		\item There exist $r\in\Qp$, $c\in Y_p$ and $a\in \{1,h_p(c)\}$ such that
		\[S^TMS=\begin{pmatrix}
			a & 0 & 0 & r \\
			0 & 0 & cr & 0 \\
			0 & cr & a & 0 \\
			r & 0 & 0 & 0
		\end{pmatrix}.\]
		\item There exists $c\in Y_p\cup\{1\}$ such that
		\[S^TMS=\begin{pmatrix}
			c & 0 & 0 & 0 \\
			0 & 0 & 0 & -c \\
			0 & 0 & 0 & 0 \\
			0 & -c & 0 & 0
		\end{pmatrix}.\]
	\end{enumerate}
	Furthermore, if two matrices $S$ and $S'$ reduce $M$ to one of these normal forms on the right-hand side of the three equalities above, then the two normal forms are in the same case; if it is case (2), (3) or (4), they coincide completely, and in case (1) they coincide up to exchanging the $2$ by $2$ diagonal blocks.
\end{maintheorem}

The following statement quantifies precisely the amount of normal forms in Theorems \ref{thm:williamson4} and \ref{thm:williamson4-deg}.

\begin{maintheorem}[Number of $4$-by-$4$ $p$-adic matrix normal forms]\label{thm:num-forms1}
	\letpprime. Let $X_p,Y_p$ be the non-residue sets in Definition \ref{def:sets}. Let $h_p:Y_p\to\Qp$ be the non-residue function in Definition \ref{def:h}.
	\begin{enumerate}
		\item If $p\equiv 1\mod 4$, there are exactly $49$ infinite families of normal forms of $p$-adic $4$-by-$4$ matrices with two degrees of freedom, exactly $35$ infinite families with one degree of freedom, and exactly $20$ isolated normal forms, up to congruence via a symplectic matrix.
		\item If $p\equiv 3\mod 4$, there are exactly $32$ infinite families of normal forms of $p$-adic $4$-by-$4$ matrices with two degrees of freedom, exactly $27$ infinite families with one degree of freedom, and exactly $20$ isolated normal forms, up to congruence via a symplectic matrix.
		\item If $p=2$, there are exactly $211$ infinite families of normal forms of $p$-adic $4$-by-$4$ matrices with two degrees of freedom, exactly $103$ infinite families with one degree of freedom, and exactly $72$ isolated normal forms, up to congruence via a symplectic matrix.
	\end{enumerate}
	
	In the three cases above, the infinite families with two degrees of freedom are given by
	\[\Big\{\Big\{\begin{pmatrix}
		r & 0 & 0 & 0 \\
		0 & c_1r & 0 & 0 \\
		0 & 0 & s & 0 \\
		0 & 0 & 0 & c_2s
	\end{pmatrix}:r,s\in\Qp\Big\}:c_1,c_2\in X_p\Big\}
	\cup\Big\{\Big\{\begin{pmatrix}
		0 & s & 0 & r \\
		s & 0 & cr & 0 \\
		0 & cr & 0 & s \\
		r & 0 & s & 0
	\end{pmatrix}:r,s\in\Qp\Big\}:c\in Y_p\Big\}\]
	\[\cup\Big\{\Big\{\renewcommand{\arraystretch}{2}
	\begin{pmatrix}
		\dfrac{ac(r-bs)}{c-b^2} & 0 & \dfrac{sc-rb}{c-b^2} & 0 \\
		0 & \dfrac{-r(t_1+bt_2)-s(bt_1+ct_2)}{a} & 0 & -r(bt_1+ct_2)-sc(t_1+bt_2) \\
		\dfrac{sc-rb}{c-b^2} & 0 & \dfrac{r-bs}{a(c-b^2)} & 0 \\
		0 & -r(bt_1+ct_2)-sc(t_1+bt_2) & 0 & ac(-r(t_1+bt_2)-s(bt_1+ct_2))
	\end{pmatrix}:\]
	\[r,s\in\Qp\Big\}:(a,b)\in\Big\{(1,0),(a_1,b_1)\Big\},c,t_1,t_2,a_1,b_1\text{ in one row of Table \ref{table:canonical}}\Big\},\]
	those with one degree of freedom are
	\[\Big\{\Big\{\begin{pmatrix}
		r & 0 & 0 & 0 \\
		0 & c_1r & 0 & 0 \\
		0 & 0 & s & 0 \\
		0 & 0 & 0 & 0
	\end{pmatrix}:r\in\Qp\Big\}:c_1\in X_p,s\in Y_p\cup\{1\}\Big\}
	\cup\Big\{\Big\{\begin{pmatrix}
		0 & r & 0 & 0 \\
		r & 0 & 1 & 0 \\
		0 & 1 & 0 & r \\
		0 & 0 & r & 0
	\end{pmatrix}:r\in\Qp\Big\}\Big\}\]
	\[\cup\Big\{\Big\{\begin{pmatrix}
		a & 0 & 0 & r \\
		0 & 0 & cr & 0 \\
		0 & cr & a & 0 \\
		r & 0 & 0 & 0
	\end{pmatrix}:r\in\Qp\Big\}:c\in Y_p,a\in\{1,h_p(c)\}\Big\},\]
	and the isolated forms are
	\[\Big\{\begin{pmatrix}
		r & 0 & 0 & 0 \\
		0 & 0 & 0 & 0 \\
		0 & 0 & s & 0 \\
		0 & 0 & 0 & 0
	\end{pmatrix}:r,s\in Y_p\cup\{1\}\Big\}
	\cup\Big\{\begin{pmatrix}
		c & 0 & 0 & 0 \\
		0 & 0 & 0 & -c \\
		0 & 0 & 0 & 0 \\
		0 & -c & 0 & 0
	\end{pmatrix}:c\in Y_p\cup\{1\}\Big\}.\]
	
	This is in contrast with the real case, where there are exactly $4$ infinite families with two degrees of freedom, exactly $7$ infinite families with one degree of freedom and exactly $6$ isolated normal forms. Here by ``infinite family'' we mean a family of normal forms of the form $r_1M_1+r_2M_2+\ldots+r_kM_k$, where $r_i$ are parameters in $\Qp$ and $k$ is the number of degrees of freedom, and by ``isolated'' we mean a form that is not part of any family.
\end{maintheorem}

\begin{remark}
	It is well known \cite{Serre-localfields} that all Galois extensions over $\Qp$ have a solvable Galois group, that is, the roots of all polynomials with $p$-adic coefficients can be expressed algebraically by extraction of successive radicals. However, since the general equation of degree $5$ is not solvable by radicals, there are no formulas, from degree $5$ onwards, to express the roots in terms of successive radicals. In other words, all our results are generalizable to any dimension with the same method we use here, but without explicit formulas for the local normal forms on symplectic manifolds of dimension $10$ or higher. These forms will be explicit for dimensions $6$ and $8$, though we do not carry this out because it includes hundreds and even thousands of possibilities for the local models even in dimension $6$.
\end{remark}

\begin{remark}
	Note that Theorems \ref{thm:num-forms2} and \ref{thm:num-forms1} refer to infinite families of matrices, but Theorem \ref{thm:num-integrable} does not mention families. This is because the equivalence between integrable systems we consider allows for a matrix $B$ to appear, as in equation \eqref{eq:integrable}.
\end{remark}

We can give a lower bound for the asymptotic behavior of the number of blocks of size $2n$ that appear in the normal forms of the matrices. In the real case, taking into account only the ``non-degenerate case'' where the eigenvalues of $\Omega_0^{-1}M$ are pairwise distinct, there are two infinite families of blocks with size two (each one with one degree of freedom), one family with size four (with two degrees of freedom) and no blocks with size greater than four: this is due to the fact that all irreducible polynomials in $\R$ have degree at most two, and has as a consequence that the number of families of normal forms of size $2n$ is quadratic in $n$. For the $p$-adic case, however, one has:

\begin{definition}\label{def:M}
	\letpprime\ and let $n$ be a positive integer. For each partition $P=(a_1,\ldots,a_k)$ of $n$, we define $M(P,p)\in\M_{2n}(\Qp)$ as the block-diagonal matrix whose blocks have sizes $(2a_1,\ldots,2a_k)$ and each block has a form which is itself block-diagonal, with blocks of the form
	\[\begin{pmatrix}
		1 & & & & & & & &\\
		& 0 & 1 & & & & & & \\
		& 1 & 0 & & & & & & \\
		& & & 0 & 1 & & & & \\
		& & & 1 & 0 & & & & \\
		& & & & & \ddots & & \\
		& & & & & & 0 & 1 & \\
		& & & & & & 1 & 0 & \\
		& & & & & & & & p
	\end{pmatrix}.\]
\end{definition}

Using the previous results we will prove the following asymptotic formulas:

\begin{maintheorem}[Asymptotic behavior of number of $(2n)$-by-$(2n)$ $p$-adic matrix normal forms]\label{thm:num-forms}
	\letpprime. Let $n$ be a positive integer. The number of $p$-adic families of non-degenerate normal forms of $(2n)$-by-$(2n)$ matrices up to congruence via a symplectic matrix, each family being of the form $r_1M_1+\ldots+r_nM_n$, where $r_i$ are parameters in $\Qp$, grows asymptotically at least with \[\frac{\e^{\pi\sqrt{2n/3}}}{4n\sqrt{3}}.\]
	Explicitly, if $P$ and $P'$ are distinct partitions of $n$ then the matrices $M(P,p)$ and $M(P',p)$ in Definition \ref{def:M} are not equivalent by multiplication by a symplectic matrix.
\end{maintheorem}

We will prove this result in \cite{CrePel-matrix} using the Hardy-Ramanujan formula \cite{HarRam}. We have computed explicit lower bounds for the number of families.

\begin{maintheorem}[Lower bounds of number of $(2n)$-by-$(2n)$ $p$-adic matrix normal forms]\label{thm:num-forms-lower-bound}
	Let $p\in\{2,3,5,7\}$. Let $n$ be a positive integer with $n\le 10$. The number of $p$-adic families of non-degenerate normal forms of $(2n)$-by-$(2n)$ matrices up to congruence via a symplectic matrix, each family being of the form $r_1M_1+\ldots+r_nM_n$, where $r_i$ are parameters in $\Qp$, and the number of $(2n)$-by-$(2n)$ blocks which may appear in those normal forms, is at least as follows:
	
	\upshape\footnotesize
	\begin{tabular}{|*{11}{c|}}\hline
		$2n$ & \multicolumn{2}{c|}{$\R$} & \multicolumn{2}{c|}{$\Q_2$} & \multicolumn{2}{c|}{$\Q_3$} & \multicolumn{2}{c|}{$\Q_5$} & \multicolumn{2}{c|}{$\Q_7$} \\ \cline{2-11}
		& blocks & forms & blocks & forms & blocks & forms & blocks & forms & blocks & forms \\ \hline
		2 & 2 & 2 & 11 & 11 & 5 & 5 & 7 & 7 & 5 & 5 \\ \hline
		4 & 1 & 4 & 145 & 211 & 17 & 32 & 21 & 49 & 17 & 32 \\ \hline
		6 & 0 & 6 & 2 & 1883 & 3 & 123 & 3 & 234 & 9 & 129 \\ \hline
		8 & 0 & 9 & 1 & 21179 & 2 & 495 & 4 & 1054 & 2 & 525 \\ \hline
		10 & 0 & 12 & 2 & 161343 & 3 & 1595 & 3 & 4021 & 3 & 1787 \\ \hline
		12 & 0 & 16 & 1 & 1374427 & 2 & 5111 & 4 & 14493 & 6 & 5874 \\ \hline
		14 & 0 & 20 & 2 & 9232171 & 3 & 14491 & 3 & 47462 & 3 & 17586 \\ \hline
		16 & 0 & 25 & 1 & 65570626 & 2 & 40244 & 4 & 148087 & 2 & 50614 \\ \hline
		18 & 0 & 30 & 2 & 397086458 & 3 & 103484 & 3 & 433330 & 9 & 137311 \\ \hline
		20 & 0 & 36 & 1 & 2469098766 & 2 & 259712 & 4 & 1217761 & 2 & 359463 \\ \hline
	\end{tabular}

	\normalsize\itshape
	For comparison, the table includes the exact number of forms in the real case in each dimension.
\end{maintheorem}

\section{Symplectic linear algebra over algebraically closed fields}\label{sec:algclosed}
In this section we prove several essential results about normal forms of matrices in an algebraically closed field. They are a crucial ingredient of our strategy to obtain general classifications of local normal forms of $p$-adic integrable systems, so we cannot postpone their proofs to \cite{CrePel-matrix} as we are doing with the majority of the results about $p$-adic matrices. As we will see, the problem in this case reduces to an equality of eigenvalues, and to an equality of normal forms if there are multiple eigenvalues.

\begin{definition}\label{def:symplectic}
	Let $n$ be a positive integer.
	\begin{itemize}
		\item Let $F$ be a field. A \emph{symplectic vector space} over $F$ is a pair $(V,\omega)$ where $V$ is a $2n$-dimensional vector space over $F$ and $\omega:V\times V\to F$ is a non-degenerate bilinear map such that $\omega(v,v)=0$ for all $v\in V$. We say that $\omega$ is a \emph{linear symplectic form.} If the characteristic of $F$ is not $2$, the last condition is equivalent to $\omega$ being antisymmetric: $\omega(v,w)=-\omega(w,v)$ for all $v,w\in V$.
		\item Let $F$ be a field. A \emph{linear symplectomorphism} between two symplectic vector spaces $(V_1,\omega_1)$ and $(V_2,\omega_2)$ over $F$ is a linear isomorphism $\phi:V_1\to V_2$ that preserves the symplectic form, that is, $\phi^*\omega_2=\omega_1$. In this case we say that $\omega_1$ and $\omega_2$ are \emph{linearly symplectomorphic}.
	\end{itemize}
\end{definition}

\begin{remark}
	By definition, an automorphism of a symplectic space is a linear symplectomorphism if and only if its matrix is a symplectic matrix.
\end{remark}

\begin{definition}
	Let $n$ be a positive integer. Let $F$ be a field with multiplicative identity element $1$. We define $\Omega_0$ as the $(2n)$-by-$(2n)$ matrix whose blocks are all
	\[\begin{pmatrix}
		0 & 1 \\
		-1 & 0
	\end{pmatrix},\]
	that is,
	\[\Omega_0=\begin{pmatrix}
		0 & 1 & & & & & \\
		-1 & 0 & & & & & \\
		& & 0 & 1 & & & \\
		& & -1 & 0 & & & \\
		& & & & \ddots & & \\
		& & & & & 0 & 1 \\
		& & & & & -1 & 0
	\end{pmatrix}.\]
	(It is also common to take $\Omega_0$ to have the blocks in the ``other'' diagonal.) $\Omega_0$ is called the \emph{standard symplectic form on $F^{2n}$} (which has the same expression as when $F=\R$ in Section \ref{sec:matrices-real}, keeping in mind that $1$ here represents an abstract multiplicative identity).
\end{definition}

We start with a result which finds a basis with good properties with respect to the symplectic form.

\begin{lemma}\label{lemma:eig}
	Let $n$ be a positive integer. Let $F$ be an algebraically closed field with characteristic different from $2$. Let $\Omega_0$ be the matrix of the standard symplectic form on $F^{2n}$ and let $M\in\M_{2n}(F)$ be a symmetric matrix such that the eigenvalues of $\Omega_0^{-1}M$ are pairwise distinct. Then there exists a basis $\{u_1,v_1,\ldots,u_n,v_n\}$ of $F^{2n}$ such that
	\begin{itemize}
		\item $u_i$ and $v_i$ are eigenvectors of $\Omega_0^{-1}M$ with eigenvalues of opposite sign, and
		\item the standard symplectic form is represented in this basis by a block-diagonal matrix with blocks of size two.
	\end{itemize}
\end{lemma}
\begin{proof}
	For ease of notation, we write the proof for $n=2$, but the proof is the same for any $n$.
	
	Let $A=\Omega_0^{-1}M$. We have that
	\begin{align*}
		\det(\lambda I-A) & =\det(\lambda I-\Omega_0^{-1}M) \\
		& =\det(\lambda I-(\Omega_0^{-1}M)^T) \\
		& =\det(\lambda I+M\Omega_0^{-1}) \\
		& =\det(\lambda \Omega_0+M)\det(\Omega_0^{-1}) \\
		& =\det(\lambda I+\Omega_0^{-1}M) \\
		& =\det(\lambda I+A).
	\end{align*}
	This implies that the eigenvalues of $A$ must come in pairs, that is, if $\lambda$ is an eigenvalue, $-\lambda$ also is. In particular, $0$ is not an eigenvalue, because it would be at least double, contradicting the hypothesis. So the eigenvalues are $\lambda,-\lambda,\mu,-\mu$, for $\lambda,\mu\in F^*$.
	
	Now let $w_1$ and $w_2$ be two eigenvectors of $A$ with eigenvalues $\alpha_1$ and $\alpha_2$, such that $\alpha_1\ne-\alpha_2$. Then
	\begin{align*}
		\alpha_1 w_1^T\Omega_0 w_2 & =(\Omega_0^{-1}Mw_1)^T\Omega_0 w_2 \\
		& =w_1^TM(-\Omega_0^{-1})\Omega_0 w_2 \\
		& =-w_1^TMw_2 \\
		& =-w_1^T\Omega_0\Omega_0^{-1}Mw_2 \\
		& =-\alpha_2w_1^T\Omega_0 w_2.
	\end{align*}
	As $\alpha_1\ne-\alpha_2$, this implies $w_1^T\Omega_0 w_2=0$.
	
	Let $u_1,v_1,u_2,v_2$ be the eigenvectors of $\lambda,-\lambda,\mu,-\mu$, respectively. By the previous result, $w_1^T\Omega_0 w_2=0$ for any two vectors $w_1,w_2\in\{u_1,v_1,u_2,v_2\}$ which are not a pair $u_i,v_i$. We call $\Psi$ the matrix with $(u_1,v_1,u_2,v_2)$ as columns. Then,
	\begin{align*}
		\Psi^T\Omega_0\Psi & =\begin{pmatrix}
			0 & u_1^T\Omega_0 v_1 & 0 & 0 \\
			v_1^T\Omega_0 u_1 & 0 & 0 & 0 \\
			0 & 0 & 0 & u_2^T\Omega_0 v_2 \\
			0 & 0 & v_2^T\Omega_0 u_2 & 0
		\end{pmatrix} \\
		& =\begin{pmatrix}
			0 & u_1^T\Omega_0 v_1 & 0 & 0 \\
			-u_1^T\Omega_0 v_1 & 0 & 0 & 0 \\
			0 & 0 & 0 & u_2^T\Omega_0 v_2 \\
			0 & 0 & -u_2^T\Omega_0 v_2 & 0
		\end{pmatrix},
	\end{align*}
	as we wanted.
\end{proof}

The basis in Lemma \ref{lemma:eig} is almost but not quite a symplectic basis:
\begin{definition}
	Let $n$ be a positive integer. Let $F$ be a field. We say that a basis $\{u_1,v_1,\ldots,u_n,v_n\}$ of $F^{2n}$ is \emph{symplectic} if, for any two vectors $w_1,w_2$ in the basis, $w_1^T\Omega_0w_2=1$ if $w_1=u_i$ and $w_2=v_i$ for some $i$ with $1\le i\le n$, and otherwise $w_1^T\Omega_0w_2=0$. This condition is equivalent to saying that the matrix in $\M_{2n}(F)$ with $u_1,v_1,\ldots,u_n,v_n$ as columns is symplectic.
\end{definition}

We can rescale the vectors $v_i$ in Lemma \ref{lemma:eig} such that the basis becomes symplectic. But this may break the structure of the eigenvectors: for example, if $F=\C$ and $-\lambda=\bar{\lambda}$, we can take $v_1=\bar{u}_1$, which will no more hold after rescaling $v_1$. We leave the lemma as such because we do not need that rescaling.

In the case where the eigenvalues are not all different, we can do something similar to Lemma \ref{lemma:eig}.

\begin{lemma}\label{lemma:eig2}
	Let $n$ be a positive integer. Let $F$ be an algebraically closed field with characteristic different from $2$ and let $M\in\M_{2n}(F)$ be a symmetric matrix. Let $\Omega_0$ be the matrix of the standard symplectic form on $F^{2n}$. Let $A=\Omega_0^{-1}M$. Then, the number of nonzero eigenvalues of $A$ is even, that is, $2m$ for some integer $m$ with $0\le m\le n$, and there exists a set \[\Big\{u_1,v_1,\ldots,u_m,v_m\Big\}\subset F^{2n}\] which satisfies the following properties:
	\begin{itemize}
		\item $Au_i=\lambda_iu_i+\mu_iu_{i-1}$ and $Av_i=-\lambda_iv_i+\mu_{i+1}v_{i+1}$ for $1\le i\le m$, where $\lambda_i\in F$, $\mu_i=0$ or $1$, $\mu_1=\mu_{m+1}=0$, and $\mu_i=1$ only if $\lambda_i=\lambda_{i-1}$. (That is to say, the vectors are a ``Jordan basis''.)
		\item The vectors can be completed to a symplectic basis: given $w_1,w_2$ in the set, $w_1^T\Omega_0w_2=1$ if $w_1=u_i,w_2=v_i$ for some $i$ with $1\le i\le m$, and otherwise $w_1^T\Omega_0w_2=0$.
	\end{itemize}
\end{lemma}

\begin{proof}
	Let $J$ be the Jordan form of $A$ and $\Psi$ such that $\Psi^{-1}A\Psi=J$. We have that
	\begin{align*}
		J & =\Psi^{-1}A\Psi \\
		& =\Psi^{-1}\Omega_0^{-1}M\Psi \\
		& =\Psi^{-1}\Omega_0^{-1}M\Omega_0^{-1}\Omega_0\Psi \\
		& =\Psi^{-1}\Omega_0^{-1}(-\Omega_0^{-1}M)^T\Omega_0\Psi \\
		& =\Psi^{-1}\Omega_0^{-1}(-\Psi J\Psi^{-1})^T\Omega_0\Psi \\
		& =(\Psi^T\Omega_0\Psi)^{-1}(-J^T)\Psi^T\Omega_0\Psi.
	\end{align*}
	But $J$ can only be similar to $-J^T$ if for each block having $\lambda$ in the diagonal there is another having $-\lambda$ in the diagonal, and with the same size. We can split the blocks in three parts, that is, there is a $\Phi$ such that
	\begin{equation}\label{eq:matrix}
		\Phi^{-1}A\Phi=\begin{pmatrix}
			J_+ & 0 & 0 \\
			0 & J_- & 0 \\
			0 & 0 & J_0
		\end{pmatrix},
	\end{equation}
	for some matrices $J_+,J_-$ and $J_0$ in Jordan form, such that $J_+$ and $J_-$ have eigenvalues with opposite sign and $J_0$ has only $0$ as eigenvalue. Let $m$ be the size of $J_+$ and $J_-$. Now $2m$ is the number of nonzero eigenvalues.
	
	Let the first $m$ columns of $\Phi$ be $u_1,\ldots,u_m$. Because of \eqref{eq:matrix}, we have
	\[A u_i=\lambda_i u_i+\mu_i u_{i-1},\]
	for adequate $\lambda_i$ and $\mu_i$, with $\mu_i=0$ or $1$.
	
	Now we change sign and transpose, getting
	\begin{equation}\label{eq:matrix2}
		-(\Phi^{-1}A\Phi)^T=\begin{pmatrix}
			-J_+^T & 0 & 0 \\
			0 & -J_-^T & 0 \\
			0 & 0 & -J_0^T
		\end{pmatrix}
	\end{equation}
	The left-hand side equals
	\begin{align*}
		-\Phi^TA^T(\Phi^{-1})^T & =\Phi^TM\Omega_0^{-1}(\Phi^{-1})^T \\
		& =\Phi^T\Omega_0A\Omega_0^{-1}(\Phi^{-1})^T \\
		& =\Phi^T\Omega_0A(\Phi^T\Omega_0)^{-1}.
	\end{align*}
	
	Again, let the first $m$ columns of $(\Phi^T\Omega_0)^{-1}$ be $v_1,\ldots,v_m$. By (\ref{eq:matrix2}), we have
	\[A v_i=-\lambda_i v_i-\mu_{i+1} v_{i+1}.\]
	
	It is left to prove that the $2m$ vectors form a partial symplectic basis. To do this, first we see that $u_i^T\Omega_0u_j=0$ for any $i,j$, by induction on $i+j$. The base case is when $i=j=1$, which is trivial. Supposing it true for $(i-1,j)$ and $(i,j-1)$, we prove it for $(i,j)$:
	\begin{align*}
		\lambda_i u_i^T\Omega_0 u_j & =(\Omega_0^{-1}Mu_i-\mu_i u_{i-1})^T\Omega_0 u_j \\
		& =u_i^TM(-\Omega_0^{-1})\Omega_0 u_j-\mu_i\cdot 0 \\
		& =-u_i^TMu_j \\
		& =-u_i^T\Omega_0\Omega_0^{-1}Mu_j \\
		& =-u_i^T\Omega_0(\lambda_j u_j+\mu_j u_{j-1}) \\
		& =-\lambda_j u_i^T\Omega_0 u_j
	\end{align*}
	which implies that $u_i^T\Omega_0 u_j=0$ because $\lambda_i\ne-\lambda_j$ (the opposites of the eigenvalues in the part $J_+$ are all in $J_-$). Analogously we prove that $v_i^T\Omega_0 v_j=0$, making the induction backwards.
	
	Finally, $u_i^T\Omega_0 v_j$ is the element in the position $(i,j)$ of $\Phi^T\Omega_0(\Phi^T\Omega_0)^{-1}$, which is $1$ if $i=j$ and $0$ otherwise, so the vectors can be completed to a symplectic basis.
\end{proof}

\section{Normal forms of $p$-adic singularities}\label{sec:singularities}

The Weierstrass-Williamson's classification of $p$-adic matrices given in Section \ref{sec:padic-williamson} can be used to classify critical points of $p$-adic analytic functions.

We refer to appendix \ref{sec:critical} for the definition of analytic function and critical point of a function on a $p$-adic manifold. It does not make sense to talk about the rank of such a critical point, because there is only one function and consequently only one differential form.

\begin{definition}[Non-degenerate critical point of $p$-adic analytic function, symplectic sense]\label{def:nondeg}
	Let $n$ be a positive integer. \letpprime. Let $(M,\omega)$ be a $2n$-dimensional $p$-adic analytic symplectic manifold. Let $f:M\to\Qp$ be a $p$-adic analytic function and let $m\in M$ be a critical point of $f$ (i.e. $\dd f(m)=0$). Let $\Omega$ be the matrix of $\omega$. We say that $m$ is \emph{non-degenerate} if the eigenvalues of $\Omega^{-1}\dd^2 f(m)$ are all distinct.
\end{definition}

This is not the usual notion of non-degenerate critical point (which states that the Hessian of $f$ is invertible), but the two are related as the following proposition shows.

\begin{proposition}\label{prop:nondeg}
	Let $n$ be a positive integer. \letpprime. Let $M$ be a $p$-adic analytic $2n$-dimensional manifold, $f:M\to \Qp$ a $p$-adic analytic function, and $m$ a critical point of $f$. Then the following are equivalent:
	\begin{enumerate}
		\item $m$ is a non-degenerate critical point of $f$ in the usual sense;
		\item There exists a linear symplectic form $\omega$ such that $m$ is a non-degenerate critical point of $f:(M,\omega)\to \Qp$ in the symplectic sense (Definition \ref{def:nondeg}).
		\item There exist infinitely many linear symplectic forms $\omega$ such that $m$ is a non-degenerate critical point of $f:(M,\omega)\to \Qp$ in the symplectic sense.
	\end{enumerate}
\end{proposition}

\begin{proof}
	Suppose (2) holds. Let $\Omega$ be the matrix of $\omega$. Then $\Omega^{-1}\dd^2 f(m)$ has all eigenvalues distinct. Applying Lemma \ref{lemma:eig} to the Hessian, if zero was an eigenvalue, it would be at least double, contradicting (2). So $\Omega^{-1}\dd^2 f(m)$ is invertible, which implies $\dd^2 f(m)$ is invertible and (1) holds.
	
	Now suppose (1) holds. Let $H=\dd^2 f(m)$. We first solve the problem for $H$ diagonal and then the general case.
	
	If $H$ is diagonal, (1) means that all diagonal elements are nonzero. Let $h_i$ be the $i$-th diagonal element. We will take $\Omega^{-1}$ with $a_i$ in the row $2i-1$ and column $2i$, $-a_i$ in the row $2i$ and column $2i-1$, and $0$ the rest. After multiplying $\Omega^{-1}$ by $H$, $a_i$ becomes $a_ih_{2i-1}$ and $-a_i$ becomes $-a_ih_{2i}$. The eigenvalues are \[\pm a_i\sqrt{-h_{2i-1}h_{2i}}\text{ for }1\le i\le n.\] Since $\Qp$ is infinite, it is always possible to choose $a_i$ so that these values are all different, independently of whether they are in $\Qp$ or not (just choose each $a_i$ in turn, and there will always be a possible value).
	
	In the general case, we can apply Gram-Schmidt orthogonalization to the canonical basis to obtain a basis $\{v_1,\ldots,v_n\}$ such that $v_i^THv_j\ne 0$ if and only if $i=j$. Taking these vectors as columns, we have a matrix $M$ such that $M^THM$ is diagonal. Applying the diagonal case to this matrix, we obtain an antisymmetric $\Omega$ such that $\Omega^{-1} M^THM$ has all eigenvalues different. This matrix is similar to $M\Omega^{-1}M^TH$, so this has also all eigenvalues different, and $M\Omega^{-1}M^T$ is the matrix we want.
	
	That (3) implies (2) is trivial. If (2) holds, $\Omega^{-1}M$ has all eigenvalues different and the same happens for any perturbation of $\Omega$, hence (3) holds.
\end{proof}

In the following when we speak of non-degenerate critical points of a function we always do it in the symplectic sense of Definition \ref{def:nondeg}.

In the real case, as a consequence of the Weierstrass-Williamson classification, it is always possible to choose local symplectic coordinates $(x_1,\xi_1,\ldots,x_n,\xi_n)$ with the origin at any $m\in M$ such that
\[f=\sum_{i=1}^nr_ig_i+\ocal(3)\]
for some $r_i\in\R$, where $g_i:V\to\R$ has one of the following forms: $(x_i^2+\xi_i^2)/2$ (\textit{elliptic component}), $x_i\xi_i$ (\textit{hyperbolic component}), or $x_i\xi_{i+1}-x_{i+1}\xi_i$ with the next function equal to $x_i\xi_i+x_{i+1}\xi_{i+1}$ (\textit{focus-focus component}). In the $p$-adic case we are also able to make this conclusion.

\begin{lemma}\label{lemma:hessian}
	Let $n$ be a positive integer. \letpprime. Let $M$ be a $p$-adic analytic manifold of dimension $2n$ and let $m$ be a non-degenerate critical point of a $p$-adic analytic function $f:M\to \Qp$. Then there exists an open neighborhood $U$ of $m$ and local coordinates $(x_1,\ldots,x_n)$ on $U$ with the origin at $m$ such that the restriction of $f$ to $U$, that is, $f|_U:U\to\Qp$ is given by a power series
	\[\sum_{I\in\N^n,i_1+\ldots+i_n\ge 2}a_Ix_1^{i_1}\ldots x_n^{i_n}.\]
	Moreover, the matrix in $\M_{2n}(\Qp)$ with the coefficient of $x_ix_j$ in the row $i$ and column $j$, for $i\ne j$, and the coefficient of $x_i^2$ multiplied by $2$ in the row and column $i$, is exactly the Hessian of $f$ at $m$ in the coordinates $x_1,\ldots,x_n$.
\end{lemma}

\begin{proof}
	By definition, $f$ is given by a power series converging in some open set $U$ which contains $m$. The center of the power series can be arbitrarily chosen in $U$, which means that we can choose $m$ as center. The degree $1$ terms are $0$ because $m$ is a critical point, and the degree $2$ terms are of the form $x^THx/2$ for some matrix $H\in\M_{2n}(\Qp)$. Differentiating twice, we get that the Hessian of $f$ is precisely $H$.
\end{proof}

\begin{corollary}[Normal form of critical points in dimension $2$]\label{cor:williamson}
	\letpprime. Let $(M,\omega)$ be a $p$-adic symplectic manifold of dimension $2$ and let $f:M\to \Qp$ be a $p$-adic analytic function. Let $m\in M$ be a critical point of $f$. Then there are local symplectic coordinates $(x,\xi)$ with the origin at $m$ such that $f-f(m)$ coincides with $r(x^2+c\xi^2)$ up to order $2$, for some $r\in\Qp$ and $c\in X_p$, or $r\in Y_p\cup\{1\}$ and $c=0$. Furthermore, if $f-f(m)$ has this form for two different local symplectic coordinates with the origin at $m$, then the two forms coincide.
\end{corollary}

\begin{proof}
	By Theorem \ref{thm:darboux}, there exists symplectic coordinates $(x,\xi)$ in a neighborhood of $m$ in $M$. Hence, we may assume that $M$ is a neighborhood of the origin in $(\Qp)^2$ and $\omega=\omega_0$.
	
	Applying Theorem \ref{thm:williamson} to $\dd^2 f$, we get a symplectic matrix $S$, which is the matrix of a linear symplectomorphism $\phi$, such that
	\[\dd^2(\phi^*f)=S^T\dd^2fS=\begin{pmatrix}
		r & 0 \\
		0 & cr
	\end{pmatrix}=\dd^2(rx^2+cr\xi^2).\]
	By Lemma \ref{lemma:hessian}, $f$ has the desired form, and applying $\phi$ to the symplectic form at every point of $M$ preserves the symplectic form $\omega_0$. Uniqueness follows from Theorem \ref{thm:williamson}.
\end{proof}

\begin{corollary}\label{cor:num-forms}
	\letpprime. Let $(M,\omega)$ be a $p$-adic analytic symplectic $2$-manifold. Let $X_p,Y_p$ be the non-residue sets in Definition \ref{def:sets}. Then the following statements hold.
	\begin{enumerate}
		\item If $p\equiv 1\mod 4$, there are exactly $7$ families of local normal forms for a non-degenerate critical point of a $p$-adic analytic function $f:(M,\omega)\to\Qp$ up to local symplectomorphisms centered at the critical point, where the normal forms in each family differ by multiplication by a constant $r$, and exactly $4$ normal forms for a degenerate critical point which only differ in the constant $r$.
		\item If $p\equiv 3\mod 4$, there are exactly $5$ families of local normal forms for a non-degenerate critical point of a $p$-adic analytic function $f:(M,\omega)\to\Qp$ up to local symplectomorphisms centered at the critical point, where the normal forms in each family differ by multiplication by a constant $r$, and exactly $4$ normal forms for a degenerate critical point which only differ in the constant $r$.
		\item If $p=2$, there are exactly $11$ families of local normal forms for a non-degenerate critical point of a $p$-adic analytic function $f:(M,\omega)\to\Qp$ up to local symplectomorphisms centered at the critical point, where the normal forms in each family differ by multiplication by a constant $r$, and exactly $8$ normal forms for a degenerate critical point which only differ in the constant $r$.
	\end{enumerate}
	
	In the three cases the above normal forms for a non-degenerate point are given by
	\[\Big\{\Big\{r(x^2+c\xi^2):r\in\Qp\Big\}:c\in X_p\Big\}\]
	and those for a degenerate point are given by
	\[\Big\{rx^2:r\in Y_p\cup\{1\}\Big\}.\]
\end{corollary}

\begin{proof}
	This follows from Theorem \ref{thm:num-forms2} and Corollary \ref{cor:williamson}.
\end{proof}

\begin{corollary}[Normal form of non-degenerate critical points in dimension $4$]\label{cor:williamson4}
	\letpprime. Let $X_p,Y_p,\mathcal{C}_i^k,\mathcal{D}_i^k$ be the non-residue sets and coefficient functions in Definition \ref{def:sets}. Let $(M,\omega)$ be a $p$-adic symplectic manifold of dimension $4$ and let $f:M\to \Qp$ be a $p$-adic analytic function. Let $m\in M$ be a non-degenerate critical point of $f$. Then there are local symplectic coordinates $(x,\xi,y,\eta)$ with the origin at $m$ such that in these coordinates we have: \[f-f(m)=rg_1+sg_2+\ocal(3),\] where $g_1$ and $g_2$ have one of the following forms:
	\begin{enumerate}
		\item $g_1(x,\xi,y,\eta)=x^2+c_1\xi^2$ and $g_2(x,\xi,y,\eta)=y^2+c_2\eta^2$, for $c_1,c_2\in X_p$.
		\item $g_1(x,\xi,y,\eta)=x\eta+cy\xi$ and $g_2(x,\xi,y,\eta)=x\xi+y\eta$, for $c\in Y_p$.
		\item
		\[g_k(x,\xi,y,\eta)=\sum_{i=0}^{2}\mathcal{C}_i^k(c,t_1,t_2,a,b)x^iy^{2-i}+\sum_{i=0}^{2}\mathcal{D}_i^k(c,t_1,t_2,a,b)\xi^i\eta^{2-i},\]
		for $k\in\{1,2\}$, where $c,t_1$ and $t_2$ correspond to one row of Table \ref{table:canonical} and $(a,b)$ are either $(1,0)$ or $(a_1,b_1)$ of the corresponding row.
	\end{enumerate}
	If $f-f(m)$ has this form for two different local symplectic coordinates, then the two forms coincide, except perhaps for swapping $g_1$ and $g_2$ at point (1).
\end{corollary}

\begin{proof}
	Analogously to the proof of Corollary \ref{cor:williamson}, we first apply Theorem \ref{thm:darboux} to transform a neighborhood of $m$ in $M$ to a neighborhood of the origin in $(\Qp)^4$ with the symplectic form $\omega_0$. Then we apply Theorem \ref{thm:williamson4} to the Hessian of $f$, and by Lemma \ref{lemma:hessian}, $f$ has the desired form. The three cases (1), (2) and (3) for the resulting matrix correspond to the three cases of this corollary, because
	\[\dd^2(\phi^*f)=S^T\dd^2 fS=r\dd^2 g_1+s\dd^2 g_2=\dd^2(rg_1+sg_2).\qedhere\]
\end{proof}

\begin{corollary}[Normal form of degenerate critical points in dimension $4$]\label{cor:williamson4-deg}
	\letpprime. let $X_p,Y_p$ be the non-residue sets in Definition \ref{def:sets}. Let $h_p:Y_p\to\Qp$ be the non-residue function in Definition \ref{def:h}. Let $(M,\omega)$ be a $p$-adic symplectic manifold of dimension $4$ and let $f:M\to \Qp$ be a $p$-adic analytic function. Let $m\in M$ be a degenerate critical point of $f$. Then there are local symplectic coordinates $(x,\xi,y,\eta)$ with the origin at $m$ such that $f-f(m)$ coincides with one of the following forms up to order $2$:
	\begin{enumerate}
		\item $r(x^2+c_1\xi^2)/2+s(y^2+c_2\eta^2)/2$, for some $c_1,c_2\in X_p\cup\{0\}$ and $r,s\in\Qp$. If $c_1=0$, $r$ can be taken in $Y_p\cup\{1\}$, and if $c_2=0$, $s$ can be taken in $Y_p\cup\{1\}$.
		\item $r(x\xi+y\eta)+y\xi$, for some $r\in\Qp$.
		\item $r(x\eta+cy\xi)+a(x^2+y^2)/2$, for some $r\in\Qp,c\in Y_p,a\in \{1,h_p(c)\}$.
		\item $c(x^2/2+\xi\eta)$, for some $c\in Y_p\cup\{1\}$.
	\end{enumerate}
	Furthermore, if $f-f(m)$ has this form for two different local symplectic coordinates, then the two forms coincide, except perhaps for swapping $(r,c_1)$ and $(s,c_2)$ at point (1).
\end{corollary}

\begin{proof}
	It is analogous the previous two proofs, but with Theorem \ref{thm:williamson4-deg}.
\end{proof}

\begin{corollary}\label{cor:num-williamson4}
	\letpprime. Let $(M,\omega)$ be a $p$-adic analytic symplectic $4$-manifold. Let $X_p,Y_p,\mathcal{C}_i^k,\mathcal{D}_i^k$ be the non-residue sets and coefficient functions in Definition \ref{def:sets}. Let $h_p:Y_p\to\Qp$ be the non-residue function in Definition \ref{def:h}. Then the following statements hold.
	\begin{enumerate}
		\item If $p\equiv 1\mod 4$, there are exactly $49$ infinite families of local normal forms with two degrees of freedom for a critical point of a $p$-adic analytic function on a $4$-dimensional $p$-adic symplectic manifold $f:(M,\omega)\to\Qp$ up to local symplectomorphisms centered at the critical point, exactly $35$ infinite families with one degree of freedom, and exactly $20$ isolated normal forms.
		\item If $p\equiv 3\mod 4$, there are exactly $32$ infinite families of local normal forms with two degrees of freedom for a critical point of a $p$-adic analytic function on a $4$-dimensional $p$-adic symplectic manifold $f:(M,\omega)\to\Qp$ up to local symplectomorphisms centered at the critical point, exactly $27$ infinite families with one degree of freedom, and exactly $20$ isolated normal forms.
		\item If $p=2$, there are exactly $211$ infinite families of local normal forms with two degrees of freedom for a critical point of a $p$-adic analytic function on a $4$-dimensional $p$-adic symplectic manifold $f:(M,\omega)\to\Qp$ up to local symplectomorphisms centered at the critical point, exactly $103$ infinite families with one degree of freedom, and exactly $72$ isolated normal forms.
	\end{enumerate}
	
	In the three cases above, the infinite families with two degrees of freedom are given as
	\[\Big\{\Big\{r(x^2+c_1\xi^2)+s(y^2+c_2\eta^2):r,s\in\Qp\Big\}:c_1,c_2\in X_p\Big\}\]
	\[\cup\Big\{\Big\{r(x\eta+cy\xi)+s(x\xi+y\eta):r,s\in\Qp\Big\}:c\in Y_p\Big\}\]
	\[\cup\Big\{\Big\{r\left(\sum_{i=0}^{2}\mathcal{C}_i^1(c,t_1,t_2,a,b)x^iy^{2-i}+\sum_{i=0}^{2}\mathcal{D}_i^1(c,t_1,t_2,a,b)\xi^i\eta^{2-i}\right)\]\[+s\left(\sum_{i=0}^{2}\mathcal{C}_i^2(c,t_1,t_2,a,b)x^iy^{2-i}+\sum_{i=0}^{2}\mathcal{D}_i^2(c,t_1,t_2,a,b)\xi^i\eta^{2-i}\right):\]\[r,s\in\Qp\Big\}:(a,b)\in\Big\{(1,0),(a_1,b_1)\Big\},c,t_1,t_2,a_1,b_1\text{ in one row of Table \ref{table:canonical}}\Big\},\]
	those with one degree of freedom are
	\[\Big\{\Big\{r(x^2+c_1\xi^2)+sy^2/2:r\in\Qp\Big\}:c_1\in X_p,s\in Y_p\cup\{1\}\Big\}
	\cup\{\{r(x\xi+y\eta)+y\xi:r\in\Qp\}\}\]
	\[\cup\Big\{\Big\{r(x\eta+cy\xi)+a(x^2+y^2)/2:r\in\Qp\Big\}:c\in Y_p,a\in\{1,h_p(c)\}\Big\},\]
	and the isolated forms are
	\[\Big\{(rx^2+sy^2)/2:r,s\in Y_p\cup\{1\}\Big\}
	\cup\Big\{c(x^2/2+\xi\eta):c\in Y_p\cup\{1\}\Big\}.\]
	Here by ``infinite family'' we mean a family of normal forms of the form $r_1f_1+r_2f_2+\ldots+r_kf_k$, where $r_i$ are parameters and $k$ is the number of degrees of freedom, and by ``isolated'' we mean a form that is not part of any family.
\end{corollary}

\section{Normal forms of singularities of integrable systems}\label{sec:integrable}

The real Weierstrass-Williamson classification (Appendix \ref{sec:real}) is one of the foundational results used in the symplectic theory of integrable systems (in particular in Eliasson's linearization theorems \cite{Eliasson-thesis,Eliasson}). A consequence of the Weierstrass-Williamson classification states that, given an integrable system $F=(f_1,\ldots,f_n):(M,\omega)\to\R^n$ and a \textit{non-degenerate} critical point $m$ of $F$ (in a precise sense which we will define shortly), it is always possible to choose local symplectic coordinates $(x_1,\xi_1,\ldots,x_n,\xi_n)$ with the origin at $m$ such that in these coordinates
\[B\circ(F-F(m))=(g_1,\ldots,g_n)+\ocal(3),\]
where $B$ is a $n$-by-$n$ matrix of reals and each $g_i,i\in\{1,\ldots,n\}$ has one of the following forms: $\xi_i$ (\emph{regular component}), $(x_i^2+\xi_i^2)/2$ (\emph{elliptic component}), $x_i\xi_i$ (\emph{hyperbolic component}), or $x_i\xi_{i+1}-x_{i+1}\xi_i$ with the next function equal to $x_i\xi_i+x_{i+1}\xi_{i+1}$ (\emph{focus-focus component}). See Figure \ref{fig:normal-forms} for a representation.

\subsection{Non-degenerate critical points of integrable systems}

As we see next, the classification theorems for critical points of functions on symplectic manifolds can be applied to classify critical points of integrable systems. In order to do this, first we recall the notion of non-degeneracy for a critical point of an integrable system $F:(M,\omega)\to(\Qp)^n$ on a $p$-adic analytic symplectic manifold which we use in the paper, and which in the real case is equivalent to the usual definition in Vey's paper \cite{Vey}, see for example \cite[Section 4.2.1]{PelVuN-integrable} and \cite[Lemma 2.5]{DulPel}.

\begin{definition}\label{def:critical}
	Let $n$ be a positive integer. \letpprime. Let $(M,\omega)$ be a $p$-adic analytic symplectic manifold of dimension $2n$. Let $F=(f_1,\ldots,f_n):(M,\omega)\to(\Qp)^n$ be a $p$-adic analytic integrable system. A point $m\in M$ is a \emph{critical point} of $m$ if the $1$-forms $\dd f_1(m),\ldots,\dd f_n(m)$ are linearly dependent. The number of linearly independent forms among $\dd f_1(m),\ldots,\dd f_n(m)$ is called the \emph{rank} of the critical point.
\end{definition}

In the following definition, a subspace $U$ of a symplectic vector space $(V,\omega)$ is said to be \emph{isotropic} if $\omega(u,v)=0$ for any $u,v\in U$.

\begin{definition}[Non-degenerate critical point of $p$-adic analytic integrable system] \label{def:nondeg-integrable}
	Let $n$ be a positive integer. \letpprime. Let $(M,\omega)$ be a $p$-adic analytic symplectic manifold. Let $\Omega$ be the matrix of the linear symplectic form $\omega_m$ on the vector space $\mathrm{T}_mM$. A rank $0$ critical point $m$ of a $p$-adic analytic integrable system $F=(f_1,\ldots,f_n):(M,\omega)\to(\Qp)^n$ is \emph{non-degenerate} if the Hessians evaluated at $m$: \[\dd^2 f_1(m),\ldots,\dd^2 f_n(m)\] are linearly independent and if there exist $a_1,\ldots,a_n\in\Qp$ such that the matrix $$\Omega^{-1}\sum_{i=1}^{n}a_i\dd^2 f_i(m)$$ has $n$ different eigenvalues.
	
	If $m$ has rank $r$, then the vectors $X_{f_1}(m),\ldots,X_{f_n}(m)$ obtained by evaluating the Hamiltonian vector fields $X_{f_1},\ldots,X_{f_n}$ of $f_1,\ldots,f_n$ at $m$, form an isotropic linear subspace $L$ of $\mathrm{T}_mM$, whose dimension is $r$; suppose that $X_{f_1}(m),\ldots,X_{f_r}(m)$ are linearly independent. Then $\dd f_{r+1},\ldots,\dd f_n$ descend to $L^{\omega}/L$ in such a way that the origin is a rank $0$ critical point of the integrable system induced by $F$ on $L^{\omega}/L$. We say that the point is \emph{non-degenerate} if the origin is a non-degenerate critical point of this induced integrable system.
\end{definition}

\begin{remark}
	Definition \ref{def:nondeg-integrable} in the $p$-adic case is motivated by the fact that in the real case the notion of being non-degenerate for a critical point on an integrable system can also be defined in this way. Indeed, the usual definition is given in terms of Cartan subalgebras as follows: if $(M,\omega)$ is a real symplectic manifold of dimension $2n$ and $F=(f_1,\ldots,f_n):M \to \mathbb{R}^n$ is an integrable system, a critical point $m$ of $F$ of rank $0$ is \emph{non-degenerate} if the Hessians $\dd^2 f_1(m),\ldots,\dd^2 f_n(m)$ span a Cartan subalgebra of the symplectic Lie algebra of quadratic forms on the tangent space $(\mathrm{T}_mM,\omega_m)$.
\end{remark}

A $p$-adic analytic integrable system $(f_1,\ldots,f_n):(M,\omega)\to(\Qp)^n$ is \emph{non-degenerate} if all of its critical points are non-degenerate.

\subsection{Proof of Theorems \ref{thm:integrable}, \ref{thm:num-integrable} and \ref{thm:num-integrable2}}
By Theorem \ref{thm:darboux}, we may assume without loss of generality that $M$ is a neighborhood of the origin in $(\Qp)^4$ and $\omega=\omega_0$.

For a rank $1$ critical point, we can take $\dd\eta$ in the direction of the nonzero differential, and the problem reduces to classify the critical point in $L^{\omega_0}/L$. This is a system with one function (the linear combination of $f_1$ and $f_2$ with differential $0$) in dimension $2$, so we apply Corollary \ref{cor:williamson} and get the result.

For a rank $0$ critical point, first we prove existence. The fact that $f_1$ and $f_2$ form an integrable system implies that $\{f_1,f_2\}=0$, that is, \[(\dd f_1)^T\Omega_0^{-1}\dd f_2=0.\] Differentiating this twice, evaluating at $m$ and using that $\dd f_1(m)=\dd f_2(m)=0$,
\[\dd^2 f_1(m)\Omega_0^{-1}\dd^2 f_2(m)=\dd^2 f_2(m)\Omega_0^{-1}\dd^2 f_1(m)\]
We define $A_i=\Omega_0^{-1}\dd^2 f_i(m)$. The previous expression implies that $A_1$ and $A_2$ commute.

Let $u\in(\Cp)^4$ be an eigenvector of $A_1$, where $\Cp$ is the field of complex $p$-adic numbers. Then
\[A_1A_2u=A_2A_1u=\lambda A_2 u\]
for $\lambda\in\Cp$. This implies that $A_2u$ is also an eigenvector of $A_1$ with value $\lambda$. But the critical point is non-degenerate, which means that the only eigenvector with value $\lambda$ is $u$. Hence, $A_2u=\mu u$ for some $\mu\in\Cp$, and $u$ is also eigenvector of $A_2$. So $A_1$ and $A_2$ have the same eigenvectors.

We apply Corollary \ref{cor:williamson4} to $f_1$ and $f_2$ in order to find local symplectic coordinates which reduce them to a normal form. By Theorem \ref{thm:williamson4}, the matrix $S$ and the family of normal forms of each Hessian, hence the normal forms of $f_1$ and $f_2$, are determined by the eigenvectors of $A_1$ and $A_2$, which coincide. This means that the local symplectic coordinates and the normal forms are the same for $f_1$ than for $f_2$. Let these forms be $r_1g_1+s_1g_2$ and $r_2g_1+s_2g_2$, for some $r_1,r_2,s_1,s_2\in\Qp$ and $g_1$ and $g_2$ are among the possibilities of Corollary \ref{cor:williamson4}. Since the Hessians are linearly independent, the matrix
\[\begin{pmatrix}
	r_1 & s_1 \\
	r_2 & s_2
\end{pmatrix}\]
changing $(g_1,g_2)$ to $(f_1,f_2)$ can be inverted, giving the matrix $B$ that we need, and the proof of existence of Theorem \ref{thm:integrable} is complete. Uniqueness follows directly from Corollary \ref{cor:williamson4}.

Theorems \ref{thm:num-integrable} and \ref{thm:num-integrable2} follow from applying Theorems \ref{thm:num-forms1} and \ref{thm:num-forms}, respectively, to the Hessians of the components of the system: each normal form of the matrix gives a normal form of the integrable system, and the Hessian of $f_{P,p}$ is exactly the matrix $M(P,p)$.

\subsection{Degenerate critical points of integrable systems}\label{sec:deg-real}

Based on Theorem \ref{thm:williamson4-deg}, a version of Theorem \ref{thm:integrable} can also be deduced for \emph{degenerate} singularities (a topic of growing interest in real symplectic geometry; see for instance \cite{DulPel,EHS,HohPal} and the references therein).

In the real case, very little is known about degenerate singularities. In the literature only some results are available for some special kinds of degenerate points, see for instance \cite{Zung-2000,Zung-2002}. We give here the existence part of the classification for real integrable systems (which must be smooth but need not be analytic). We do not know how to obtain the uniqueness part, that is, we do not know which models are equivalent to each other.

\begin{theorem}\label{thm:real-degenerate}
	Let $(M,\omega)$ be a $4$-dimensional symplectic manifold. Let $F=(f_1,f_2):(M,\omega)\to \R^2$ be an integrable system and $m$ a degenerate critical point of $F$. Then there exist local symplectic coordinates $(x,\xi,y,\eta)$ with the origin at $m$ and an invertible matrix $B\in\M_2(\R)$ such that in these coordinates we have:
	\[
	B\circ(F-F(m))=(g_1,g_2)+\ocal(3),
	\]
	where the expression of $(g_1,g_2)$ depends on the rank of $m\in\{0,1\}$. If the rank is $1$, $g_1(x,\xi,y,\eta)\in\{0,x^2\}$ and $g_2(x,\xi,y,\eta)=\eta$. If the rank is $0$, one of the following expressions holds:
	\begin{enumerate}
		\item[(N0)] $g_1(x,\xi,y,\eta)=0,g_2(x,\xi,y,\eta)=0$.
		\item[(N1)] $g_1(x,\xi,y,\eta)=x^2/2,g_2(x,\xi,y,\eta)=0$.
		\item[(N2$(c)$)] $g_1(x,\xi,y,\eta)=(x^2+cy^2)/2,g_2(x,\xi,y,\eta)=0$, for $c\in\{-1,1\}$.
		\item[(N3)] $g_1(x,\xi,y,\eta)=y\xi,g_2(x,\xi,y,\eta)=0$.
		\item[(N4)] $g_1(x,\xi,y,\eta)=x^2/2+y\xi,g_2(x,\xi,y,\eta)=0$.
		\item[(N5)] $g_1(x,\xi,y,\eta)=x^2/2,g_2(x,\xi,y,\eta)=\eta^2/2$.
		\item[(N6$(a)$)] $g_1(x,\xi,y,\eta)=x^2/2,g_2(x,\xi,y,\eta)=xy+a\eta^2/2$, for $a\in\{-1,0,1\}$.
		\item[(N7$(a)$)] $g_1(x,\xi,y,\eta)=(x^2-y^2)/2,g_2(x,\xi,y,\eta)=xy+ay^2/2$, for $a\in\{0,2\}$.
		\item[(S1$(c)$)] $g_1(x,\xi,y,\eta)=(x^2+c\xi^2)/2,g_2(x,\xi,y,\eta)=0$, for $c\in\{-1,1\}$.
		\item[(S2$(c,a)$)] $g_1(x,\xi,y,\eta)=(x^2+c\xi^2)/2+ay^2,g_2(x,\xi,y,\eta)=0$, for $c\in\{-1,1\},a\in\{-1,1\}$.
		\item[(S3$(c)$)] $g_1(x,\xi,y,\eta)=(x^2+c\xi^2)/2,g_2(x,\xi,y,\eta)=y^2$, for $c\in\{-1,1\}$.
		\item[(R1)] $g_1(x,\xi,y,\eta)=x\xi+y\eta,g_2(x,\xi,y,\eta)=0$.
		\item[(R2)] $g_1(x,\xi,y,\eta)=x\xi+y\eta+y\xi,g_2(x,\xi,y,\eta)=0$.
		\item[(R3)] $g_1(x,\xi,y,\eta)=x\xi+y\eta,g_2(x,\xi,y,\eta)=y\xi$.
		\item[(I1)] $g_1(x,\xi,y,\eta)=(x^2+\xi^2+y^2+\eta^2)/2,g_2(x,\xi,y,\eta)=0$.
		\item[(I2)] $g_1(x,\xi,y,\eta)=x\eta-y\xi,g_2(x,\xi,y,\eta)=0$.
		\item[(I3)] $g_1(x,\xi,y,\eta)=x\eta-y\xi,g_2(x,\xi,y,\eta)=(x^2+y^2)/2$.
	\end{enumerate}
\end{theorem}

For the $p$-adic case, we need the following definition.

\begin{definition}\label{def:degenerate-set}
	\letpprime. Let $d_0\in\Z$ be such that $d_0^2+1$ is the smallest quadratic non-residue. Let $Y_p$ be the non-residue set in Definition \ref{def:sets}. We define the \emph{degenerate residue set} $Z_{p,c}$ as follows.
	\[\begin{cases}
		Z_{p,1}=\{2d_0,2\ii,2\ii+p,2\ii-p\},Z_{p,c_0}=Z_{p,p}=Z_{p,c_0p}=\{0\} & \text{if }p\equiv 1\mod 4; \\
		Z_{p,1}=\{2d_0\},Z_{p,-1}=\{0,2,2+p,2-p\},Z_{p,p}=Z_{p,-p}=\{0\} & \text{if }p\equiv 3\mod 4; \\
		Z_{p,1}=\{1,2,6\},Z_{p,-1}=\{0,1,2,4,6,10,22,26\},Z_{p,3}=\{0,1,4\}, & \\
		Z_{p,-3}=\{0,2,6\},Z_{p,2}=Z_{p,-2}=Z_{p,6}=Z_{p,-6}=\{0,2,4\} & \text{if }p=2.
	\end{cases}\]
\end{definition}

The choice of the smallest quadratic non-residue in Definition \ref{def:degenerate-set} is not essential, in analogy with Definition \ref{def:sets}.

\begin{theorem}\label{thm:padic-degenerate}
	\letpprime. Let $Y_p$ be the non-residue set in Definition \ref{def:sets} and for any $c\in Y_p\cup\{1\}$ let $Z_{p,c}$ be the degenerate residue set in Definition \ref{def:degenerate-set}. Let $(M,\omega)$ be a $4$-dimensional $p$-adic analytic symplectic manifold. Let $F=(f_1,f_2):(M,\omega)\to \Qp^2$ be a $p$-adic analytic integrable system and $m$ a degenerate critical point of $F$. Then there exist local symplectic coordinates $(x,\xi,y,\eta)$ with the origin at $m$ and an invertible matrix $B\in\M_2(\Qp)$ such that in these coordinates we have:
	\[
	B\circ(F-F(m))=(g_1,g_2)+\ocal(3),
	\]
	where the expression of $(g_1,g_2)$ depends on the rank of $m\in\{0,1\}$. If the rank is $1$, $g_1(x,\xi,y,\eta)\in\{0,x^2\}$ and $g_2(x,\xi,y,\eta)=\eta$. If the rank is $0$, one of the following expressions holds:
	\begin{enumerate}
		\item[(N0)] $g_1(x,\xi,y,\eta)=0,g_2(x,\xi,y,\eta)=0$.
		\item[(N1)] $g_1(x,\xi,y,\eta)=x^2/2,g_2(x,\xi,y,\eta)=0$.
		\item[(N2$(c)$)] $g_1(x,\xi,y,\eta)=(x^2+cy^2)/2,g_2(x,\xi,y,\eta)=0$, for $c\in Y_p\cup\{1\}$.
		\item[(N3)] $g_1(x,\xi,y,\eta)=y\xi,g_2(x,\xi,y,\eta)=0$.
		\item[(N4)] $g_1(x,\xi,y,\eta)=x^2/2+y\xi,g_2(x,\xi,y,\eta)=0$.
		\item[(N5)] $g_1(x,\xi,y,\eta)=x^2/2,g_2(x,\xi,y,\eta)=\eta^2/2$.
		\item[(N6$(a)$)] $g_1(x,\xi,y,\eta)=x^2/2,g_2(x,\xi,y,\eta)=xy+a\eta^2/2$, for $a\in Y_p\cup\{0,1\}$.
		\item[(N7$(c,a)$)] $g_1(x,\xi,y,\eta)=(x^2+cy^2)/2,g_2(x,\xi,y,\eta)=xy+ay^2/2$, for $c\in Y_p\cup\{1\}, a\in Z_{p,c}$.
		\item[(S1$(c)$)] $g_1(x,\xi,y,\eta)=(x^2+c\xi^2)/2,g_2(x,\xi,y,\eta)=0$, for $c\in X_p$.
		\item[(S2$(c,a)$)] $g_1(x,\xi,y,\eta)=(x^2+c\xi^2)/2+ay^2,g_2(x,\xi,y,\eta)=0$, for $c\in X_p,a\in\{-1,1\}$.
		\item[(S3$(c)$)] $g_1(x,\xi,y,\eta)=(x^2+c\xi^2)/2,g_2(x,\xi,y,\eta)=y^2$, for $c\in X_p$.
		\item[(R1)] $g_1(x,\xi,y,\eta)=x\xi+y\eta,g_2(x,\xi,y,\eta)=0$.
		\item[(R2)] $g_1(x,\xi,y,\eta)=x\xi+y\eta+y\xi,g_2(x,\xi,y,\eta)=0$.
		\item[(R3)] $g_1(x,\xi,y,\eta)=x\xi+y\eta,g_2(x,\xi,y,\eta)=y\xi$.
		\item[(I1$(c)$)] $g_1(x,\xi,y,\eta)=(x^2-c\xi^2+y^2-c\eta^2)/2,g_2(x,\xi,y,\eta)=0$, for $c\in Y_p$.
		\item[(I2$(c)$)] $g_1(x,\xi,y,\eta)=x\eta+cy\xi+(1+c)(x^2+cy^2)/4c,g_2(x,\xi,y,\eta)=0$, for $c\in Y_p$.
		\item[(I3$(c)$)] $g_1(x,\xi,y,\eta)=x\eta+cy\xi+(1+c)y^2/2,g_2(x,\xi,y,\eta)=(x^2-cy^2)/2$, for $c\in Y_p$.
	\end{enumerate}
\end{theorem}

\begin{proof}[Proof of Theorems \ref{thm:real-degenerate} and \ref{thm:padic-degenerate}]
	Let $K$ be either $\R$ or $\Qp$. By Theorem \ref{thm:darboux}, the same as we did in the proof of Corollary \ref{cor:williamson}, we may assume that the manifold is $(K^4,\omega_0)$. Let $F=(f_1,f_2):(K^4,\omega_0)\to K^2$ be an integrable system with a degenerate critical point at the origin.
	
	The case of rank $1$ points is easy to solve: one component of the normal form must be regular, so we take $f_2=\eta$. The other component $f_1$ can be put in terms of $(x,\xi)$ and in these variables it has a degenerate critical point at the origin, which we know how to classify by Corollary \ref{cor:williamson}: degenerate points have the form $rx^2$ for some $r\in\{-1,0,1\}$ in the real case and $r\in Y_p\cup\{0,1\}$ in the $p$-adic case. This gives two forms up to multiplication by constants: $(0,\eta)$ and $(x^2,\eta)$.
	
	For the rank $0$ points, we start by making a classification of the matrices in three types.
	
	\begin{definition}
		\letpprime. Let $K$ be $\R$ or $\Qp$.A matrix $M\in\M_4(K)$ is \emph{nilpotent} if the four eigenvalues of $M$ are zero, \emph{semi-nilpotent} if only two of them are zero, and \emph{non-nilpotent} if no eigenvalue is zero.
		
		A $2$-dimensional vector subspace $V\subset\M_4(K)$ is \emph{nilpotent} if all matrices in $V$ are nilpotent, \emph{semi-nilpotent} if they are at least semi-nilpotent, and \emph{non-nilpotent} otherwise.
	\end{definition}
	
	Let $M_i$ be the Hessian of $f_i$, for $i\in\{1,2\}$, and let $A_i=\Omega_0^{-1}M_i$. We will reduce the matrices to their normal forms in several steps. It is important that we are allowed to multiply the functions $(f_1,f_2)$ by a matrix $B$, that is, they can be replaced by any linear combination. In terms of the matrices $A_i$, this means that their exact forms are not as relevant for the classification as the linear space that they generate. Let \[\Span(A_1,A_2)\subset\M_4(K)\] be this space. We distinguish three cases, depending on the type of $\Span(A_1,A_2)$.
	
	\begin{enumerate}
		\item \emph{Nilpotent case.}
		
		In this case both $A_1$ and $A_2$ are nilpotent, as well as all their linear combinations. If $A_1=A_2=0$, it is case (N0); otherwise, we suppose that $A_1\ne 0$.
		
		By Theorem \ref{thm:williamson4-deg}, $A_1$ can be brought to one of the forms
		\[\begin{pmatrix}
			0 & 0 & 0 & 0 \\
			1 & 0 & 0 & 0 \\
			0 & 0 & 0 & 0 \\
			0 & 0 & 0 & 0
		\end{pmatrix},
		\begin{pmatrix}
			0 & 0 & 0 & 0 \\
			1 & 0 & 0 & 0 \\
			0 & 0 & 0 & 0 \\
			0 & 0 & r & 0
		\end{pmatrix},
		\begin{pmatrix}
			0 & 0 & 1 & 0 \\
			0 & 0 & 0 & 0 \\
			0 & 0 & 0 & 0 \\
			0 & -1 & 0 & 0
		\end{pmatrix},
		\begin{pmatrix}
			0 & 0 & 0 & 1 \\
			1 & 0 & 0 & 0 \\
			0 & 1 & 0 & 0 \\
			0 & 0 & 0 & 0
		\end{pmatrix}\]
		by a symplectic matrix in $\M_4(K)$, where $r\in Y_p\cup\{1\}$ in the $p$-adic case and $r\in\{1,-1\}$ in the real case. We tackle each form separately. In what follows, we use $a_{ij}$ for the element in the $i$-th row and $j$-th column of $A_2$.
		
		\begin{enumerate}
			\item \emph{First form.}
			
			Applying that $A_1A_2=A_2A_1$ and that $M_2=\Omega_0A_2$ is symmetric, we have that
			\[A_2=\begin{pmatrix}
				0 & 0 & 0 & 0 \\
				a_{21} & 0 & a_{23} & a_{24} \\
				-a_{24} & 0 & a_{33} & a_{34} \\
				a_{23} & 0 & a_{43} & -a_{33}
			\end{pmatrix}.\]
			Since $A_2$ is nilpotent, the lower right block must have the two eigenvalues zero, that is, \[a_{33}^2+a_{34}a_{43}=0.\] This means that, if $a_{33}\ne 0$, after multiplying by the symplectic matrix
			\[S_1=\begin{pmatrix}
				1 & 0 & 0 & 0 \\
				0 & 1 & 0 & 0 \\
				0 & 0 & 1 & 0 \\
				0 & 0 & \frac{a_{43}}{a_{33}} & 1
			\end{pmatrix},\]
			we have $S_1^{-1}A_1S_1=A_1$ and
			\[B_2=S_1^{-1}A_2S_1=\begin{pmatrix}
				0 & 0 & 0 & 0 \\
				b_{21} & 0 & b_{23} & b_{24} \\
				-b_{24} & 0 & 0 & b_{34} \\
				b_{23} & 0 & 0 & 0
			\end{pmatrix}.\]
			If $a_{33}=0$, we can also achieve this form by exchanging, if needed, the third and fourth coordinates and changing the sign of one of them to preserve $\Omega_0$. Note that $b_{21}$ does not affect the normal form, because it can be eliminated adding a multiple of $A_1$.
			
			If $b_{23}=b_{24}=b_{34}=0$, this is form (N1). If $b_{23}=b_{24}=0$ and $b_{34}\ne 0$, it is (N5). If $b_{23}=0$ and $b_{24}\ne 0$, after exchanging the coordinates we get
			\[C_2=\begin{pmatrix}
				0 & 0 & 0 & 0 \\
				c_{21} & 0 & c_{23} & 0 \\
				0 & 0 & 0 & 0 \\
				c_{23} & 0 & c_{43} & 0
			\end{pmatrix}.\]
			If $c_{43}=0$, we are in form (N6(0)). If $c_{43}\ne 0$, we multiply by a constant to make $c_{43}=1$ and then by a symplectic matrix
			\[\begin{pmatrix}
				1 & 0 & 0 & 0 \\
				0 & 1 & 0 & c_{23} \\
				-c_{23} & 0 & 1 & 0 \\
				0 & 0 & 0 & 1
			\end{pmatrix}\]
			to get to form (N5).
			
			If $b_{23}\ne 0$, we multiply $B_2$ by a constant to make $b_{23}=1$ and then use
			\[S_2=\begin{pmatrix}
				1 & 0 & 0 & 0 \\
				0 & 1 & 0 & 0 \\
				0 & 0 & 1 & -b_{24} \\
				0 & 0 & 0 & 1
			\end{pmatrix}\]
			to make $S_2^{-1}A_1S_2=A_1$ and
			\[D_2=S_2^{-1}B_2S_2=\begin{pmatrix}
				0 & 0 & 0 & 0 \\
				d_{21} & 0 & 1 & 0 \\
				0 & 0 & 0 & d_{34} \\
				1 & 0 & 0 & 0
			\end{pmatrix}.\]
			
			We can now make $d_{34}\in\{-1,0,1\}$ if $K=\R$ and $d_{34}\in Y_p\cup\{0,1\}$ if $K=\Qp$ by multiplying by a symplectic matrix of the form
			\[\begin{pmatrix}
				1 & 0 & 0 & 0 \\
				0 & 1 & 0 & 0 \\
				0 & 0 & k & 0 \\
				0 & 0 & 0 & \frac{1}{k}
			\end{pmatrix},\]
			and we are in form (N6$(d_{34})$).
			
			\item \emph{Second form.}
			
			Applying again that $A_1A_2=A_2A_1$ and that $M_2=\Omega_0A_2$ is symmetric,
			\[A_2=\begin{pmatrix}
				0 & 0 & ra_{24} & 0 \\
				a_{21} & 0 & a_{23} & a_{24} \\
				-a_{24} & 0 & 0 & 0 \\
				a_{23} & -ra_{24} & a_{43} & 0
			\end{pmatrix}.\]
			That $A_2$ is nilpotent is equivalent to saying that $a_{24}=0$, so the matrix has the form
			\[A_2=\begin{pmatrix}
				0 & 0 & 0 & 0 \\
				a & 0 & b & 0 \\
				0 & 0 & 0 & 0 \\
				b & 0 & c & 0
			\end{pmatrix}\]
			for $a,b,c\in K$.
			
			If $b=0$ and $c=ra$, after subtracting a multiple of $A_1$ we end up in case (N2$(r)$). Suppose now that $b\ne 0$ or $c\ne ra$.
			
			Let \[t=(c-ra)^2+4rb^2.\] If $t=0$, that can only happen if $r=-s^2$ for some $s\in K$ and $c-ra=2bs$. In the real case, this means that $r=-1$ and $c-ra=2b$; after subtracting a multiple of $A_1$ and dividing by a constant, we end up in case (N7(2)). In the $p$-adic case, we can see that $2s\in Z_{p,r}$: if $p\equiv 1\mod 4$, $r=1$ and $s=\ii$; otherwise, $r=-1$ and $s=1$. After subtracting a multiple of $A_1$ and dividing by a constant, we make $a=0$, $b=1$ and $c=2s$, arriving at case (N7$(r,2s)$).
			
			Now suppose that $t\ne 0$. If $t$ is a square in $K$, we can reduce to case (N5): to do that, we define
			\[d=\frac{ra-c+\sqrt{t}}{2rb}\]
			and
			\[S_1=\begin{pmatrix}
				1 & 0 & rd & 0 \\
				0 & 1 & 0 & d \\
				-d & 0 & 1 & 0 \\
				0 & -rd & 0 & 1
			\end{pmatrix}.\]
			We have that $S_1^T\Omega_0S_1=\Omega_0$, $S_1^{-1}A_1S_1=A_1$ and $S_1^{-1}A_2S_1$ has the same form as $A_2$ but with $b=0$. After recombining $A_1$ and $A_2$, we finally get the normal form.
			
			If $t$ is not a square, in the $p$-adic case, $t$ must be in some class of \[\bDSq(K,r),\] and there is a $k\in Z_{p,r}$ such that $k^2+4r$ is in the same class and $t(k^2+4r)$ is a square in $K$. In the real case, $t$ not being a square implies that $t<0$ and $r=-1$, so the same holds for $k=0$. Using now
			\[d=\frac{k(c-ra)-4rb+\sqrt{t(k^2+4r)}}{2r(ra-c-kb)}\]
			and the same matrix $S_1$ as before, we have again that $S_1^T\Omega_0S_1=\Omega_0$, $S_1^{-1}A_1S_1=A_1$ and
			\[S_1^{-1}A_2S_1=\begin{pmatrix}
				0 & 0 & 0 & 0 \\
				a' & 0 & b' & 0 \\
				0 & 0 & 0 & 0 \\
				b' & 0 & ra'+kb' & 0
			\end{pmatrix},\]
			that can be changed to the form (N7$(k)$) or (N7$(r,k)$).
			
			\item \emph{Third form.}
			
			In this case the conditions that $A_1A_2=A_2A_1$ and $\Omega_0A_2$ is symmetric imply that
			\[A_2=\begin{pmatrix}
				a_{11} & a_{12} & a_{13} & 0 \\
				0 & -a_{11} & 0 & 0 \\
				0 & 0 & a_{11} & 0 \\
				0 & -a_{13} & a_{43} & -a_{11}
			\end{pmatrix}.\]
			Adding that $A_2$ must be nilpotent, we must have $a_{11}=0$. Adding a multiple of $A_1$ we make $a_{13}=0$. The matrix has at this point the form
			\[A_2=\begin{pmatrix}
				0 & a_{12} & 0 & 0 \\
				0 & 0 & 0 & 0 \\
				0 & 0 & 0 & 0 \\
				0 & 0 & a_{43} & 0
			\end{pmatrix}.\]
			If $a_{12}\ne 0$, multiplying by a constant we make $a_{12}=1$; if $a_{12}=0$ and $a_{43}\ne 0$, we have the same by changing the order of the coordinates; if $a_{12}=a_{43}=0$, we are in case (N3).
			
			Once we have made $a_{12}=1$, multiplying by a symplectic matrix of the form
			\[\begin{pmatrix}
				1 & 0 & 0 & 0 \\
				0 & 1 & 0 & 0 \\
				0 & 0 & k & 0 \\
				0 & 0 & 0 & \frac{1}{k}
			\end{pmatrix},\]
			we can make $a_{43}\in\{-1,0,1\}$ in the real case and $a_{43}\in Y_p\cup\{0,1\}$ in the $p$-adic case. The same symplectic matrix changes $A_1$ to $kA_1$. After rearranging the first and second coordinates, we get to the matrices
			\[B_1=\begin{pmatrix}
				0 & 0 & 0 & 0 \\
				0 & 0 & 1 & 0 \\
				0 & 0 & 0 & 0 \\
				1 & 0 & 0 & 0
			\end{pmatrix},
			B_2=\begin{pmatrix}
				0 & 0 & 0 & 0 \\
				1 & 0 & 0 & 0 \\
				0 & 0 & 0 & 0 \\
				0 & 0 & r & 0
			\end{pmatrix}.\]
			Taking this $B_2$ as a new $A_1$, we are in the first case if $r=0$ and the second otherwise, and we continue from there.
			
			\item \emph{Fourth form.}
			
			In this case, that $A_1A_2=A_2A_1$ and $\Omega_0A_2$ is symmetric implies that
			\[A_2=\begin{pmatrix}
				0 & 0 & 0 & a_{14} \\
				a_{14} & 0 & 0 & 0 \\
				0 & a_{14} & 0 & a_{34} \\
				0 & 0 & 0 & 0
			\end{pmatrix}.\]
			We make $a_{14}=0$ by adding a multiple of $A_1$. If $a_{34}=0$, we are in case N4; otherwise, after rearranging coordinates, we get to case (N6(1)).
		\end{enumerate}
		
		\item \emph{Semi-nilpotent case.}
		
		The semi-nilpotent degenerate matrices have eigenvalues of the form $\{\lambda,-\lambda,0,0\}$. $A_1$ can be changed to one of the forms
		\[\begin{pmatrix}
			\lambda & 0 & 0 & 0 \\
			0 & -\lambda & 0 & 0 \\
			0 & 0 & 0 & 0 \\
			0 & 0 & a & 0
		\end{pmatrix},\]
		where $a\in\{-1,0,1\}$ if $K=\R$ and $a\in Y_p\cup\{0,1\}$ if $K=\Qp$, by a symplectic matrix. Let $B_1$ be this matrix, and let $S_1$ be the matrix such that $S_1^{-1}A_1S_1=B_1$ and $S_1^T\Omega_0S_1=\Omega_0$. Let also $B_2=S_1^{-1}A_2S_1$.
		
		Using that $B_1B_2=B_2B_1$ and $\Omega_0B_2$ is symmetric, we have that
		\[B_2=\begin{pmatrix}
			\mu_1 & 0 & 0 & 0 \\
			0 & -\mu_1 & 0 & 0 \\
			0 & 0 & \mu_2 & \mu_3 \\
			0 & 0 & \mu_4 & -\mu_2
		\end{pmatrix}\]
		respectively. If $\lambda\notin K$, $K[\lambda]$ is a degree two extension of $K$. The first two columns of $S_1$ are of the form $u$ and $a\bar{u}$ for some $a\in K[\lambda]$, where $\bar{u}$ denotes the conjugate in $K[\lambda]$. This in turn implies that $\mu_1$ is imaginary, that is, it is of the form $r\lambda$ for some $r\in K$. That is, whether or not $\lambda=1$, we can always add a multiple of $B_1$ to $B_2$ to get a new matrix $B_3$, which is like $B_2$ but with $\mu_1=0$.
		
		The last two columns of $S_1$ must always be in $K$, because they correspond to eigenvalues $0$. This implies that $\mu_2$, $\mu_3$ and $\mu_4$ are always in $K$.
		
		Applying Theorem \ref{thm:williamson} to the matrix $B_1$ and the dimension $2$ subspace spanned by the first two vectors in the symplectic basis formed by the columns of $S_1$, we can multiply the matrix $B_1$ by a symplectic matrix
		\[S_2=\begin{pmatrix}
			\phi_1 & \phi_2 & 0 & 0 \\
			\phi_3 & \phi_4 & 0 & 0 \\
			0 & 0 & 1 & 0 \\
			0 & 0 & 0 & 1
		\end{pmatrix}\]
		such that $S_1S_2$ has entries in $K$ and
		\[S_2^{-1}B_1S_2=C_1=
		\begin{pmatrix}
			0 & -cr & 0 & 0 \\
			r & 0 & 0 & 0 \\
			0 & 0 & 0 & 0 \\
			0 & 0 & a & 0
		\end{pmatrix},\]
		for $r\in K$ and $c\in X_p$ in the $p$-adic case, and $c\in\{-1,1\}$ in the real case. $B_3$, which only has the lower right block different from zero, is not altered. We can suppose $r=1$.
		
		If $a=0$, this lower right block must be nilpotent, because otherwise some linear combination $C_1+kB_3$ with $k$ small enough would not be degenerate. This means we can always transform it to the form
		\[\begin{pmatrix}
			0 & 0 \\
			\mu_4 & 0
		\end{pmatrix}.\]
		If $\mu_4=0$, this is form (S1$(c)$); otherwise we make $\mu_4=1$ multiplying by a constant and we are in form (S3$(c)$).
		
		If $a\ne 0$ and $B_3=0$, by commutativity we must have $\mu_2=\mu_3=0$, and we are in case (S2$(c,a)$). If $a\ne 0$ and $B_3\ne 0$, we have again $\mu_2=\mu_3=0$ and we can add a multiple of $B_3$ to $B_1$ to make $a=0$, and then make $\mu_4=1$ by multiplying by a constant, so we end in case (S3$(c)$).
		
		\item \emph{Non-nilpotent case.}
		
		If a degenerate matrix is non-nilpotent, its eigenvalues must be of the form \[\{\lambda,\lambda,-\lambda,-\lambda\}.\]
		
		By Lemma \ref{lemma:eig2}, $A_1$ can be brought to one of the matrices
		\[\begin{pmatrix}
			\lambda & 0 & 0 & 0 \\
			0 & -\lambda & 0 & 0 \\
			0 & 0 & \lambda & 0 \\
			0 & 0 & 0 & -\lambda
		\end{pmatrix}\text{ or }
		\begin{pmatrix}
			\lambda & 0 & 1 & 0 \\
			0 & -\lambda & 0 & 0 \\
			0 & 0 & \lambda & 0 \\
			0 & -1 & 0 & -\lambda
		\end{pmatrix}\]
		by a symplectic matrix. Let $B_1$ be the resulting matrix, and let $S_1$ be the matrix such that $S_1^{-1}A_1S_1=B_1$ and $S_1^T\Omega_0S_1=\Omega_0$. Let also $B_2=S_1^{-1}A_2S_1$. As $B_2$ commutes with $B_1$, it must have the form
		\[\begin{pmatrix}
			\mu_1 & 0 & \mu_2 & 0 \\
			0 & \mu_5 & 0 & \mu_6 \\
			\mu_3 & 0 & \mu_4 & 0 \\
			0 & \mu_7 & 0 & \mu_8
		\end{pmatrix}\]
		We also know that $S_1^TM_2S=\Omega_0 B_2$ is a symmetric matrix, so we have $\mu_1=-\mu_5$, $\mu_2=-\mu_7$, $\mu_3=-\mu_6$ and $\mu_4=-\mu_8$.
		
		There are two possible cases, whether $\lambda\in K$ or $\lambda\not\in K$, which we call ``real'' and ``imaginary'' by extension of what happens when $K=\R$. In both cases $\lambda^2\in K$.
		
		\begin{enumerate}
			\item \emph{Real case.}
			
			In this case $B_1$ and $S_1$, and hence also $B_2$, have entries in the base field $K$. This means that we can add a multiple of $B_1$ to $B_2$ to get a new basis $(B_1,B_3)$ of $\Span(B_1,B_2)$ where $B_3$ is like $B_2$ but with \[\mu_1=-\mu_4=-\mu_5=\mu_8.\]
			
			If $B_1$ is not diagonal, the relation $B_1B_3=B_3B_1$ implies that \[\mu_1=\mu_3=\mu_4=\mu_5=\mu_6=\mu_8=0,\] that is,
			
			\[B_3=\begin{pmatrix}
				0 & 0 & \mu_2 & 0 \\
				0 & 0 & 0 & 0 \\
				0 & 0 & 0 & 0 \\
				0 & -\mu_2 & 0 & 0
			\end{pmatrix}\]
			
			If $\mu_2=0$, we are in case (R2); otherwise we divide this by $\mu_2$, subtract the resulting matrix from $B_1$, and divide it by $\lambda$, resulting in case (R3).
			
			If $B_1$ is diagonal, we can divide $B_1$ by $\lambda$ to make $\lambda=1$. If $B_3=0$, we are in case (R1). Otherwise, we will prove that $B_3$ must be nilpotent.
			
			Suppose otherwise and let \[B_4=B_3+kB_1,\] where $k\in K$ is small (in the real or the $p$-adic sense). The characteristic polynomial of
			\[B_4=\begin{pmatrix}
				\mu_1+k & 0 & \mu_2 & 0 \\
				0 & -\mu_1-k & 0 & -\mu_3 \\
				\mu_3 & 0 & -\mu_1+k & 0 \\
				0 & -\mu_2 & 0 & \mu_1-k
			\end{pmatrix}\]
			is
			\[((t-k)^2-\mu_1^2-\mu_2\mu_3)((t+k)^2-\mu_1^2-\mu_2\mu_3).\]
			$B_3$ has the same characteristic polynomial but with $k=0$; as it is not nilpotent, $\mu_1^2+\mu_2\mu_3\ne 0$, and if $k$ is small enough, this polynomial has four different roots and the origin is non-degenerate, a contradiction.
			
			That $B_3$ is nilpotent and not zero implies that there are numbers $\phi_1,\phi_2,\phi_3,\phi_4\in K$ such that
			\begin{equation}\label{eq:diag-deg}
				\begin{pmatrix}
					\phi_1 & \phi_2 \\
					\phi_3 & \phi_4
				\end{pmatrix}^{-1}
				\begin{pmatrix}
					\mu_1 & \mu_2 \\
					\mu_3 & -\mu_1
				\end{pmatrix}
				\begin{pmatrix}
					\phi_1 & \phi_2 \\
					\phi_3 & \phi_4
				\end{pmatrix}
				=\begin{pmatrix}
					0 & 1 \\
					0 & 0
				\end{pmatrix}.
			\end{equation}
			
			Let
			\[\begin{pmatrix}
				\phi_5 & \phi_6 \\
				\phi_7 & \phi_8
			\end{pmatrix}
			=\begin{pmatrix}
				\phi_1 & \phi_3 \\
				\phi_2 & \phi_4
			\end{pmatrix}^{-1}.\]
			
			Equation \eqref{eq:diag-deg} implies that
			\begin{equation}\label{eq:diag-deg2}
				\begin{pmatrix}
					\phi_5 & \phi_6 \\
					\phi_7 & \phi_8
				\end{pmatrix}^{-1}
				\begin{pmatrix}
					-\mu_1 & -\mu_3 \\
					-\mu_2 & \mu_1
				\end{pmatrix}
				\begin{pmatrix}
					\phi_5 & \phi_6 \\
					\phi_7 & \phi_8
				\end{pmatrix}
				=\begin{pmatrix}
					0 & 0 \\
					-1 & 0
				\end{pmatrix}.
			\end{equation}
			
			Now we define
			\[S_2=\begin{pmatrix}
				\phi_1 & 0 & \phi_2 & 0 \\
				0 & \phi_5 & 0 & \phi_6 \\
				\phi_3 & 0 & \phi_4 & 0 \\
				0 & \phi_7 & 0 & \phi_8
			\end{pmatrix}.\]
			
			We have that $S_2$ is a symplectic matrix because
			\[S_2^T\Omega_0S_2=\begin{pmatrix}
				0 & \phi_1\phi_5+\phi_3\phi_7 & 0 & \phi_1\phi_6+\phi_3\phi_8 \\
				-\phi_1\phi_5-\phi_3\phi_7 & 0 & -\phi_2\phi_5-\phi_4\phi_7 & 0 \\
				0 & \phi_2\phi_5+\phi_4\phi_7 & 0 & \phi_2\phi_6+\phi_4\phi_8 \\
				-\phi_1\phi_6-\phi_3\phi_8 & 0 & -\phi_2\phi_6-\phi_4\phi_8 & 0
			\end{pmatrix}=\Omega_0.\]
			
			It is clear that $S_2^{-1}B_1S_2=B_1$. Finally, putting together \eqref{eq:diag-deg} and \eqref{eq:diag-deg2}, we get
			\[S_2^{-1}B_3S_2=\begin{pmatrix}
				0 & 0 & 1 & 0 \\
				0 & 0 & 0 & 0 \\
				0 & 0 & 0 & 0 \\
				0 & -1 & 0 & 0
			\end{pmatrix},\]
			and we are in case (R3).
			
			\item \emph{Imaginary case.}
			
			In this case $K[\lambda]$ is a degree two extension of $K$ and $\mu_i\in K[\lambda]$ for $1\le i \le 8$. We can multiply $\lambda$ by any number in $K$, that is, we can make $\lambda=\ii$ if $K=\R$, and $\lambda=\sqrt{c}$ for some $c\in Y_p$ if $K=\Qp$. As usual, we denote by $x\mapsto\bar{x}$ the conjugate in $K[\lambda]$. We have two cases, depending on whether $B_1$ is diagonal or not.
			
			\begin{enumerate}
				\item \emph{$B_1$ is diagonal.}
				
				In this case we have that the columns of $S_1$ have the form $(u,a\bar{u},v,b\bar{v})$, for $u,v\in K[\lambda]^4$ and $a,b\in K[\lambda]$. We also have that
				\[au^T\Omega_0\bar{u}=1=\bar{1}=\overline{au^T\Omega_0\bar{u}}=-\bar{a}u^T\Omega_0\bar{u},\]
				that is, $\bar{a}=-a$, and analogously $\bar{b}=-b$. Using that $A_2S_1=S_1 B_2$,
				\[A_2\begin{pmatrix}
					u & a\bar{u} & v & b\bar{v}
				\end{pmatrix}=
				\begin{pmatrix}
					\mu_1 u+\mu_3 v & \mu_5a\bar{u}+\mu_7b\bar{v} & \mu_2 u+\mu_4 v & \mu_6a\bar{u}+\mu_8b\bar{v}
				\end{pmatrix}\]
				which implies that \[a(\overline{\mu_1 u+\mu_3 v})=\mu_5a\bar{u}+\mu_7b\bar{v}\] and \[b(\overline{\mu_2 u+\mu_4 v})=\mu_6a\bar{u}+\mu_8b\bar{v},\] and as $u$ and $v$ are linearly independent, we must have
				\[a\bar{\mu}_1=a\mu_5=-a\mu_1,\]
				\[a\bar{\mu}_3=b\mu_7=-b\mu_2,\]
				\[b\bar{\mu}_2=a\mu_6=-a\mu_3,\]
				\[b\bar{\mu}_4=b\mu_8=-b\mu_4.\]
				This means that
				\[B_2=\begin{pmatrix}
					s_1\lambda & 0 & a(r_2+s_2\lambda) & 0 \\
					0 & -s_1\lambda & 0 & -b(r_2-s_2\lambda) \\
					b(r_2-s_2\lambda) & 0 & s_3\lambda & 0 \\
					0 & -a(r_2+s_2\lambda) & 0 & -s_3\lambda
				\end{pmatrix}\]
				for some $s_1,r_2,s_2,s_3\in K$. We can add a multiple of $B_1$ to this matrix so that it still has the same form but with $s_3=-s_1.$ Let $B_3$ be the resulting matrix, such that $\Span(B_1,B_2)=\Span(B_1,B_3)$.
				
				If $B_3=0$, we can now apply the symplectic matrix
				\[S_2=\begin{pmatrix}
					1 & \lambda & 0 & 0 \\
					-\frac{1}{2\lambda} & \frac{1}{2} & 0 & 0 \\
					0 & 0 & 1 & \lambda \\
					0 & 0 & -\frac{1}{2\lambda} & \frac{1}{2}
				\end{pmatrix}\]
				to obtain
				\[S_2^{-1}B_1S_2=\begin{pmatrix}
					0 & \lambda^2 & 0 & 0 \\
					1 & 0 & 0 & 0 \\
					0 & 0 & 0 & \lambda^2 \\
					0 & 0 & 1 & 0
				\end{pmatrix}\]
				and we are in case (I1$(\lambda^2)$).
				
				If $B_3\ne 0$, we can prove that $B_3$ is nilpotent by the same reason as in the real case. This means that
				\[s_1^2\lambda^2+ab(r_2^2-s_2^2\lambda^2)=0\]
				
				We define
				\[t_1=\frac{1+\lambda^2}{1-\lambda^2},t_2=\frac{2\lambda}{1-\lambda^2}\]
				and
				\begin{align*}
					\phi_1 & =-2s_1t_2\lambda, \\
					\phi_2 & =\frac{a}{2s_1}-s_1t_1t_2, \\
					\phi_3 & =\frac{2s_1^2t_2\lambda^2}{a(r_2+s_2\lambda)}, \\
					\phi_4 & =-\frac{\lambda(a+2s_1^2t_1t_2)}{2a(r_2+s_2\lambda)}.
				\end{align*}
				
				We have that
				\begin{align}
					\frac{1}{a}\phi_1\bar{\phi}_1+\frac{1}{b}\phi_3\bar{\phi}_3 & =\frac{4s_1^2t_2^2\lambda^2}{a}+\frac{1}{b}\frac{4s_1^4t_2^2\lambda^4}{a^2(r_2^2-s_2^2\lambda^2)} \label{eq:phi11}\\
					& =\frac{4s_1^2t_2^2\lambda^2}{a}-\frac{4s_1^2t_2^2\lambda^2}{a} \nonumber\\
					& =0;\nonumber
				\end{align}
				\begin{align}
					\frac{1}{a}\phi_1\bar{\phi}_2+\frac{1}{b}\phi_3\bar{\phi}_4 & =\frac{1}{a}(-2s_1t_2\lambda)\left(-\frac{a}{2s_1}+s_1t_1t_2\right)+\frac{1}{b}\frac{2s_1^2t_2\lambda^2}{a(r_2+s_2\lambda)}\frac{\lambda(a+2s_1^2t_1t_2)}{-2a(r_2-s_2\lambda)} \label{eq:phi12}\\
					& =t_2\lambda-\frac{2s_1^2t_1t_2^2\lambda}{a}-\frac{s_1^2t_2\lambda^3(a+2s_1^2t_1t_2)}{a^2b(r_2^2-s_2^2\lambda^2)} \nonumber\\
					& =t_2\lambda-\frac{2s_1^2t_1t_2^2\lambda}{a}+\frac{t_2\lambda(a+2s_1^2t_1t_2)}{a} \nonumber\\
					& =2t_2\lambda;\nonumber
				\end{align}
				\begin{align}
					\frac{1}{a}\phi_2\bar{\phi}_2+\frac{1}{b}\phi_4\bar{\phi}_4 & =\frac{1}{a}\left(\frac{a}{2s_1}-s_1t_1t_2\right)\left(\frac{-a}{2s_1}+s_1t_1t_2\right)+\frac{1}{b}\frac{\lambda^2(a+2s_1^2t_1t_2)^2}{-4a^2(r_2^2-s_2^2\lambda^2)} \label{eq:phi22}\\
					& =-\frac{1}{a}\left(\frac{a}{2s_1}-s_1t_1t_2\right)^2+\frac{(a+2s_1^2t_1t_2)^2}{4as_1^2} \nonumber\\
					& =-\frac{(a-2s_1^2t_1t_2)^2}{4as_1^2}+\frac{(a+2s_1^2t_1t_2)^2}{4as_1^2} \nonumber\\
					& =\frac{8as_1^2t_1t_2}{4as_1^2} \nonumber\\
					& =2t_1t_2.\nonumber
				\end{align}
				
				Now we set
				\[\phi_5=-\frac{\bar{\phi}_2}{a},\phi_6=\frac{\bar{\phi}_1}{a},\phi_7=-\frac{\bar{\phi}_4}{b},\phi_8=\frac{\bar{\phi}_3}{b}\]
				and, as in the real case,
				\[S_2=\begin{pmatrix}
					\phi_1 & 0 & \phi_2 & 0 \\
					0 & \phi_5 & 0 & \phi_6 \\
					\phi_3 & 0 & \phi_4 & 0 \\
					0 & \phi_7 & 0 & \phi_8
				\end{pmatrix}.\]
				
				This makes
				\[S_2^T\Omega_0S_2=\begin{pmatrix}
					0 & \phi_1\phi_5+\phi_3\phi_7 & 0 & \phi_1\phi_6+\phi_3\phi_8 \\
					-\phi_1\phi_5-\phi_3\phi_7 & 0 & -\phi_2\phi_5-\phi_4\phi_7 & 0 \\
					0 & \phi_2\phi_5+\phi_4\phi_7 & 0 & \phi_2\phi_6+\phi_4\phi_8 \\
					-\phi_1\phi_6-\phi_3\phi_8 & 0 & -\phi_2\phi_6-\phi_4\phi_8 & 0
				\end{pmatrix},\]
				which, using the relations \eqref{eq:phi11}--\eqref{eq:phi22}, becomes
				\[S_2^T\Omega_0S_2=\begin{pmatrix}
					0 & -2t_2\lambda & 0 & 0 \\
					2t_2\lambda & 0 & 2t_1t_2 & 0 \\
					0 & -2t_1t_2 & 0 & -2t_2\lambda \\
					0 & 0 & 2t_2\lambda & 0
				\end{pmatrix}.\]
				
				Let $C_i=S_2^{-1}B_iS_2$, for $i=1$ or $3$. It is clear that $C_1=B_1$. Respecting to $C_3$, first we have that
				\begin{equation}\label{eq:diag-deg-imag}
					\begin{pmatrix}
						\phi_1 & \phi_2 \\
						\phi_3 & \phi_4
					\end{pmatrix}^{-1}
					\begin{pmatrix}
						s_1\lambda & a(r_2+s_2\lambda) \\
						b(r_2-s_2\lambda) & -s_1\lambda
					\end{pmatrix}
					\begin{pmatrix}
						\phi_1 & \phi_2 \\
						\phi_3 & \phi_4
					\end{pmatrix}
				\end{equation}
				\[=\begin{pmatrix}
					\phi_1 & \phi_2 \\
					\phi_3 & \phi_4
				\end{pmatrix}^{-1}
				\begin{pmatrix}
					-2s_1^2t_2\lambda^2+2s_1^2t_2\lambda^2 & \frac{a\lambda}{2}-s_1^2t_1t_2\lambda-\frac{\lambda(a+2s_1^2t_1t_2)}{2} \\
					-2s_1t_2\lambda b(r_2-s_2\lambda)-s_1\lambda\frac{2s_1^2t_2\lambda^2}{a(r_2+s_2\lambda)} & b(r_2-s_2\lambda)(\frac{a}{2s_1}-s_1t_1t_2)+s_1\lambda\frac{\lambda(a+2s_1^2t_1t_2)}{2a(r_2+s_2\lambda)}
				\end{pmatrix}\]
				\[=\begin{pmatrix}
					\phi_1 & \phi_2 \\
					\phi_3 & \phi_4
				\end{pmatrix}^{-1}
				\begin{pmatrix}
					0 & -s_1^2t_1t_2\lambda \\
					-s_1\lambda\frac{2abt_2\lambda(r_2^2-s_2^2\lambda^2)+2s_1^2t_2\lambda^2}{a(r_2+s_2\lambda)} & \frac{ab(r_2^2-s_2^2\lambda^2)(a-2s_1^2t_1t_2)+s_1^2\lambda^2(a+2s_1^2t_1t_2)}{2as_1(r_2+s_2\lambda)}
				\end{pmatrix}\]
				\[=\begin{pmatrix}
					\phi_1 & \phi_2 \\
					\phi_3 & \phi_4
				\end{pmatrix}^{-1}
				\begin{pmatrix}
					0 & -2s_1^2t_1t_2\lambda \\
					0 & \frac{2s_1^3\lambda^2t_1t_2}{a(r_2+s_2\lambda)}
				\end{pmatrix}
				=\begin{pmatrix}
					\phi_1 & \phi_2 \\
					\phi_3 & \phi_4
				\end{pmatrix}^{-1}
				\begin{pmatrix}
					0 & s_1t_1\phi_1 \\
					0 & s_1t_1\phi_3
				\end{pmatrix}
				=\begin{pmatrix}
					0 & s_1t_1 \\
					0 & 0
				\end{pmatrix}.\]
				
				Relations \eqref{eq:phi11}--\eqref{eq:phi22} imply that
				\[\begin{pmatrix}
					\phi_1 & \phi_2 \\
					\phi_3 & \phi_4
				\end{pmatrix}
				\begin{pmatrix}
					\phi_5 & \phi_6 \\
					\phi_7 & \phi_8
				\end{pmatrix}
				=\begin{pmatrix}
					-2t_2\lambda & 0 \\
					-2t_1t_2 & -2t_2\lambda
				\end{pmatrix}\]
				which together with \eqref{eq:diag-deg-imag} give
				\[\begin{pmatrix}
					\phi_5 & \phi_6 \\
					\phi_7 & \phi_8
				\end{pmatrix}^{-1}
				\begin{pmatrix}
					-s_1\lambda & -b(r_2-s_2\lambda) \\
					-a(r_2+s_2\lambda) & s_1\lambda
				\end{pmatrix}
				\begin{pmatrix}
					\phi_5 & \phi_6 \\
					\phi_7 & \phi_8
				\end{pmatrix}
				=\begin{pmatrix}
					0 & 0 \\
					-s_1t_1 & 0
				\end{pmatrix}\]
				and finally
				\begin{equation}\label{eq:C3}
					C_3=S_2^{-1}B_3S_2=\begin{pmatrix}
						0 & 0 & s_1t_1 & 0 \\
						0 & 0 & 0 & 0 \\
						0 & 0 & 0 & 0 \\
						0 & -s_1t_1 & 0 & 0
					\end{pmatrix}.
				\end{equation}
				
				Now we take
				\[S_3=\begin{pmatrix}
					-\frac{t_1}{t_2\lambda^2} & 1 & 0 & \frac{1}{\lambda} \\
					\frac{1}{t_2\lambda} & 0 & \frac{1}{t_2} & 0 \\
					-\frac{1}{t_2\lambda} & 0 & \frac{1}{t_2} & 0 \\
					\frac{t_1}{t_2\lambda^2} & 1 & 0 & -\frac{1}{\lambda}
				\end{pmatrix}.\]
				This leads to
				\[S_3^TS_2^T\Omega_0S_2S_3=S_3^T\begin{pmatrix}
					0 & -2t_2\lambda & 0 & 0 \\
					2t_2\lambda & 0 & 2t_1t_2 & 0 \\
					0 & -2t_1t_2 & 0 & -2t_2\lambda \\
					0 & 0 & 2t_2\lambda & 0
				\end{pmatrix}S_3=\Omega_0,\]
				\begin{equation}\label{eq:S3C1}
					S_3^{-1}C_1S_3=\begin{pmatrix}
						0 & 0 & -\lambda^2 & 0 \\
						\frac{t_1}{t_2\lambda} & 0 & 0 & 1 \\
						-1 & 0 & 0 & 0 \\
						0 & \lambda^2 & \frac{t_1\lambda}{t_2} & 0
					\end{pmatrix},
				\end{equation}
				and
				\[S_3^{-1}C_3S_3=\begin{pmatrix}
					0 & 0 & 0 & 0 \\
					-\frac{s_1t_1}{t_2\lambda} & 0 & 0 & 0 \\
					0 & 0 & 0 & 0 \\
					0 & 0 & \frac{s_1t_1\lambda}{t_2} & 0
				\end{pmatrix}.\]
				After dividing the second matrix by $-s_1t_1/t_2\lambda$ and adding it to the first multiplied by $-t_1t_2/\lambda$, the matrices become
				\[\begin{pmatrix}
					0 & 0 & -\lambda^2 & 0 \\
					0 & 0 & 0 & 1 \\
					-1 & 0 & 0 & 0 \\
					0 & \lambda^2 & \frac{2t_1\lambda}{t_2} & 0
				\end{pmatrix}
				\text{ and }
				\begin{pmatrix}
					0 & 0 & 0 & 0 \\
					1 & 0 & 0 & 0 \\
					0 & 0 & 0 & 0 \\
					0 & 0 & -\lambda^2 & 0
				\end{pmatrix}\]
				and we are in case (I3$(\lambda^2)$). We need to show that the symplectic matrix $S_1S_2S_3$ by which we are multiplying is in $K$:
				\[S_1S_2S_3=\begin{pmatrix}
					u & a\bar{u} & v & b\bar{v}
				\end{pmatrix}
				\begin{pmatrix}
					-\frac{t_1\phi_1}{t_2\lambda^2}-\frac{\phi_2}{t_2\lambda} & \phi_1 & \frac{\phi_2}{t_2} & \frac{\phi_1}{\lambda} \\
					\frac{t_1\phi_6}{t_2\lambda^2}+\frac{\phi_5}{t_2\lambda} & \phi_6 & \frac{\phi_5}{t_2} & -\frac{\phi_6}{\lambda} \\
					-\frac{t_1\phi_3}{t_2\lambda^2}-\frac{\phi_4}{t_2\lambda} & \phi_3 & \frac{\phi_4}{t_2} & \frac{\phi_3}{\lambda} \\
					\frac{t_1\phi_8}{t_2\lambda^2}+\frac{\phi_7}{t_2\lambda} & \phi_8 & \frac{\phi_7}{t_2} & -\frac{\phi_8}{\lambda}
				\end{pmatrix}\]
				This is real because \[\overline{\phi_1u+\phi_3v}=\phi_6a\bar{u}+\phi_8b\bar{v}\] and \[\overline{\phi_2u+\phi_4v}=-\phi_5a\bar{u}-\phi_7b\bar{v},\] which follows from the definitions of the $\phi_i$'s.
				
				\item $B_1$ is not diagonal
				
				As $B_1B_2=B_2B_1$, the form of $B_1$ implies $\mu_1=\mu_4$, $\mu_5=\mu_8$ and $\mu_3=\mu_6=0$.
				
				Now the columns of $S_1$ have instead the form $(u,-a\bar{v},v,a\bar{u})$, for $u,v\in K[\lambda]^4$ and $a\in K[\lambda]$. We also have that
				\[-au^T\Omega_0\bar{v}=1=\bar{1}=\overline{av^T\Omega_0\bar{u}}=-\bar{a}u^T\Omega_0\bar{v},\]
				that is, $\bar{a}=a$ and $a\in K$. Using that $A_2S_1=S_1 B_2$,
				\[A_2\begin{pmatrix}
					u & -a\bar{v} & v & a\bar{u}
				\end{pmatrix}=
				\begin{pmatrix}
					\mu_1 u & -\mu_5a\bar{v}+\mu_7a\bar{u} & \mu_2 u+\mu_1 v & \mu_5a\bar{u}
				\end{pmatrix}\]
				which implies that \[a\overline{\mu_1 u}=\mu_5a\bar{u}=-\mu_1 a\bar{u},\] which implies that $\mu_1$ is imaginary, and \[-a(\overline{\mu_2 u+\mu_1 v})=-\mu_5a\bar{v}+\mu_7a\bar{u}=\mu_1a\bar{v}-\mu_2a\bar{u},\] which in turn implies that $\mu_2$ is in $K$. At this point we have that
				\[B_2=\begin{pmatrix}
					s_1\lambda & 0 & r_2 & 0 \\
					0 & -s_1\lambda & 0 & 0 \\
					0 & 0 & -s_1\lambda & 0 \\
					0 & -r_2 & 0 & s_1\lambda
				\end{pmatrix}.\]
				Making $C_1=B_1$ and $C_3=B_2-s_1B_1$, we are in the same situation as in the previous case \eqref{eq:C3}, but with $r_2$ instead of $s_1t_1$. If $r_2\ne 0$, we can apply the same matrix $S_3$ and we get again to case (I3$(\lambda^2)$). Otherwise, $C_3=0$, and $C_1$ can be brought by the same $S_3$ to the matrix in \eqref{eq:S3C1}, and we are in case (I2$(\lambda^2)$).\qedhere
			\end{enumerate}
		\end{enumerate}
	\end{enumerate}
\end{proof}

\begin{remark}
	In Theorems \ref{thm:real-degenerate} and \ref{thm:padic-degenerate}, unlike the other classifications, we do not have a uniqueness result. The reason is that, to find a set of classes that contains all cases, one has to find for each case a way to change it into the normal form of the class, which here means to find a local symplectomorphism, and in particular a linear symplectomorphism (that is, a multiplication by a symplectic matrix). But to prove that the class is unique for a given case, one has to \textit{prove that such a matrix does not exist}. In the non-degenerate cases this was possible thanks to the difference in eigenvalues, but now we have many cases with the same eigenvalues and the same strategy does not work.
\end{remark}

\begin{remark}
	In Theorem \ref{thm:padic-degenerate} we need $M$ and $F$ to be analytic, while we do not need that in Theorem \ref{thm:real-degenerate}. That is because Lemma \ref{lemma:hessian} only works if the function is analytic, while the corresponding result in the real case is true for any smooth function.
\end{remark}

\begin{remark}
	There is an extensive theory of quadratic forms over different types of fields, we refer to the classical treatment \cite{Meara} and the more recent works by Alsina-Bayer \cite{AlsBay}, Bhargava \cite{Bhargava}, Casselman \cite{Casselman} and Lam \cite{Lam} and the references therein. In Theorem \ref{thm:integrable} and Theorem \ref{thm:padic-degenerate} we have presented a list of normal forms of integrable systems up to local symplectic transformations, given by sums of binary quadratic forms, but we have not carried out a further analysis of the structure/properties of these forms since this does not appear to us as applicable in our context of symplectic geometry of integrable systems.
\end{remark}

\subsection{Symplectic dynamics of integrable systems and their level sets}

We now calculate the vector fields generated by the integrable systems of Theorem \ref{thm:integrable}.

\begin{proposition}
	The vector fields generated by the integrable systems of Theorem \ref{thm:integrable} are as follows:
	\begin{enumerate}
		\item $X_{g_1}=(2c_1\xi,-2x,0,0),X_{g_2}=(0,0,2c_2\eta,-2y)$.
		\item $X_{g_1}=(cy,-\eta,x,-c\xi),X_{g_2}=(x,-\xi,y,-\eta)$.
		\item
		\[X_{g_1}=\left(-\frac{t_1+bt_2}{a}\xi-(bt_1+ct_2)\eta,\frac{acx-by}{b^2-c},-ac(t_1+bt_2)\eta-(bt_1+ct_2)\xi,\frac{y-abx}{a(b^2-c)}\right),\]
		\[X_{g_2}=\left(-\frac{bt_1+ct_2}{a}\xi-c(t_1+bt_2)\eta,\frac{cy-abcx}{b^2-c},-ac(bt_1+ct_2)\eta-c(t_1+bt_2)\xi,\frac{acx-by}{a(b^2-c)}\right).\]
	\end{enumerate}
\end{proposition}

\begin{proof}
	All fields are calculated applying directly the equation $\imath_{X_f}\omega_0=\dd f$.
\end{proof}

We can check that, in each system, $g_1$ and $g_2$ Poisson commute, which is equivalent to checking that, if $A_i=\Omega_0^{-1}\dd^2 g_i$, we have $A_1A_2=A_2A_1$. This matrix is zero in case (1),
\[\begin{pmatrix}
	0 & 0 & c & 0 \\
	0 & 0 & 0 & 1 \\
	1 & 0 & 0 & 0 \\
	0 & c & 0 & 0
\end{pmatrix}\]
in case (2), and
\[\begin{pmatrix}
	ct_2 & 0 & \frac{t_1}{a} & 0 \\
	0 & ct_2 & 0 & act_1 \\
	act_1 & 0 & ct_2 & 0 \\
	0 & \frac{t_1}{a} & 0 & ct_2
\end{pmatrix}\]
in case (3).

We can also calculate the fibers of the systems.

\begin{proposition}
	The fibers $F^{-1}(0,0)$ of the integrable systems of Theorem \ref{thm:integrable} with a critical point of rank $0$ are as follows:
	\begin{enumerate}
		\item $\{(\pm d_1\xi,\xi,\pm d_2\eta,\eta):\xi,\eta\in\Qp\}$ if $-c_1=d_1^2$ and $-c_2=d_2^2$, $\{(\pm d_1\xi,\xi,0,0):\xi\in\Qp\}$ if $-c_1=d_1^2$ and $-c_2$ is not a square, and $\{(0,0,0,0)\}$ if $-c_1$ and $-c_2$ are not squares.
		\item $\{(x,0,y,0):x,y\in\Qp\}\cup\{(0,\xi,0,\eta):\xi,\eta\in\Qp\}$.
		\item $\{(0,0,0,0)\}$.
	\end{enumerate}
	For those in which the origin has rank $1$, the fibers are $\{(\pm d\xi,\xi,y,0):\xi,y\in\Qp\}$ if $-c=d^2$ and $\{(0,0,y,0):y\in\Qp\}$ otherwise.
\end{proposition}

\begin{proof}
	The part about rank $1$ is immediate from the formula $(x^2+c\xi^2,\eta)$. For the rank $0$ part, we have:
	\begin{enumerate}
		\item This follows from making $x^2+c_1\xi^2=y^2+c_2\eta^2=0$.
		\item We make $x\eta+cy\xi=x\xi+y\eta=0$. Considering this a system in $(\xi,\eta)$, we have two possible cases:
		\begin{itemize}
			\item The determinant of the coefficient matrix is $0$. Then $cy^2-x^2=0$. Since $c$ is not a square, $x=y=0$.
			\item The determinant of the coefficient matrix is not $0$. Then $\xi=\eta=0$.
		\end{itemize}
		\item We consider the coordinate change given by
		\[\Psi=\begin{pmatrix}
			(b+\alpha)\gamma & -(b+\alpha)\gamma & (b-\alpha)\bar{\gamma} & -(b-\alpha)\bar{\gamma} \\
			a\alpha & a\alpha & -a\alpha & -a\alpha \\
			a\alpha\gamma(b+\alpha) & -a\alpha\gamma(b+\alpha) & -a\alpha\bar{\gamma}(b-\alpha) & a\alpha\bar{\gamma}(b-\alpha) \\
			1 & 1 & 1 & 1
		\end{pmatrix},\]
		where $\alpha=\sqrt{c}$, $\gamma=\sqrt{t_1+t_2\alpha}$, $x\mapsto\bar{x}$ is the automorphism of $\Qp[\alpha]$ which sends $\alpha$ to $-\alpha$, and $x\mapsto\hat{x}$ is the automorphism of $\Qp[\gamma,\bar{\gamma}]$ which sends $\gamma$ to $-\gamma$ and $\bar{\gamma}$ to $-\bar{\gamma}$. We define
		\[\begin{pmatrix}
			x \\ \xi \\ y \\ \eta
		\end{pmatrix}=\Psi
		\begin{pmatrix}
			x' \\ \xi' \\ y' \\ \eta'
		\end{pmatrix}.\]
		We have that the first column of $\Psi$ is the hat-conjugate of the second, the same happens for the third and fourth, and the original coordinates are all in $\Qp$ and they are their own conjugates, hence
		\[\begin{pmatrix}
			x \\ \xi \\ y \\ \eta
		\end{pmatrix}=
		\begin{pmatrix}
			\hat{x} \\ \hat{\xi} \\ \hat{y} \\ \hat{\eta}
		\end{pmatrix}=
		\hat{\Psi}\begin{pmatrix}
			\hat{x}' \\ \hat{\xi}' \\ \hat{y}' \\ \hat{\eta}'
		\end{pmatrix}=
		\Psi\begin{pmatrix}
			\hat{\xi}' \\ \hat{x}' \\ \hat{\eta}' \\ \hat{y}'
		\end{pmatrix}\]
		which implies $\xi'=\hat{x}'$ and $\eta'=\hat{y}'$. Now, if $(x,\xi,y,\eta)$ is a point in the fiber,
		\[0=\begin{pmatrix}
			x & \xi & y & \eta
		\end{pmatrix}
		M\begin{pmatrix}
			x \\ \xi \\ y \\ \eta
		\end{pmatrix}=
		\begin{pmatrix}
			x' & \xi' & y' & \eta'
		\end{pmatrix}
		\Psi_2^TM\Psi_2\begin{pmatrix}
			x' \\ \xi' \\ y' \\ \eta'
		\end{pmatrix}\]
		where $M$ is the matrix of the normal form. We know that $\Psi_2$ diagonalizes $\Omega_0^{-1}M$, so
		\[\Psi_2^TM\Psi_2=\Psi_2^T\Omega_0\Psi_2\Psi_2^{-1}\Omega_0^{-1}M\Psi_2\]
		\[=\begin{pmatrix}
			0 & 4a\alpha\gamma(b+\alpha) & 0 & 0 \\
			-4a\alpha\gamma(b+\alpha) & 0 & 0 & 0 \\
			0 & 0 & 0 & -4a\alpha\bar{\gamma}(b-\alpha) \\
			0 & 0 & 4a\alpha\bar{\gamma}(b-\alpha) & 0
		\end{pmatrix}\begin{pmatrix}
			\lambda & 0 & 0 & 0 \\
			0 & -\lambda & 0 & 0 \\
			0 & 0 & \mu & 0 \\
			0 & 0 & 0 & -\mu
		\end{pmatrix}\]
		\[=-4a\alpha\begin{pmatrix}
			0 & \gamma(b+\alpha)\lambda & 0 & 0 \\
			\gamma(b+\alpha)\lambda & 0 & 0 & 0 \\
			0 & 0 & 0 & \bar{\gamma}(b-\alpha)\mu \\
			0 & 0 & \bar{\gamma}(b-\alpha)\mu
		\end{pmatrix}.\]
		Putting this together, we get
		\[0=\gamma(b+\alpha)\lambda x'\xi'+\bar{\gamma}(b-\alpha)\mu y'\eta'=\gamma^2(b+\alpha)(r+s\alpha)x'\hat{x}'+\bar{\gamma}^2(b-\alpha)(r-s\alpha)y'\hat{y}'\]
		This must hold for all $r,s\in\Qp$. Putting $(r,s)=(1,0)$ and $(0,1)$,
		\[\gamma^2(b+\alpha)x'\hat{x}'+\bar{\gamma}^2(b-\alpha)y'\hat{y}'=\alpha\gamma^2(b+\alpha)x'\hat{x}'-\alpha\bar{\gamma}^2(b-\alpha)y'\hat{y}'=0\]
		These two equations imply $x'\hat{x}'=0$ and $y'\hat{y}'=0$, that is, $x'=y'=0$, which in turn implies $\xi'=\eta'=0$ and the vector is zero.\qedhere
	\end{enumerate}
\end{proof}

In analogy with the real case, a submanifold $N$ of a $2n$-dimensional $p$-adic analytic symplectic manifold $(M,\omega)$ is said to be \emph{isotropic} if the tangent space at each point of $N$ is an isotropic subspace of the tangent space of $M$, that is, if $\omega(u,v)=0$ for any two vectors $u,v\in\mathrm{T}_mN$. It is called \emph{Lagrangian} if it is isotropic and with dimension $n$.

The fibers of regular points of real integrable systems are Lagrangian, and homeomorphic to tori (for this reason, it is called a singular Lagrangian torus fibration). For $p$-adic integrable systems the situation is more complicated to describe in general, even for concrete examples (such as the Jaynes-Cummings model treated in our paper \cite{CrePel-JC}). However, there is a common point:

\begin{proposition}
	Let $n$ be a positive integer. \letpprime. Let $(M,\omega)$ be a $2n$-dimensional $p$-adic symplectic manifold. Let $F:(M,\omega)\to(\Qp)^n$ be a $p$-adic analytic integrable system. Suppose that the components of $F$ are either regular components $\xi_i$ or given by one of the normal forms of Theorem \ref{thm:integrable}. Then the fiber $F^{-1}(0)$ is a union of isotropic subspaces intersecting at the origin, and if all the components are regular, then the fiber is a Lagrangian subspace.
\end{proposition}

\begin{proof}
	For the regular case, where the system is $(\xi_1,\ldots,\xi_n)$, clearly the fiber of $0$ is a Lagrangian subspace. Otherwise, it is enough to prove the statement for the dimension $2$ and $4$ normal forms, and the conclusion follows by multiplying, because the symplectic form $\omega_0$ does not mix the coordinates of different components. For the normal forms, the result reduces to checking that
	\[\omega_0((\pm d_1,1,0,0),(0,0,\pm d_2,1))=\omega_0((1,0,0,0),(0,0,1,0))\]\[=\omega_0((\pm d,1,0,0),(0,0,1,0))=0.\qedhere\]
\end{proof}

From the point of view of integrable systems and symplectic singularity theory, the fibers and images of the local models in the $p$-adic case are very interesting and include for example the images displayed in Figures \ref{fig:images1} and \ref{fig:images2} and the fibers displayed in Figures \ref{fig:fibers1} and \ref{fig:fibers2}.

\begin{figure}
	\begin{center}
		\begin{tabular}{ccc}
			\begin{tikzpicture}
				\fill[green] (-1,-1)--(-1,1)--(1,1)--(1,-1);
				\axes
			\end{tikzpicture} &
			\begin{tikzpicture}
				\fill[green] (0,-1)--(0,1)--(1,1)--(1,-1);
				\axes
				\draw[line width=2,blue] (0,-1)--(0,1);
			\end{tikzpicture} &
			\begin{tikzpicture}
				\fill[green] (-1,-1)--(-1,1)--(1,1)--(1,-1);
				\axes
				\draw[line width=2,blue] (0,-1)--(0,1);
			\end{tikzpicture} \\
			Regular & Transversally & Transversally \\
			& elliptic & hyperbolic
		\end{tabular}
		
		\begin{tabular}{cccc}
			\begin{tikzpicture}
				\fill[green] (0,0)--(0,1)--(1,1)--(1,0);
				\axes
				\draw[line width=2,blue] (1,0)--(0,0)--(0,1);
				\fill[red] (0,0) circle (0.05);
			\end{tikzpicture} &
			\begin{tikzpicture}
				\fill[green] (0,-1)--(0,1)--(1,1)--(1,-1);
				\axes
				\draw[line width=2,blue] (0,-1)--(0,1);
				\draw[line width=2,blue] (0,0)--(1,0);
				\fill[red] (0,0) circle (0.05);
			\end{tikzpicture} &
			\begin{tikzpicture}
				\fill[green] (-1,-1)--(-1,1)--(1,1)--(1,-1);
				\axes
				\draw[line width=2,blue] (0,-1)--(0,1);
				\draw[line width=2,blue] (-1,0)--(1,0);
				\fill[red] (0,0) circle (0.05);
			\end{tikzpicture} &
			\begin{tikzpicture}
				\fill[green] (-1,-1)--(-1,1)--(1,1)--(1,-1);
				\axes
				\fill[red] (0,0) circle (0.05);
			\end{tikzpicture} \\
			Elliptic- & Elliptic- & Hyperbolic- & Focus-focus \\
			elliptic & hyperbolic & hyperbolic & 
		\end{tabular}
	\end{center}
	\caption{Images of some normal forms for the real case and the $p$-adic case with $p\not\equiv 1\mod 4$. In the real case, the positive and negative sides of the axes represent, as usual, positive and negative numbers; if $p=2$, the ``positive'' and ``negative'' sides represent numbers whose second digit is $0$ and $1$, respectively; finally, if $p\equiv 3\mod 4$, the ``positive'' and ``negative'' sides represent even-order and odd-order numbers, respectively. (The points on the axes themselves have, as usual, a zero coordinate.) In each drawing, the green region represents regular values, the blue points are rank $1$ critical values, and the red points are rank $0$ critical values. An elliptic component has $c=1$ in the notation of Corollary \ref{cor:williamson}, a hyperbolic one has $c=-1$, and a focus-focus one has $c=-1$ in part (2) of Theorem \ref{thm:integrable}.}
	\label{fig:images1}
\end{figure}
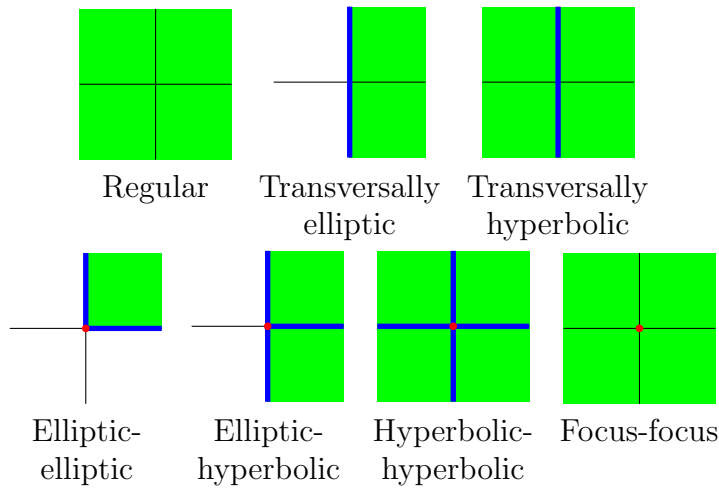

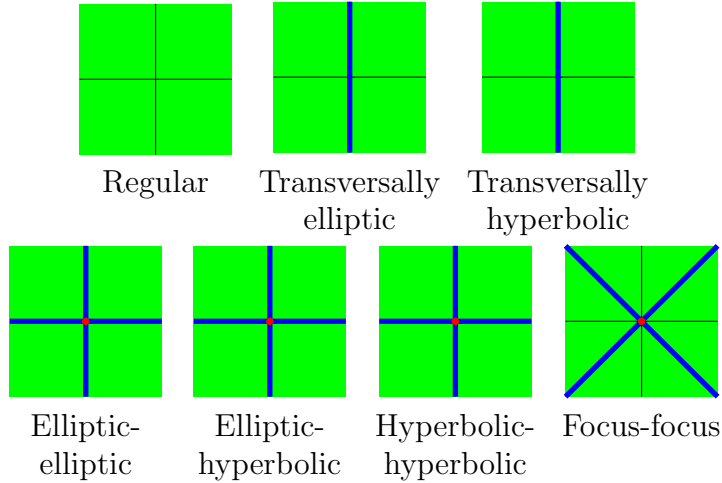
\begin{figure}
	\begin{center}
		\begin{tabular}{ccc}
			\begin{tikzpicture}
				\fill[green] (-1,-1)--(-1,1)--(1,1)--(1,-1);
				\axes
			\end{tikzpicture} &
			\begin{tikzpicture}
				\fill[green] (-1,-1)--(-1,1)--(1,1)--(1,-1);
				\axes
				\draw[line width=2,blue] (0,-1)--(0,1);
			\end{tikzpicture} &
			\begin{tikzpicture}
				\fill[green] (-1,-1)--(-1,1)--(1,1)--(1,-1);
				\axes
				\draw[line width=2,blue] (0,-1)--(0,1);
			\end{tikzpicture} \\
			Regular & Transversally & Transversally \\
			& elliptic & hyperbolic
		\end{tabular}
		
		\begin{tabular}{cccc}
			\begin{tikzpicture}
				\fill[green] (-1,-1)--(-1,1)--(1,1)--(1,-1);
				\axes
				\draw[line width=2,blue] (0,-1)--(0,1);
				\draw[line width=2,blue] (-1,0)--(1,0);
				\fill[red] (0,0) circle (0.05);
			\end{tikzpicture} &
			\begin{tikzpicture}
				\fill[green] (-1,-1)--(-1,1)--(1,1)--(1,-1);
				\axes
				\draw[line width=2,blue] (0,-1)--(0,1);
				\draw[line width=2,blue] (-1,0)--(1,0);
				\fill[red] (0,0) circle (0.05);
			\end{tikzpicture} &
			\begin{tikzpicture}
				\fill[green] (-1,-1)--(-1,1)--(1,1)--(1,-1);
				\axes
				\draw[line width=2,blue] (0,-1)--(0,1);
				\draw[line width=2,blue] (-1,0)--(1,0);
				\fill[red] (0,0) circle (0.05);
			\end{tikzpicture} &
			\begin{tikzpicture}
				\fill[green] (-1,-1)--(-1,1)--(1,1)--(1,-1);
				\axes
				\draw[line width=2,blue] (-1,-1)--(1,1);
				\draw[line width=2,blue] (-1,1)--(1,-1);
				\fill[red] (0,0) circle (0.05);
			\end{tikzpicture} \\
			Elliptic- & Elliptic- & Hyperbolic- & Focus-focus \\
			elliptic & hyperbolic & hyperbolic & 
		\end{tabular}
	\end{center}
	\caption{Images of the same normal forms for $p\equiv 1\mod 4$. The slopes of the blue lines in the last image are $\ii$ and $-\ii$. Note that the images of the systems with the same rank coincide, except perhaps for a coordinate change, because now the singular components of the systems belong to the same class: part (1) of Theorem \ref{thm:integrable}, with $c_1=c_2=1$.}
	\label{fig:images2}
\end{figure}

We could define a $p$-adic version of the ``Williamson type'' defined for the real case at \cite[p. 41]{BolFom}. In the real case, it consists of a tuple of integers $(k_\e,k_\h,k_\f)$ that count the number of elliptic, hyperbolic and focus-focus components of the normal form. The problem with this approach is that the components of the normal forms are associated to blocks in the normal forms of matrices, which in the real case take only three possible forms. In the $p$-adic case, by Theorem \ref{thm:num-forms}, there can appear countably many different blocks, so the Williamson type will be a sequence (with a finite number of elements different from zero) instead of a tuple.

For example, for $p=2$, the first $11$ elements of the sequence count the number of components with each possible $c$ in Corollary \ref{cor:williamson} (associated to blocks of size two), the next $145$ elements count the number of pairs of components with each possible form in parts (2) and (3) of Theorem \ref{thm:integrable} (associated to blocks of size four), after which come the counts of trios of components associated to blocks of size six, and so on.

We close this section with a mention of the Eliasson-Vey's linearization theorem \cite{Eliasson-thesis,Eliasson,Russmann,Vey,VuNWac}, which in the real case states that any \emph{smooth} integrable system can be brought to its Williamson normal form by a symplectomorphism. The \emph{analytic} case of this theorem is due to R\"u\ss mann \cite{Russmann} for two degrees of freedom and Vey \cite{Vey} in arbitrary dimension. In the real case Eliasson's Theorem (assuming that there are no hyperbolic components) says that there is a local diffeomorphism $\varphi$ and symplectic coordinates $\phi^{-1}=(x,\xi,y,\eta)$ such that $F\circ\phi=\varphi(g_1,g_2)$, where $g_i$ is one of the elliptic, real or focus-focus models. The $p$-adic equivalent of this theorem is well beyond the scope of this paper, and we state it as a question.

\begin{question}[A $p$-adic Eliasson-Vey's theorem?]\label{question:eliasson}
	Let $n$ be a positive integer. \letpprime. Given a $2n$-dimensional $p$-adic analytic symplectic manifold $(M,\omega)$, an integrable system $F:(M,\omega)\to (\Qp)^n$ and a non-degenerate critical point $m$ of $F$, determine under which conditions on the Williamson type of the critical point $m$ there are open sets $U\subset M$ and $V\subset (\Qp)^{2n}$, a $p$-adic analytic symplectomorphism $\phi:V\to U$ and a local diffeomorphism $\varphi$ of $(\Qp)^n$ such that $\phi(0)=m$ and
	\[(F-F(m))\circ \phi=\varphi\circ(g_1,\ldots,g_n),\]
	where $(g_1,\ldots,g_n)$ is the Williamson normal form of $F$ in $m$. (In the real case it is enough that there are no hyperbolic blocks.)
\end{question}

\section{Application to classical mechanical systems}\label{sec:JC}

In this section we explain how Theorem \ref{thm:integrable} can be applied to further study the $p$-adic Jaynes-Cummings model introduced and studied in \cite{CrePel-JC}. We recommend the books by Abraham-Marsden \cite{AbMa}, de León-Rodrigues \cite{LeRo} and Marsden-Ratiu \cite{MarRat} for an introduction to the mathematical study of mechanics and its connections to symplectic/differential geometry.

In our paper \cite{CrePel-JC} we studied the Jaynes-Cummings model with $p$-adic coefficients (see Figure \ref{fig:real-JC-image} for the fibers in the real case and Figure \ref{fig:padic-JC-image} for a fiber in the $p$-adic case). The system was defined therein, in analogy with the real case, as follows. For any number $p$, first we consider the product $p$-adic analytic manifold $\mathrm{S}_p^2\times(\Qp)^2$ with the $p$-adic symplectic form $\omega_{\sphere}+\dd u\wedge\dd v$. Here we recall that
\[\sphere=\Big\{(x,y,z)\in\Qp^2:x^2+y^2+z^2=1\Big\}\]
and $\omega_{\sphere}$ is the area form in the sphere given by
\[\omega_{\sphere}=-\frac{1}{z}\dd x\wedge\dd y=\frac{1}{y}\dd x\wedge\dd z=-\frac{1}{x}\dd y\wedge\dd z.\]
The \emph{$p$-adic Jaynes-Cummings model} is given by the $p$-adic analytic map \[F=(J,H):\mathrm{S}_p^2\times(\Qp)^2\to(\Qp)^2,\] where
\[\left\{\begin{aligned}
	J(x,y,z,u,v) & = \frac{u^2+v^2}{2}+z; \\
	H(x,y,z,u,v) & = \frac{ux+vy}{2},
\end{aligned}\right.\]
where $(x,y,z)\in\mathrm{S}_p^2$ and $(u,v)\in(\Qp)^2$.

\begin{figure}
	\begin{tikzpicture}[scale=2]
		\node (im) at (0.8,0) {\includegraphics[width=8cm]{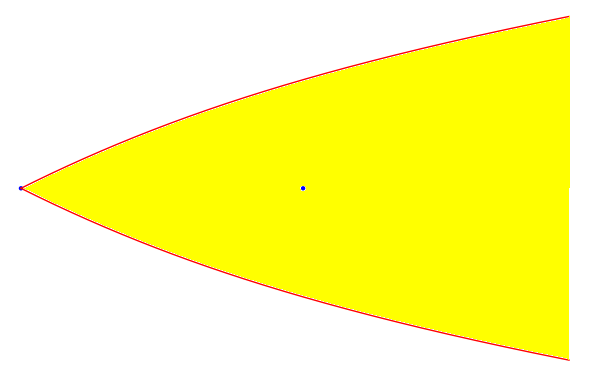}};
		\draw (-1.05,0)--(-1.5,0.7);
		\fill[blue] (-1.6,0.8) circle (0.02);
		\draw (-0.5,0.25) -- (-0.8,0.9);
		\draw[blue] (-0.8,1.5) ellipse (0.3 and 0.5);
		\draw (0.85,0) -- (0.85,1.5);
		\node (foco) at (0.85,3) {\includegraphics[height=6cm,trim=20cm 0 20cm 0,clip]{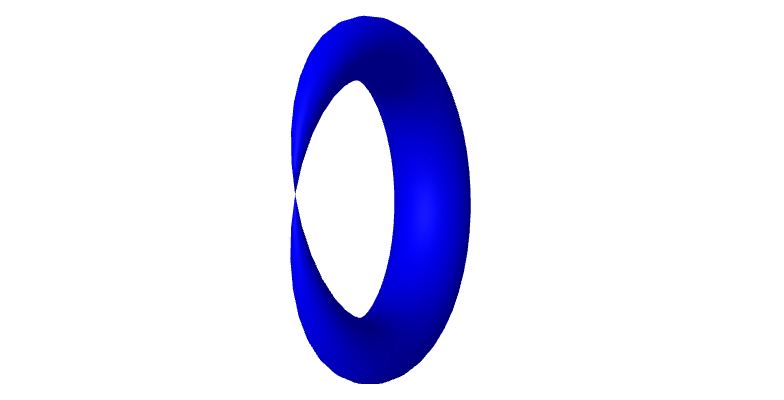}};
		\draw (2,0.5) -- (2.4,1.5);
		\node (toro) at (2.6,3) {\includegraphics[height=6cm,trim=19cm 0 19cm 0,clip]{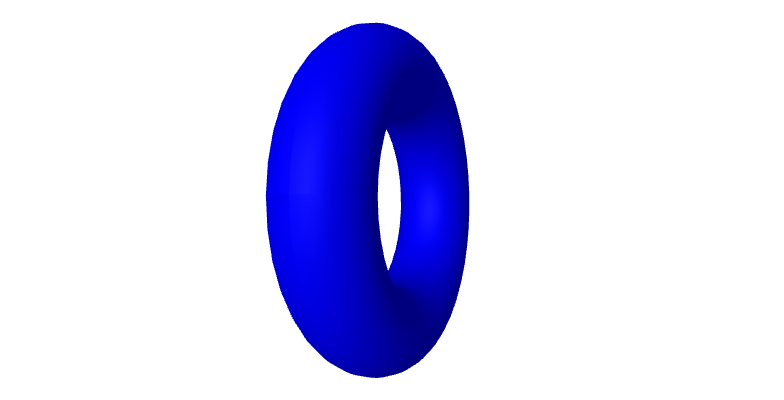}};
	\end{tikzpicture}
	\caption{Image and fibers of the real Jaynes-Cummings model. The red curve consists of rank $1$ critical points, and the two blue points are rank $0$. The Jaynes\--Cummings model is an example of a class of integrable systems called \emph{semitoric systems}. The fibers of this system are a point, circles, $2$-tori (generic fiber) and a pinched torus.}
	\label{fig:real-JC-image}
\end{figure}

\begin{figure}
	\includegraphics{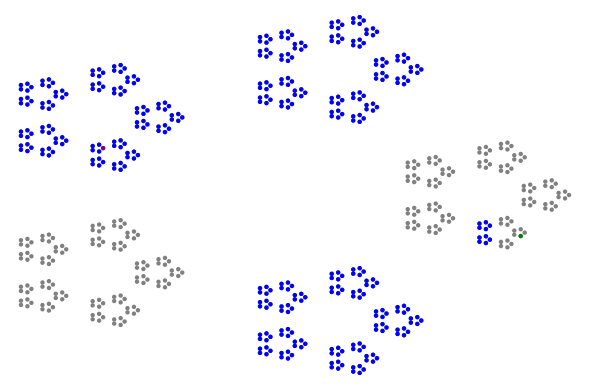}
	\caption{Fiber $F^{-1}(72,1)$ of the $5$-adic Jaynes-Cummings model. The blue points are values of $z$ for which the coordinates $(x,y,u,v)$ form two $p$-adic circles, at the green points they only form one circle, at the purple point $z=j=72$ they have dimension $1$ but they are not a circle, and the values of $z$ which appear in grey are not in the fiber.}
	\label{fig:padic-JC-image}
\end{figure}

By \cite[Proposition 2.5]{CrePel-JC}, at $m_1=(0,0,-1,0,0)$, there is a $p$-adic local symplectomorphism changing the local coordinates to $(x,\xi,y,\eta)$ in which the $p$-adic symplectic form is given by $\omega=(\dd x\wedge\dd\xi+\dd y\wedge\dd\eta)/2$ and
\begin{equation}
	F_1(x,\xi,y,\eta)=\frac{1}{2}(x^2+\xi^2,y^2+\eta^2)+\ocal((x,\xi,y,\eta)^3).\label{eq:JC-elliptic}
\end{equation}
Here $F_1=B\circ(F-F(0,0,-1,0,0))$ with
\[B=\begin{pmatrix}
	1 & 2 \\
	1 & -2
\end{pmatrix}.\]
At $m_2=(0,0,1,0,0)$, after the change, the $p$-adic symplectic form is also given by $\omega=(\dd x\wedge\dd\xi+\dd y\wedge\dd\eta)/2$ and
\begin{equation}
	F_2(x,\xi,y,\eta)=(x\eta-y\xi,x\xi+y\eta)+\ocal((x,\xi,y,\eta)^3).\label{eq:JC-focus}
\end{equation}
Here $F_2=B\circ(F-F(0,0,1,0,0))$ with
\[B=\begin{pmatrix}
	2 & 0 \\
	0 & 4
\end{pmatrix}.\]

Hence we have the following consequence of Theorem \ref{thm:integrable}, whose proof can be found in \cite[Section 8]{CrePel-matrix}.

\begin{corollary}\label{cor:JC}
	\letpprime. Then there exist open sets $U_1$ and $U_2$ such that $m_1\in U_1$ and $m_2\in U_2$ and a local symplectomorphism $\phi:U_1\to U_2$ centered at $m_1$, such that \[F_2(\phi(x,\xi,y,\eta))=F_1(x,\xi,y,\eta)+\ocal((x,\xi,y,\eta)^3)\] for $(x,\xi,y,\eta)\in U_1$, if and only if $p\equiv 1\mod 4$, where $F_1$ and $F_2$ are as described in \eqref{eq:JC-elliptic} and \eqref{eq:JC-focus}.
\end{corollary}

More information about the Jaynes-Cummings model and other models of interest in physics can be found at \cite{AbMa,LeRo}.

\section{Examples}\label{sec:examples}

In this section we show examples which illustrate our theorems, so that they can be understood more concretely. Although Theorems \ref{thm:integrable}, \ref{thm:num-integrable} and \ref{thm:num-integrable2} completely determine the classification of critical points of integrable systems, in order to find the explicit class of a given system, we also need the methods used to prove the matrix classification in part II. Hence the resolutions of these examples can be found in \cite[Section 9]{CrePel-matrix}.

\begin{example}
	Consider the following functions with a critical point in the origin (the first one is non-degenerate and the second one is degenerate):
	\[f_1(x,\xi,y,\eta)=\frac{x^2}{2}+2x\xi+3xy+4x\eta+\frac{5}{2}\xi^2+6\xi y+7\xi\eta+4y^2+9y\eta+5\eta^2\]
	\[f_2(x,\xi,y,\eta)=x^2+6x\xi-2xy-3x\eta+\frac{11}{2}\xi^2+\xi y-5\xi\eta-3y^2-2y\eta+3\eta^2\]
	
	The class of the critical point of $f_1$ is that of Corollary \ref{cor:williamson4}(3), with $c=3,t_1=t_2=-2,a=-1,b=0$, for $p=2$, that of part (2) with $c=-1$ for $p=3$, and that of part (1) with $c_1=1$ and $c_2=2$ for $p=5$.
	
	The class of the critical point of $f_2$ is that of Corollary \ref{cor:williamson4-deg}(3) with $c=5$ for $p=5$, and that of part (2) for $p=11$.
\end{example}

\begin{example}
	Let $p$ be a prime number such that $p\equiv 1 \mod 4$. Then the functions
	\[F_{\mathrm{ee}},F_{\mathrm{eh}},F_{\mathrm{hh}},F_{\mathrm{ff}}:((\Qp)^4,\omega_0)\to(\Qp)^2\]
	and the symplectic form $\omega_0=\dd x\wedge\dd\xi+\dd y\wedge\dd\eta$ on $\Qp^4$:
	\begin{itemize}
		\item Elliptic-elliptic: $F_{\mathrm{ee}}(x,\xi,y,\eta)=(\frac{x^2+\xi^2}{2},\frac{y^2+\eta^2}{2})$;
		\item Elliptic-hyperbolic: $F_{\mathrm{eh}}(x,\xi,y,\eta)=(\frac{x^2+\xi^2}{2},y\eta)$;
		\item Hyperbolic-hyperbolic: $F_{\mathrm{hh}}(x,\xi,y,\eta)=(x\xi,y\eta)$;
		\item Focus-focus: $F_{\mathrm{ff}}(x,\xi,y,\eta)=(x\eta-y\xi,x\xi+y\eta)$,
	\end{itemize}
	are non-degenerate $p$-adic analytic integrable systems. Furthermore, all four systems are $p$-adically locally symplectomorphic. This follows from Theorem \ref{thm:integrable}.
\end{example}

\begin{example}
	\letpprime\ such that $p \not \equiv 1 \mod 4$. Then any two distinct systems among those four are not locally symplectomorphic. This follows from Theorem \ref{thm:integrable}.
\end{example}

\begin{remark}
	In Corollary \ref{cor:williamson}, the elliptic function corresponds to $c=1$. The hyperbolic one corresponds to $c=1$ for $p\equiv 1\mod 4$ and $c=-1$ otherwise. That is, six of the seven forms for $p\equiv 1\mod 4$, three of the five for $p\equiv 3\mod 4$, and nine of the eleven for $p=2$, have no real equivalent.
	
	In Theorem \ref{thm:integrable}, the elliptic-elliptic model corresponds to (1) with $c_1=c_2=1$. Changing elliptic components to hyperbolic results in changing the corresponding $c_i$ to $-1$, except if $p\equiv 1\mod 4$, where there is no change. The focus-focus model is the same for $p\equiv 1\mod 4$, and otherwise it is (2) with $c=-1$. \emph{The vast majority of $p$-adic normal forms, including all those in point (3), have no real equivalent.}
\end{remark}

We now comment on the $\Circle$-actions on $\Qp^2$ induced by those systems (we refer to \cite[Appendix C]{CrePel-JC} for a review of the concept of $p$-adic action), where we recall that $\Circle$ is defined by
\[\Circle=\Big\{(x,y)\in\Qp^2:x^2+y^2=1\Big\}.\]

\begin{remark}\label{rem:actions}
	\letpprime.
	\begin{enumerate}
		\item The elliptic function $f_{\mathrm{e}}(x,\xi)=\frac{x^2+\xi^2}{2}$ induces on $\Qp^2$ an $\Circle$-action given by
		\[(u,v)\cdot(x,\xi)=\begin{pmatrix}
			u & v \\
			-v & u
		\end{pmatrix}
		\begin{pmatrix}
			x \\ \xi
		\end{pmatrix},\]
		for $(u,v)\in\Circle$ and $(x,\xi)\in\Qp^2$.
		\item The first component $f_1(x,\xi,y,\eta)=x\eta-y\xi$ of the focus-focus system $F_\mathrm{ff}$ induces on $\Qp^4$ an action with a similar formula to the previous one, simultaneously on the plane $(x,y)$ and the plane $(\xi,\eta)$:
		\[(u,v)\cdot(x,\xi,y,\eta)=\begin{pmatrix}
			u & 0 & -v & 0 \\
			0 & u & 0 & -v \\
			v & 0 & u & 0 \\
			0 & v & 0 & u
		\end{pmatrix}
		\begin{pmatrix}
			x \\ \xi \\ y \\ \eta
		\end{pmatrix},\]
		for $(u,v)\in\Circle$ and $(x,\xi,y,\eta)\in\Qp^4$.
	\end{enumerate}
	
	Indeed, by \cite[Corollary 4.5]{CrePel-JC}, there is a subgroup of $\Circle$ isomorphic to $p\Zp$ by the correspondence $t\mapsto(\cos t,\sin t)$ that contains all elements near the origin. We need to prove that, if $\psi_t(x,\xi)=(\cos t,\sin t)\cdot(x,\xi)$, the vector field $X_t$ of this flow (in the sense that $\frac{\dd}{\dd t}\psi_t(x,\xi)=X_t(\psi_t(x,\xi))$) satisfies Hamilton's equations $\imath_{X_t}\omega_0=\dd f_{\mathrm{e}}$.
	
	We have
	\[\frac{\dd}{\dd t}\psi_t(x,\xi)=\frac{\dd}{\dd t}\begin{pmatrix}
		\cos t & \sin t \\
		-\sin t & \cos t
	\end{pmatrix}
	\begin{pmatrix}
		x \\ \xi
	\end{pmatrix}=
	\begin{pmatrix}
		-\sin t & \cos t \\
		-\cos t & -\sin t
	\end{pmatrix}
	\begin{pmatrix}
		x \\ \xi
	\end{pmatrix}\]
	so $X_t(x,\xi)=(\xi,-x)$, and \[\imath_{X_t}\omega_0=x\dd x+\xi\dd\xi=\dd f_{\mathrm{e}},\] as we wanted.
	
	For the focus-focus action, we have analogously that $X_t(x,\xi,y,\eta)=(-y,-\eta,x,\xi)$, and \[\imath_{X_t}\omega_0=\eta\dd x-y\dd\xi-\xi\dd y+x\dd\eta=\dd f_1,\] as we wanted.
\end{remark}

\section{Circular symmetries of the $p$-adic models}\label{sec:circle}

Here we generalize the content of Remark \ref{rem:actions} and analyze the problem of existence of circle actions for arbitrary models. In the real case, for a fixed symplectic space $(V,\omega)$, most multiples of a Hamiltonian that admits a circle action do not admit it. This happens because a smooth circle action over $V$ is defined by a smooth map $h:\mathrm{S}^1\times V\to V$, satisfying $h(g_1,h(g_2,m))=h(g_1g_2,m)$ and $h(1,m)=m$, for $g_1,g_2\in \mathrm{S}^1$ and $m\in V$. Concretely, considering $g:\R\to\mathrm{S}^1$ given by $g(t)=(\cos t,\sin t)$, we must have
\[h(g(t+t'),m)=h(g(t)g(t'),m)=h(g(t),h(g(t'),m))\]
and $h(1,m)=m$, for $t,t'\in\R$ and $m\in V$. The induced vector field for this action, that is, $X_h(t)$ given by
\[\frac{\dd}{\dd t}h(g(t),m)=X_h(t)(h(g(t),m)),\]
must coincide with $X_f$ induced by the Hamiltonian $f$ via the Hamilton equation $\imath_{X_f}\omega_0=\dd f$. When $f$ is multiplied by a constant $k$, that is, $f'=kf$, we also have $X_{f'}=kX_f$, so $X_h$ must also be multiplied by $k$:
\begin{align*}
	\frac{\dd}{\dd t}h'(g(t),m) & =X_{h'}(t)(h'(g(t),m)) \\
	& =kX_{h}(t)(h'(g(t),m)) \\
	& =kX_{h}(kt)(h'(g(t),m))
\end{align*}
(the last equality happens because $X_h(t)=X_f$ is independent of $t$). This equation is solved by $h'(g(t),m)=h(g(kt),m)$, that is, the action is accelerated by a factor of $k$. As we must have $h'(g(2\pi),m)=h'(1,m)=m$, in general this is only possible if $k$ is integer.

This will not happen in the $p$-adic case, because the $p$-adic circle is not closed (there is no $t\ne 0$ such that $g(t)=1$). So getting a circle action is much easier in this case, and actually any ``small enough'' multiple of a Hamiltonian admits a circle action.

\begin{proposition}\label{prop:circle-actions}
	Let $n$ be a positive integer and let $p$ be a prime number. Let $\Omega_0$ be the matrix of the standard symplectic form on $(\Qp)^{2n}$. Given a $p$-adic analytic Hamiltonian $f:(\Qp)^{2n}\to \Qp$ such that $f(m)=m^TMm/2$, for a matrix $M\in\M_{2n}(\Qp)$, $f$ admits a $p$-adic analytic $\Circle$-action (that is, there exists $h:\Circle\times(\Qp)^{2n}\to(\Qp)^{2n}$ analytic such that $h(g_1,h(g_2,m))=h(g_1g_2,m)$ and $h(1,m)=m$, for $g_1,g_2\in \Circle$ and $m\in (\Qp)^{2n}$) if and only if $\ord(\lambda)\ge 0$ for all $\lambda$ which is an eigenvalue of $\Omega_0^{-1}M$.
\end{proposition}

\begin{proof}
	Suppose that $f$ admits a circle action $h$ and let $\psi(t,v)=h(g(t),v)$ and $A=\Omega_0^{-1}M$. Hamilton's equation $\imath_{X_f}\omega_0=\dd f$ results in
	\[X_f(m)^T\Omega_0 v=\dd f(m)(v)=m^TMv\Rightarrow X_f(m)=-\Omega_0^{-1}Mm=-Am.\]
	Substituting in the flow equation,
	\[\frac{\dd}{\dd t}\psi(t,m)=-A\psi(t,m)\]
	and $\psi(0,m)=m$, which solves as
	\[\psi(t,m)=\exp(-tA)m\]
	where $\exp$ denotes the matrix exponential. This must exist for all $t$ in the domain of $g(t)=(\cos t,\sin t)$, that is, such that $|t|_p\le k_p$, where $k_p=1/p$ for $p\ne 2$ and $k_2=1/4$. The exponential of $-tA$ exists if and only if the eigenvalues $\mu$ of $-tA$ satisfy $|\mu|_p\le k_p$, which implies \[|-t\lambda|_p=|t|_p|\lambda|_p\le k_p.\] As it must exist when $|t|_p=k_p$, we have $|\lambda|_p\le 1$.
	
	Conversely, suppose that $|\lambda|_p\le 1$ for all eigenvalues of $A$. Then, for all $t$ in the domain of $g$, \[|-t\lambda|_p=|t|_p|\lambda|_p\le k_p\] and $\exp(-tA)$ exists. This means that $\psi(t,m)$ is well defined, and $h(g_0,m)$ exists for $g_0\in\im(g)$. By \cite[Corollary 4.5]{CrePel-JC}, the quotient $\Circle/\im(g)$ is a discrete group, so we can define its action arbitrarily without affecting the flow equation, and we have an action of $\Circle$.
\end{proof}

\begin{remark}
	Often symplectic classifications of group actions in equivariant symplectic geometry include the assumption that the action is effective, as it is the case for example in Delzant's classification, Duistermaat-Pelayo \cite{DuiPel-symplectic} and Pelayo \cite{Pelayo-symplectic}. With this restriction (being effective), no multiple of a Hamiltonian which admits a circle action also admits an action. However, as Proposition \ref{prop:circle-actions} shows, in the $p$-adic case all small multiples of any Hamiltonian admit an effective action.
\end{remark}

However, if we want the action to have the form $h(g,m)=h((u,v),m)=(uI+vB)m$, like the actions in Remark \ref{rem:actions}, the situation changes completely.

\begin{proposition}
	Let $n$ be a positive integer and let $p$ be a prime number. Let $\Omega_0$ be the matrix of the standard symplectic form on $(\Qp)^{2n}$. Given a $p$-adic analytic Hamiltonian $f:(\Qp)^{2n}\to \Qp$ such that $f(m)=m^TMm/2$, for a matrix $M\in\M_{2n}(\Qp)$, $f$ admits a $p$-adic analytic $\Circle$-action of the form $h((u,v),m)=(uI+vB)m$ if and only if the eigenvalues of $\Omega_0^{-1}M$ are $\ii$ and $-\ii$, both with multiplicity $n$. In that case, $B=-\Omega_0^{-1}M$.
\end{proposition}

\begin{proof}
	The flow equation implies
	\[\frac{\dd}{\dd t}h((\cos t,\sin t),m)=-Ah((\cos t,\sin t),m),\]
	that is
	\[\frac{\dd}{\dd t}(\cos tI+\sin tB)m=(-\sin tI+\cos tB)m=-A(\cos tI+\sin tB)m,\]
	which implies $AB=I$ and $B=-A$. Hence, the action exists if and only if $A^2=-I$. If this happens, the only possible eigenvalues are $\ii$ and $-\ii$, and since they must come in opposite pairs, each one appears $n$ times. Conversely, if the eigenvalues of $A$ are $\ii$ and $-\ii$, those of $A^2$ are $-1$, which implies that $A^2=-I$.
\end{proof}

Of the normal forms in Theorem \ref{thm:integrable}, those with a circle action of this form are the following:
\begin{itemize}
	\item At point (1), only the elliptic component gives eigenvalues $\ii$ and $-\ii$. The remaining ones have different eigenvalues. Hence, we recover Remark \ref{rem:actions}(1).
	\item At point (2), only the focus-focus component, if $p\not\equiv 1\mod 4$, has eigenvalues $(\ii,\ii,-\ii,-\ii)$. We recover Remark \ref{rem:actions}(2).
	\item At point (3), $t_1+t_2\sqrt{c}$ is the square of an eigenvalue, so we must have $t_1=-1$ and $t_2=0$ to have this kind of circle action.
\end{itemize}

\section{Final remarks}\label{sec:final}

\begin{remark}
	Pelayo-Voevodsky-Warren \cite[Definition 7.1]{PVW} defined $p$-adic integrable systems in a different way: their second condition was that the set where the $n$ differential $1$-forms $\dd f_1,\ldots,\dd f_n$ are linearly dependent has $p$-adic measure zero. In view of some examples, including the Jaynes-Cummings model treated in our previous paper \cite{CrePel-JC}, and which have come to our attention since \cite{PVW} was written, we believe it is convenient to replace that condition, which is too restrictive in the $p$-adic case (unlike in the real case, which is the usual formulation one gives) by the condition that the set where the $n$ differential $1$-forms are linearly independent is dense in $M$. We use this updated definition throughout the paper.
\end{remark}

\begin{remark}
	The choice of $c_0$ as the least quadratic residue in Definition \ref{def:sets} is not essential, in the sense that the normal forms that appear in Theorem \ref{thm:integrable} are ``equivalent'' (in the precise sense that the resulting set of normal forms still contains a representative of each class modulo local symplectomorphism).
\end{remark}

\begin{remark}
	There are many more local models of $p$-adic integrable systems than real ones, so from a physical viewpoint they should be able to model physical or other phenomena beyond the applications with real coefficients, see for example \cite{GKPSW,HuaHov,KOL} and the references therein. For applications in biology see \cite{AveBik,DraMis}.
\end{remark}

\begin{remark}
	Identifying a Hessian with its quadratic form, the formula \eqref{eq:integrable} would be written $\dd^2 F=(g_1,g_2)$. In the expression for $g_1$ and $g_2$ in any of the cases above, if we change the values of the parameters, the resulting functions still form an integrable system, because the concrete values do not affect the Poisson bracket, but they are not a normal form. We can apply the theorem to these new functions, resulting in a new normal form $(g_1',g_2')$ locally symplectomorphic to $(g_1,g_2)$, in the same or a different case.
\end{remark}

\appendix
\section{Basic properties of $p$-adic numbers}\label{sec:padic}

\subsection{Critical points of $p$-adic analytic functions}\label{sec:critical}

A \textit{power series} in $(\Qp)^n$ is given by
$f(x)=\sum_{I\in\N^n}a_I(x-x_0)^I$
where $x^I$ means $x_1^{i_1}\ldots x_n^{i_n}$ and $a_I$ are coefficients in $\Qp$. A function $f:(\Qp)^n\to\Qp$ is \textit{($p$-adic) analytic} if it can be expressed as a power series. (Sometimes the term is used for functions that are expressed as a power series only piecewise, but for our purposes here this distinction does not matter because we are interested only in the local behavior of analytic functions.)

We define $\dd f(m)$ as the vector whose coordinates are the partial derivatives of $f$, defined in the usual sense:
\[\dd f(m)_i=\frac{\partial f(m)}{\partial x_i}:=\lim_{t\rightarrow 0}\frac{f(x+t\e_i)-f(x)}{t}\]
where the definition of limit is the usual one from metric spaces (the limit of $f$ at a point $x_0\in\Qp^n$ is equal to $y_0\in\Qp$ if for any $\epsilon>0$ there is $\delta>0$ such that $|f(x)-f(x_0)|_p<\epsilon$ whenever $|x_i-(x_0)_i|_p<\delta$ for $1\le i\le n$). This limit always exists for a power series, and it is equal to
\[\frac{\partial f}{\partial x_i}=\sum_{I\in\N^n}a_Ii_j(x-x_0)^{I_j},\]
where, for $I=(i_1,\ldots,i_n)$, $I_j$ is defined as $(i_1,\ldots,i_j-1,\ldots,i_n)$.

In the same way, we define the Hessian of $f$, which we denote as $\dd^2 f$, the matrix with the second derivatives of $f$ as entries:
\[(\dd^2 f)_{ij}=\frac{\partial^2 f}{\partial x_i\partial x_j}.\]

In general, if $M$ is a $p$-adic analytic manifold, there are charts $\phi$ from $U\subset(\Qp)^n$ to $V\subset M$ for some $n$, and we can define \textit{analytic function} in $V$ as a function $f:V\to\Qp$ such that $f\circ\phi$ is given by a power series, or more generally \textit{analytic function} in $M$ as $f:M\to\Qp$ given piecewise by power series. The differential of $f$ is defined as $\dd f=\dd(f\circ\phi)$, and its Hessian is $\dd^2 f=\dd^2(f\circ\phi)$. See Appendix B of \cite{CrePel-JC} for the precise definitions of $p$-adic manifolds and analytic functions.

A \textit{critical point} of $f:M\to\Qp$ is a point $m\in M$ such that $\dd f(m)=0$.

\subsection{Application of $p$-adic Darboux's Theorem}

It is relatively easy to prove that any two linear symplectic forms are linearly symplectomorphic:

\begin{theorem}\label{thm:omega0}
	Let $F$ be a field and $n$ be a positive integer. Every symplectic form on $F^{2n}$ is linearly symplectomorphic to the form $\omega_0$ which has as matrix
	\[\Omega_0=\begin{pmatrix}
		0 & 1 & & & & & \\
		-1 & 0 & & & & & \\
		& & 0 & 1 & & & \\
		& & -1 & 0 & & & \\
		& & & & \ddots & & \\
		& & & & & 0 & 1 \\
		& & & & & -1 & 0
	\end{pmatrix}.\]
	Hence, every two linear symplectic forms on $F^{2n}$ are actually symplectomorphic.
\end{theorem}

\begin{proof}
	What we want to prove is that there is a basis with respect to which $\omega$ has the matrix $\Omega_0$. Then, the symplectomorphism that sends this basis to the canonical one will send $\omega$ to $\omega_0$.
	
	We go by induction on $n$. As $\omega$ is non-degenerate, there are $u_1$ and $v_1$ with $\omega(u_1,v_1)=1$. Of course, $\langle u_1,v_1\rangle$ is symplectic, so its complement $\langle u_1,v_1\rangle^\omega$ is also symplectic. Applying induction to this complement, we get a basis $\{u_2,v_2,\ldots,u_n,v_n\}$. Now, $\{u_1,v_1,\ldots,u_n,v_n\}$ is the basis we are looking for.
\end{proof}

In our recent paper \cite[Theorem B]{CrePel-Darboux}, we extended this result proving that it is always possible to choose local symplectic coordinates around a point in a $p$-adic analytic symplectic manifold, that is, we prove a $p$-adic analog of Darboux's Theorem in real symplectic geometry and use it to derive a strong global classification \cite[Theorem D]{CrePel-Darboux} which is a symplectic version of a result of Serre \cite[Th\'eor\`eme 1]{Serre}.

\begin{theorem}[$p$-adic analytic Darboux's Theorem {\cite[Theorem B]{CrePel-Darboux}}]\label{thm:darboux}
	Let $n$ be a positive integer. \letpprime. Let $(M,\omega)$ be a $2n$-dimensional $p$-adic analytic symplectic manifold and $m\in M$. There exist local $p$-adic analytic coordinates $(x_1,y_1,\ldots,x_n,y_n)$ around $m$ such that $\omega=\sum_{i=1}^n\dd x_i\wedge\dd y_i.$
\end{theorem}

This means that all the normal forms in the present paper can be extended to have the standard form for the symplectic form at a whole neighborhood of the critical point, instead of only at the critical point.

\section{The real Weierstrass-Williamson classification}\label{sec:real}

By the work of Weierstrass \cite{Weierstrass} any real symmetric positive definite matrix is diagonalizable by a symplectic matrix. This was generalized in an influential paper \cite{Williamson} by Williamson from 1936, where he shows that any symmetric matrix is reducible to a normal form by a symplectic matrix, and gives a classification of all matrix normal forms. In Williamson's approach, matrix reductions take place in the base field. We will recover the real Weierstrass-Williamson classification with a genuinely different strategy: \emph{we will lift the problem to suitable extension fields} where the solution is simpler.

Our proof of the Weierstrass-Williamson classification of $p$-adic matrices (which we shall give in \cite{CrePel-matrix}) is different from Williamson's proof and in particular gives another proof (self-contained, while Williamson's is not, as his proof relies on some applications of other substantial works which he cites in his paper \cite{Williamson}) of the classical (real) Weierstrass-Williamson classification in any dimension.

\begin{figure}
	\fbox{\begin{minipage}{\linewidth}\centering
			\vspace{10pt}
			\fbox{\begin{minipage}{0.8\linewidth}\centering
					\vspace{10pt}
					\fbox{\begin{minipage}{0.8\linewidth}\centering
							\vspace{10pt}
							\fbox{\begin{minipage}{0.8\linewidth}\centering
									\vspace{10pt}
									\textbf{Invertible and without multiple eigenvalues}
									
									Possible blocks: $M_\h(1,r,0),M_\e(1,r,1)$ and $M_{\f\f}(1,r,s)$, with $r,s\in\R, r,s\ne 0$
									
									All blocks are different
									\vspace{10pt}
							\end{minipage}}
							
							\vspace{10pt}
							\textbf{Invertible and diagonalizable}
							
							Possible blocks: $M_\h(1,r,0),M_\e(1,r,1)$ and $M_{\f\f}(1,r,s)$, with $r,s\in\R, r,s\ne 0$
							
							Blocks may be repeated
							\vspace{10pt}
					\end{minipage}}
					
					\vspace{10pt}
					\textbf{Invertible}
					
					Possible blocks: $M_\h(k,r,0),M_\e(k,r,a)$ and $M_{\f\f}(k,r,s)$, with $k\in\N,a\in\{-1,1\},r,s\in\R, r,s\ne 0$
					
					Blocks may be repeated
					\vspace{10pt}
			\end{minipage}}
			
			\vspace{10pt}
			\textbf{General}
			
			Possible blocks: $M_\h(k,r,0),M_\h(k,0,a),M_\e(k,r,a)$ and $M_{\f\f}(k,r,s)$, with $k\in\N,a\in\{-1,1\},r,s\in\R$
			
			Blocks may be repeated
			\vspace{10pt}
	\end{minipage}}
	\caption{Hierarchy of degeneracy levels of real matrices, according to the properties of the block decomposition of their normal forms (Theorems \ref{thm:williamson-real} and \ref{thm:williamson-real2}).}
	\label{fig:properties}
\end{figure}

In the real case, if the eigenvalues of $\Omega_0^{-1}M$ are different, the blocks up to dimension $4$ are enough to classify the matrix. Actually, a weaker condition is sufficient: see Theorem \ref{thm:williamson-real}. If the matrix is not diagonalizable or not invertible, the size of the blocks is not limited to $2$ or $4$, but instead can grow indefinitely: see Theorem \ref{thm:williamson-real2}. See Figure \ref{fig:properties} for a hierarchy of properties of the decomposition.

The explicit form of the blocks is as follows:

\begin{definition}\label{def:blocks}
	A \emph{diagonal block of hyperbolic type} is any matrix of the form
	\[M_\h(k,r,a)=\begin{pmatrix}
		& r &   &   &   & & \\
		r &   & 1 &   &   & & \\
		& 1 &   & r &   & & \\
		&   & r &   & \ddots & & \\
		&   &   & \ddots &   & 1 & \\
		&   &   &   & 1 &   & r \\
		&   &   &   &   & r & a
	\end{pmatrix},\]
	for some positive integer $k$, $r\in\R$ and $a\in\{-1,0,1\}$ with $a=0$ if $r\ne 0$, and which has a total of $2k$ rows. A \emph{diagonal block of elliptic type} is any matrix of the form
	\[M_\e(k,r,a)=\begin{pmatrix}
		M_{\e1}(r) & M_{\e2}(1,a) & & & \\
		M_{\e2}(1,a) & M_{\e1}(r) & M_{\e2}(2,a) & & \\
		& M_{\e2}(2,a) & \ddots & & \\
		& & & M_{\e1}(r) & M_{\e2}'(\ell,a) \\
		& & & M_{\e2}'(\ell,a)^T & M_{\e1}'(r)
	\end{pmatrix}\]
	if $k=2\ell+1$ is odd, and
	\[M_\e(k,r,a)=\begin{pmatrix}
		M_{\e1}(r) & M_{\e2}(1,a) & & & \\
		M_{\e2}(1,a) & M_{\e1}(r) & M_{\e2}(2,a) & & \\
		& M_{\e2}(2,a) & \ddots & & \\
		& & & M_{\e1}(r) & M_{\e2}(\ell-1,a) \\
		& & & M_{\e2}(\ell-1,a) & M_{\e1}(r)+M_{\e2}(\ell,a)
	\end{pmatrix}\]
	if $k=2\ell$ is even, for some positive integer $k$, $r\in\R$ and $a\in\{-1,1\}$, and which has a total of $2k$ rows. A \emph{diagonal block of focus-focus type} is any matrix of the form
	\[M_{\f\f}(k,r,s)=\begin{pmatrix}
		M_{\f\f1}(r,s) & M_{\e2}(1,1) & & & \\
		M_{\e2}(1,1) & M_{\f\f1}(r,s) & M_{\e2}(1,1) & & \\
		& M_{\e2}(1,1) & \ddots & & \\
		& & & M_{\f\f1}(r,s) & M_{\e2}(1,1) \\
		& & & M_{\e2}(1,1) & M_{\f\f1}(r,s)
	\end{pmatrix},\]
	for some positive integer $k$ and $r,s\in\R$, and which has a total of $4k$ rows. In the previous blocks the following sub-blocks are used:
	\[
	M_{\e1}(r)=\begin{pmatrix}
		0 & 0 & 0 & r \\
		0 & 0 & -r & 0 \\
		0 & -r & 0 & 0 \\
		r & 0 & 0 & 0
	\end{pmatrix},
	M_{\e1}'(r)=\begin{pmatrix}
		r & 0 \\
		0 & r
	\end{pmatrix},
	M_{\f\f1}(r,s)=\begin{pmatrix}
		0 & s & 0 & r \\
		s & 0 & -r & 0 \\
		0 & -r & 0 & s \\
		r & 0 & s & 0
	\end{pmatrix},\]
	\[M_{\e2}(j,a)=\begin{pmatrix}
		a & 0 & 0 & 0 \\
		0 & 0 & 0 & 0 \\
		0 & 0 & a & 0 \\
		0 & 0 & 0 & 0
	\end{pmatrix}\text{ if $j$ is odd},
	\begin{pmatrix}
		0 & 0 & 0 & 0 \\
		0 & a & 0 & 0 \\
		0 & 0 & 0 & 0 \\
		0 & 0 & 0 & a
	\end{pmatrix}\text{ if $j$ is even},\]
	\[M_{\e2}'(j,a)=\begin{pmatrix}
		a & 0 \\
		0 & 0 \\
		0 & a \\
		0 & 0
	\end{pmatrix}\text{ if $j$ is odd},
	\begin{pmatrix}
		0 & 0 \\
		a & 0 \\
		0 & 0 \\
		0 & a
	\end{pmatrix}\text{ if $j$ is even}.\]
\end{definition}

Now we are ready to state here the classical Weierstrass-Williamson classification; first we state the diagonalizable case, and then the general case. For both statements we provide new proofs in \cite{CrePel-matrix}.

\begin{maintheorem}[Real Weierstrass-Williamson classification, diagonalizable case]\label{thm:williamson-real}
	Let $n$ be a positive integer. Let $\Omega_0$ be the matrix of the standard symplectic form in $\R^{2n}$. Let $M\in\M_{2n}(\R)$ be a symmetric and invertible matrix such that $\Omega_0^{-1}M$ is diagonalizable. Then, there exists a symplectic matrix $S\in\M_{2n}(\R)$ such that $S^TMS$ is a block-diagonal matrix with blocks of hyperbolic, elliptic type or focus-focus type with $k=1$.
\end{maintheorem}

Theorem \ref{thm:williamson-real} generalizes to the following statement:

\begin{maintheorem}[Real Weierstrass-Williamson classification, general case]\label{thm:williamson-real2}
	Let $n$ be a positive integer and let $M\in\M_{2n}(\R)$ be a symmetric matrix. Then, there exists a symplectic matrix $S\in\M_{2n}(\R)$ such that $S^TMS$ is a block diagonal matrix with each of the diagonal blocks being of hyperbolic, elliptic or focus-focus type, as in Definition \ref{def:blocks}.
	
	Furthermore, if there are two matrices $S$ and $S'$ such that $N=S^TMS$ and $N'=S'^TMS'$ are normal forms, then $N=N'$ except by the order of the blocks.
\end{maintheorem}

\end{document}